\documentclass[reqno]{amsart}
\usepackage[scale=0.75, centering, headheight=14pt]{geometry}
\usepackage[latin1]{inputenc}
\usepackage[T1]{fontenc}
\usepackage{lmodern}
\usepackage[english]{babel}

\usepackage{amsmath,amssymb,amsfonts,amsthm}
\usepackage{mathtools,accents}
\usepackage{mathrsfs}
\usepackage{xfrac}
\usepackage{array} 
\usepackage{aliascnt}
\usepackage{booktabs} 
\usepackage{array} 

\usepackage{verbatim} 
\usepackage{subfig} 

\usepackage{mathrsfs, dsfont}
\usepackage{amssymb}
\usepackage{amsthm}
\usepackage{amsmath,amsfonts,amssymb,esint}
\usepackage{graphics,color}
\usepackage{enumerate}
\usepackage{mathtools,centernot}

\usepackage{microtype}
\usepackage{paralist} 
\usepackage{cases}
\usepackage[initials]{amsrefs}
\allowdisplaybreaks

\usepackage{braket}
\usepackage{bm}

\usepackage[citecolor=blue,colorlinks]{hyperref}
\addto\extrasenglish{}

\usepackage{enumerate}
\usepackage{xcolor}
\usepackage{aliascnt}

\makeatletter
\def\newaliasedtheorem#1[#2]#3{
  \newaliascnt{#1@alt}{#2}
  \newtheorem{#1}[#1@alt]{#3}
  \expandafter\newcommand\csname #1@altname\endcsname{#3}
}
\makeatother

\usepackage{indentfirst}

\usepackage{esint}
\usepackage{mathrsfs}
\usepackage{stmaryrd}

\DeclareUnicodeCharacter{00A0}{ }
\DeclareUnicodeCharacter{00A0}{~}

\makeatletter
\newsavebox{\measure@tikzpicture}

\newcommand{\weak}{\overset{*}{\rightharpoonup}}

\newcommand{\Haus}{\mathcal{H}}

\DeclareMathOperator{\id}{id}

\newcommand{\R}{\mathbb{R}}
\newcommand{\eps}{\varepsilon}
\newcommand{\N}{\mathbb{N}}

\newcommand{\A}{\mathcal{A}}
\newcommand{\G}{\mathbb{G}(n,m)}

\newcommand{\SSS}{\mathbb{S}}

\DeclareMathOperator{\dist}{dist}

\DeclareMathOperator{\dv}{div}
\DeclareMathOperator{\Epi}{Epi}
\DeclareMathOperator{\tr}{tr}

\DeclareMathOperator{\spt}{spt}

\DeclareMathOperator*{\intt}{int}

\newcommand{\E}{\mathcal{E}}
\DeclareMathOperator{\loc}{loc}

\DeclareMathOperator{\co}{co}

\DeclareMathOperator{\Lip}{Lip}
\DeclareMathOperator{\im}{Im}

\theoremstyle{plain}

\newtheorem*{NTEO}{Theorem}
\newtheorem{Teo}{Theorem}[section]
\newtheorem{lemma}[Teo]{Lemma}
\newtheorem*{conjecture}{Conjecture}
\newtheorem{prop}[Teo]{Proposition}

\newtheorem{Cor}[Teo]{Corollary}
\newtheorem{Def}[Teo]{Definition}
\newtheorem{remark}[Teo]{Remark}

\fontsize {11pt} {19pt}
\selectfont

\numberwithin{equation}{section}
\title{The double and triple bubble problem for stationary varifolds: the convex case}

\author[A. De Rosa and  R. Tione]{Antonio De Rosa\and Riccardo Tione}

\address{Antonio De Rosa
\hfill\break Department of Decision Sciences and BIDSA, Bocconi University, Milano, Italy}
\email{antonio.derosa@unibocconi.it}

\address{Riccardo Tione
	\hfill\break Dipartimento di Matematica "Giuseppe Peano", Università degli Studi di Torino, Via Carlo Alberto 10, 10123 Torino, Italy}
\email{riccardo.tione@unito.it}

\begin{document}

\maketitle

\begin{abstract}
We characterize the critical points of the double bubble problem  in $\R^n$ and the triple bubble problem in $\R^3$, in the case the bubbles are convex. 
\end{abstract}

\par
\medskip\noindent
{\sc MSC (2020):  35D30, 49Q05, 49Q20, 52A20, 53A10.
\par
}

\par
\medskip\noindent

\section{Introduction}
The $k$-bubble problem consists in separating $k$ given volumes with the least perimeter. An important literature is dedicated to this problem, which is very well described in the beautiful book of Morgan \cite{MM2}. In particular, in the special case $k=1$, the $k$-bubble problem simplifies to the isoperimetric problem. For $k=2$, Plateau \cite{PP} empirically observed that the optimal configuration is the standard double bubble, that is two spherical caps separated by a spherical cap or a flat disk, meeting at angles of 120 degrees. This observation has been stated as a conjecture by Foisy-Morgan-Sullivan \cite{FF1,MM1,SM}. The $2$-bubble conjecture was solved by Foisy-Alfaro-Brock-Hodges-Zimba in $\R^2$ \cite{FF2}, see also the important contributions of Morgan-Wichiramala \cite{MW}, of Dorff-Lawlor-Sampson-Wilson \cite{DLSW} and of Cicalese-Leonardi-Maggi \cite{CLM}. In $\R^3$, the $2$-bubble conjecture for bubbles of equal volume was settled by Hass-Hutchings-Schlafly in \cite{HHHS} and by Hass-Schlafly \cite{HH}, and for general volumes it was solved by Hutchings-Morgan-Ritor\'e-Ros \cite{HMRR1,HMRR}, see also the structural analysis of $2$-bubbles by Hutchings \cite{HH}. The $2$-bubble conjecture was then solved by Heilmann-Lai-Reichardt-Spielman in $\R^4$ \cite{HLRS} and by Reichardt in $\R^n$ \cite{Re}, see also the alternative proof of Lawlor \cite{L} for arbitrary interface weights. 
Morgan-Sullivan~\cite{SM} have conjectured the optimal shape for every $k$, referring to it as the standard $k$-bubble:
\begin{conjecture}[\cite{SM}]\label{kstandardb}
The standard $k$-bubble in $\R^{n}$ $(1\le k\le n+1)$ is the unique minimizer enclosing $k$ regions of prescribed volume.
\end{conjecture}
\noindent
In $\R^2$ the $3$-bubble conjecture was proved by Wichiramala \cite{WW}, and the $4$-bubble conjecture for bubbles of equal volume was proved by Paolini-Tortorelli-Tamagnini \cite{PPAO1,PPAO2}. Recently, Milman-Neeman proved the above conjecture for all $k\leq \min\{4,n\}$ both in $\R^n$ and in $\mathbb S^n$ in the groundbreaking result \cite{MNS}, that is, they reproved the $2$-bubble conjecture for $n\geq 2$ and they solved the $3$-bubble conjecture for $n\geq 3$ and the $4$-bubble conjecture for $n\geq 4$. In a previous work \cite{MNE}, they also solved this problem in the Gaussian setting.
\\
\\
For $k=1$, finite unions of minimizers can be also characterized as the only critical points for the isoperimetric problem. This is the celebrated Alexandrov's theorem~\cite{MR0102114} for smooth bubbles, which has been generalized to finite perimeter sets by Delgadino-Maggi \cite{DEM}. Afterwards, an alternative proof has been provided for anisotropic perimeters \cite{DRKS}. Moreover, quantitative stability versions of these rigidity theorems have been showed in~\cite{DMMN,DG,DG2}.  However, to the best of our knowledge, the characterization of the critical points of the $k$-bubble problem is completely open for $k \geq 2$. 
\\
\\
In this paper we start investigating critical points for $k\geq 2$. More precisely, we characterize the critical configurations of the $2$-bubble problem in $\R^n$ and the $3$-bubble problem in $\R^3$, in the case the bubbles are convex. These problems will be referred to as the \emph{double bubble} and the \emph{triple bubble} problem, respectively, throughout the paper. To achieve this, we combine a variety of tools, among which structural analysis of stationary varifolds, the moving plane method, convex analysis, conformal geometry and the regularity theory for free boundary problems. Our results can be informally summarized as follows.
\begin{NTEO}
For the $k$-bubble problem under the convexity assumption, the only stationary configurations are:
\begin{itemize}
\item if $k=2$:
\begin{itemize}
\item two disjoint (possibly tangent) balls;
\item the standard double bubble;
\end{itemize}
\item if $k=3$: 
\begin{itemize}
\item three disjoint (possibly tangent) balls;
\item a ball and a standard double bubble, disjoint but possibly tangent;
\item a chain triple bubble, as in Definition \ref{LU};
\item the standard triple bubble, as in Definition \ref{STTB}.
\end{itemize}
\end{itemize}
\end{NTEO}

With a more articulated combinatorics and with a similar machinery to the one developed in this paper, one could investigate the critical points for $k\geq 4$. 
This would pave the way to better understand the conjecture above for the missing case $k=4$ in $\R^3$ and $k\geq 5$ in $\R^n$ with $1\le k\le n+1$, in case the bubbles have equal volumes. Clearly a significant amount of work in this direction is still needed, as a crucial step would be to prove that a minimizer for the $k$-bubble problem, with $1\le k\le n+1$ and with bubbles of equal volumes, consists of convex bubbles. 
\\
\\
We provide here a brief outline of the proof of the main theorem. The first step is to use fine properties in convex analysis to locally parametrize in a suitable way the bubbles on the stationary tangent varifolds at every point. Subsequently, combinatorial arguments together with the stationarity of the cluster, provide a finite amount of configurations of bubbles separated by hyperplanes. One of the most difficult parts of the paper is to uniquely characterize, or at least extract enough information on, the shape of the single bubbles in these configurations. To this aim, we first prove the regularity of the bubbles up to the free boundary of the contact regions with the separating hyperplanes. In particular, we deduce the well-known fact that the bubbles are constant mean curvature surfaces, meeting at almost every point the separating hyperplanes with a constant angle of 120 degrees. Thus we can forget about the whole $k$-bubble and focus separately on the single bubbles forming the cluster. Indeed, a crucial part of  the proof consists in studying capillary surfaces in various geometric configurations. Due to the importance of the theory of capillarity in our proof, let us briefly review the main techniques and results concerning them, to put our strategy into context. Given the vastness of the topic, we cannot provide an exhaustive literature review on capillary surfaces and we refer the reader to \cite{TRIF1,PT,RLC,MCC,PRT,CFR,JWXZ} and references therein for a more complete state of the art.

A capillary hypersurface $S$ in a container $C$ is a constant mean curvature hypersurface whose boundary touches the boundary $\partial C$ forming a constant angle $\gamma$. Due to their importance in geometry and physics, these objects have a long history, see the introduction to \cite{FINB}, and have been considered under various assumptions on $S$, $C$ and $\gamma$. Various questions can be formulated about these objects, but in our case the most interesting one is whether, for specific $C$ and $\gamma$, $S$ can be characterized. In most instances, the actual question is whether $S$ has to be a piece of a sphere. In our case, $S$ is the boundary of a convex set, $C$ is either a half-space, or a strip or a wedge, i.e. a convex sector bounded by two hyperplanes meeting at a line, and $\gamma$ will always be equal to 120 degrees. The case in which $C$ is a half-space or a strip can be handled by the Alexandrov moving plane method, which requires the regularity of $S$ up to the boundary. At present, other methods are available to study in much more detail at least the case in which $C$ is a half-space, see for instance the recent papers \cite{WWX1,DELWE}. The case in which $C$ is a wedge gives rise to many more complications, and it is the reason why we need to restrict ourselves to $\R^3$ in the case $k = 3$. In particular, restricting to surfaces of dimension two allows the use of powerful tools from conformal geometry. Let $\pi_1,\pi_2$ be the two half-hyperplanes defining $\partial C$ and meeting at a line $L$. We assume $S$ intersects both $\pi_1$ and $\pi_2$ in sets of dimension two. If the surface $S$ does not touch $L$, then we can employ the spherical reflection method developed in \cite{MCC,PRT} to infer that $S$ is part of a sphere. If $S$ touches $L$ in a segment, then we need to adapt \cite{CFR}, which in turn is based on a careful study of the so-called Hopf differential of $S$. In the case $S$ touches $L$ in a single point, we originally could not show a priori that a convex capillary surface with those properties is necessarily part of a sphere, but we were able to infer enough information on the single surface $S$ to completely characterize the $3$-bubble to which it belongs. However, after the completion of this paper, we learnt from the anonymous referee (to whom we are grateful) that we could study this case by means of the Heintze-Karcher inequality of the very recent work \cite{JWXZ} to conclude that $S$ is part of a sphere.
\\
\\
We conclude the introduction by outlining the structure of the paper. In Section \ref{s:not} we give the notation used throughout the paper. In Section \ref{s:pre} we prove important consequences of the convexity assumption on the bubbles and of the stationarity of the clusters. In Section \ref{main_double}-\ref{sec_3} we prove the main theorem respectively for $k=2$ and $k=3$. To conclude, in Appendix \ref{app:stat} we recall the computation of the first variation for the $k$-bubble problem.

\subsection*{Acknowledgments}
	 A. De Rosa has been partially supported by the NSF DMS Grant No.~1906451, the NSF DMS Grant No.~2112311, the NSF DMS CAREER Award No.~2143124, and the European Union: the European Research Council (ERC), through StG ``ANGEVA'', project number: 101076411. Views and opinions expressed are however those of the authors only and do not necessarily reflect those of the European Union or the European Research Council. Neither the European Union nor the granting authority can be held responsible for them.
	 
	 Both authors are grateful to Bozhidar Velichkov and Xavier Fern\'andez-Real for useful discussion about free boundary regularity and to Jonas Hirsch for his insights about conformal geometry. This research was conducted  while R. Tione was employed at the Max Planck Institute for Mathematics in the Sciences in Leipzig.


\section{Notation}\label{s:not}

\subsection{Basic notation}
Given $A,B\in \R^{n\times n}$ and  $v,w\in \R^n$, we denote the inner products as $\langle A,B\rangle :=\sum_{ij=1}^n A_{ij}B_{ij}$ and $(v,w):=\sum_{i=1}^n v_{i}w_{i}$. $|A|$ and $|v|$ will be the norms induced by the previous inner products. $A^t$ will denote the transpose of $A$.
\\
\\
For a set $E \subset \R^n$, we write $\overline{E}$, $\intt(E)$ and $\partial E$ for the closure, the interior and the boundary of the set, respectively. For a set $E$, the symbol $sE$ for $s > 0$ denotes the dilation of $E$ by $s$. Eventually, for every $m$-dimensional plane $\pi$ of $\R^n$, we will denote by $p_\pi$ the orthogonal projection onto $\pi$. Given any set $A\subset \R^n$, we denote with $B_\eps(A):=\{x\in \R^n: \dist(x,A)<\eps\}$ the $\eps$-tubular neighborhood of $A$. For two distinct points $x_1,x_2 \in \R^n$, we write $[x_1,x_2]$ for the closed segment connecting $x_1$ and $x_2$ and $(x_1,x_2)$ for the open one. For every subset $X\subset \R^d$, we define the convex hull of $X$ as $\co{(X)}$. The open ball in $\R^n$ of radius $r$ and center $x \in \R^n$ will be denoted with $B_r(x)$, while we will denote with $B^\pi_r(z):=B_r(z) \cap \pi$ the $m$-dimensional ball contained inside the $m$-dimensional plane $\pi$ of radius $r$ centered in $z \in \pi$. Given a plane $\pi$ and $\alpha, \beta > 0$, we define the cylindrical neighborhoods of the point $x \in \pi \subset \R^n$ as follows:
\[
B^\pi_{\alpha,\beta}(x) := \{y \in \R^n: p_\pi(y) \in B_\alpha^\pi(x) \text{ and } \dist(y,\pi) < \beta\}.
\]
If $\varphi : \pi \to \R$ is a function and $\pi$ is a hyperplane with normal $\nu$, we will denote the epigraph of $\varphi$ with
\[
\Epi_\pi(\varphi) :=\{y + a\nu \in \R^n:a > \varphi(y), y \in \pi\}.
\]
\subsection{Measures  and rectifiable sets}\label{measures}
Given a locally compact separable metric space $Y$, we denote by \(\mathcal M(Y)\) the set of positive Radon measure on \(Y\). Given a Radon measure \(\mu\) we denote by \(\spt (\mu)\) its support. For a  Borel set \(E\),  \(\mu\llcorner E\) is the  restriction of \(\mu\) to \(E\), i.e. the measure defined by \([\mu\llcorner E](A)=\mu(E\cap A)\) for all Borel set $A$. If \(\eta :\R^n\to \R^n\) is a proper  (i.e. $\eta^{-1}(K)$ is compact for every $K\subset \R^n$ compact), Borel map and \(\mu\) is a Radon measure, we let \(\eta_\# \mu=\mu\circ \eta^{-1}\) be the push-forward of \(\mu\) through \(\eta\). 
Let $\mathcal{H}^m$ be the $m$-dimensional Hausdorff measure. A set \(K\subset \R^n\) is said to be \(m\)-rectifiable if it can be covered, up to an \(\mathcal{H}^m\)-negligible set, by countably many \(C^1\) $m$-dimensional submanifolds. Given a $m$-rectifiable set $K$, we denote with $T_xK$ the approximate tangent space of $K$ at $x$, which exists for $\mathcal{H}^m$-almost every point $x \in K$, \cite[Chapter 3]{SIM}.  

\subsection{Varifolds}\label{introvar}
We will use $\mathbb{G}(n,m)$ to denote the Grassmanian of (un-oriented) $m$-dimensional linear subspaces in $\R^{n}$, often referred to as $m$-planes. Moreover, we identify the spaces
\begin{equation}\label{iden}
\G = \{P \in \R^{n\times n}: P = P^t, \, P^2 = P, \, \tr(P) = m\}.
\end{equation}
Notice that for $m = n -1$, every element $P \in \G$ can be written as $P = \id_n - \nu\otimes \nu$, for a unit vector $\nu \in \R^n$.
\\
\\
An $m$-dimensional varifold $V$ in $\R^n$ is a Radon measure on $\R^{n}\times \G$. The varifold $V$ is said to be rectifiable if there exists an $m$-rectifiable set $\Gamma$ and an $\mathcal{H}^m\llcorner \Gamma$-measurable function $\theta: \Gamma \to (0, \infty)$ such that
\[
V(f) = \int_{\Gamma}f(x,T_x\Gamma)\theta(x)d\mathcal{H}^m(x), \qquad \forall f \in C_c(\R^n\times \G).
\]  
In this case, we denote $V = (\Gamma,\theta)$. If moreover $\theta$ takes values in $\mathbb N$, $V$ is said integer rectifiable. If $\theta = 1$ $\mathcal{H}^m\llcorner \Gamma$-a.e., then we will write $V = \llbracket \Gamma\rrbracket$. We will use $\|V\|$ to denote the projection in $\R^n$ of the measure $V$, i.e.
\[
\|V\|(A) := V(A\times \G), \qquad \forall A \subseteq \R^n, A \text{ Borel}.
\]
Hence  $\|V\|=p_\# V$, where \(p: \R^N\times \G\to \R^n\) is the projection onto the first factor and the push-forward  is intended in the sense of Radon measures. With a slight abuse of notation, we will often write $V\llcorner A$ rather than $V\llcorner (A\times \G)$. Given an $m$-rectifiable varifold $V = (\Sigma,\theta)$ and $\psi : \Sigma \to \R^n$ Lipschitz and proper, the image varifold of $V$ under $\psi$ is defined by
$$\psi^\#V := (\psi(\Sigma), \tilde \theta), \quad \mbox{where} \quad  \tilde \theta (y) := \sum_{x\in \Sigma \cap \psi^{-1}(y)}\theta(x).$$
Since $\psi$ is proper, we have that $\tilde \theta \mathcal{H}^m\llcorner \psi(\Sigma)$ is a Radon measure. By the area formula we get
$$\psi^\#V(f)=\int_{\psi(\Sigma)}f(x,T_x\psi(\Sigma))\tilde \theta(x)d\mathcal{H}^m(x)=\int_{\Sigma}f(\psi(x),d_x\psi(T_x\Sigma))J\psi(x,T_x\Sigma)\theta(x) d\mathcal{H}^m(x),$$
for every $f\in C^0_c(\R^N\times \G)$. Here $d_x\psi(S)$ is the image of $S$ under the linear map $d_x\psi(x)$ and 
\[
J\psi(x,S):=\sqrt{\det\Big(\big(d_x\psi\big|_S\big)^t\circ d_x\psi\big|_S\Big)}
 \]
 denotes the $m$-Jacobian determinant of the differential $d_x\psi$ restricted to the $m$-plane $S$, see \cite[Chapter 8]{SIM}. The area of a rectifiable varifold $V=(\Gamma,\theta)$ is 
 \[
\mathcal{A}(V) := \|V\|(\R^n) = \int_{\Gamma}\theta(z)d\mathcal{H}^m(z). 
 \]
We define the first variation of $V = (\Gamma,\theta)$ as 
\[
[\delta V](g) := \left.\frac{d}{d\eps}\right|_{\eps = 0}\|((\Phi_\eps)^\#(V))\|, \quad \forall g \in C^1_c(\R^n,\R^n),
\]
where $\Phi_\eps$ is the flow associated\footnote{Namely $\Phi_\varepsilon(x) = \gamma_x(\varepsilon)$, where $\gamma_x$ is the solution of the ODE $\gamma'(t) = g(\gamma(t))$ subject to the initial condition $\gamma(0) = x$.} to $g$. By \cite[Lemma A.2]{DDG} we have:
\[
[\delta V](g) = \int_{\R^n}\langle T_x\Gamma,Dg(x)\rangle d\|V\|.
\]
We say that a rectifiable varifold $V = (\Gamma,\theta)$ has \emph{bounded mean curvature in $U \subset \R^n$} if there exists a map $H \in L^\infty(\Gamma \cap U,\R^n; \mathcal{H}^m\llcorner \Gamma)$ such that
\begin{equation}\label{boundedmean}
[\delta V](g) = -\int_{\R^N}(H(x),g(x))d\|V\|(x), \quad \forall g \in C^1_c(U,\R^n).
\end{equation}
If $H \equiv 0$ we say that $V$ is \emph{stationary} and if $H$ is constant we say that $V$ \emph{has constant mean curvature}. A blow-up of a rectifiable varifold $V$ at $x \in \R^n$ is any weak-star accumulation point of the family of measures
\[
V_{x,r} := \eta^{\#}_{x,r}(V),
\]
where $\eta^{x,r}(y) := \frac{y - x}{r}$ is the dilation map centered at $x$.
If $V$ has bounded mean curvature, as a consequence of the monotonicity formula there exists at least one limit point of the family $\{V_{x,r}\}_r$, and every limit point is a cone, see \cite[Theorem 7.3]{SIM}. 
One of the possible stationary limits of this procedure is the varifold obtained by the sum (in the varifold sense) of three $m$-dimensional planes intersecting at their common boundary that is an $(m-1)$-dimensional plane, and forming pairwise angles of $120$ degrees. We name this configuration a \emph{Y cone}. The graph $\Gamma_u$ of a Lipschitz function $u : \Omega\subset \R^m \to \R^{n - m}$ defines an $m$-dimensional varifold with multiplicity 1. Without loss of generality, we can suppose that the graph is parametrized on the first $m$ coordinates, so that $\Gamma_u := \{(x,y) \in \R^{n}: y = u(x)\}$. We will also use the notation $\Gamma_{u,A} := \{(x,u(x)): x \in A\}$, for $A \subset \R^m$. If $u$ is defined on an $m$-dimensional plane $\pi$ in $\R^n$, then we will write $\Gamma_u^\pi$ to denote the graph of $u$ in $\R^n$.
For notational purposes, we define the following maps:
\begin{equation}\label{maps}
M(X):=
\left(
\begin{array}{c}
id_m\\
X
\end{array}
\right), \quad \mbox{and} \quad \A(X) := \sqrt{\det(M(X)^tM(X))}, \quad \forall X \in \R^{(n-m)\times m},
\end{equation}
where $\A(X) $ simply corresponds to the area element of $X$, and
\begin{equation}\label{h}
h: \R^{(n-m)\times m} \to \R^{n\times n}, \quad h(X) := M(X)(M(X)^tM(X))^{-1}M(X)^t.
\end{equation}
Recalling \eqref{iden}, it is easily seen that $h(\R^{(n-m)\times m}) \subseteq \G$, and that $h$ is injective.

\subsection{Finite perimeter and convex sets}

For a set $E \subset \R^n$ of finite perimeter, we denote with $\partial^*E$ its reduced boundary, see \cite[Definition 5.4]{EVG}, and by $n_E$ its exterior measure theoretic normal, so that
\[
\int_{E}\dv\Phi(x) dx = \int_{\partial^*E}(n_E(x),\Phi(x))d\Haus^{n-1}(x),\quad \forall \Phi \in C^\infty_c(\R^n,\R^n).
\]
We will mostly consider convex sets. When we write that $C$ is convex, we implicitly mean that it is an open and bounded set, and otherwise it will always be made explicit. Recall that for a convex set there is a well defined notion of dimension: the dimension of the convex set $C \subset \R^n$ is the largest number $d$ such that $C$ contains $d + 1$ points $x_0,\dots, x_d$ with $\{x_1 - x_0,\dots, x_d - x_0\}$ linearly independent. Given any plane $\pi \subset \R^n$, we have that $\pi \cap C$ is a convex set, and we will write $\partial_\pi(\pi \cap C)$ to denote the boundary of $\pi \cap C$ with respect to the induced topology on $\pi$. We will say that two convex sets $C_1$, $C_2$ are separated if there exists an $(n-1)$-dimensional plane $\pi$ such that $C_1$ lies in one of the connected component of $\R^n \setminus \pi$ and $C_2$ lies in the other one.

\section{The $k$-bubble problem, stationarity and convexity}\label{s:pre}

In this paper we consider the $k$-bubble problem. For an introduction to the subject we refer the reader to \cite[Part IV]{MAG} and \cite[Chapter 13]{MM2}. Consider $k$ disjoint open sets of locally finite perimeter $E_1,E_2,\dots, E_k \subset \R^n$, and let $\E = \{E_1,E_2,\dots,E_k\}$. Define the $(n-1)$-dimensional integer rectifiable varifold in $\R^n$
\begin{equation}\label{V}
V_\E := \frac{1}{2}\sum_{i = 1}^k\partial^*E_i + \frac{1}{2}\partial^*\left[\left(\bigcup_{i = 1}^k\overline{E_i}\right)^c\right].
\end{equation}
We denote $V_{\E} = (\Gamma_{\E},\theta)$. We say that $V_{\E}$ is \emph{stationary for the $k$-bubble problem} if 
\begin{equation}\label{doubprob1}
\left.\frac{d}{dt}\right|_{t= 0}\A((\Phi_t)^\#(V_{\E})) = 0,  \, \mbox{$\forall \Phi_t$ flow of a field $g \in C^\infty_c(\R^n,\R^n)$ satisfying
$|\Phi_t(E_i)| = |E_i| \,\, \forall t \in \R.$}
\end{equation}
Assuming, as we will do in the rest of the paper, that $E_i$ is convex $\forall 1\le i \le k$, we can show the following
\begin{prop}\label{equiva}
If $E_i$ is convex for all $1\le i \le k$, $V_\E$ is stationary for the $k$-bubble problem if and only if there exist $\lambda_i\in \R$ such that
\begin{equation}\label{doubprob}
\int_{\R^n}\langle T_x\Gamma_{\E}, Dg\rangle d\|V_\E\| = \sum_{i = 1}^k\lambda_i\int_{\partial^*E_i}(n_{E_i},g)\theta d\mathcal{H}^{n-1}, \qquad \forall g \in C^\infty_c(\R^n,\R^n).
\end{equation}
\end{prop}
We will sketch the proof of Proposition \ref{equiva} following closely the lines of \cite[Appendix B-C]{MNE} in Appendix \ref{app:stat}.

\subsection{Geometric consequences of convexity}
\begin{lemma}\label{lele}
Let $E_i$ be convex for all $1\le i \le k$. Then, the multiplicity $\theta$ of $V_\E = (\Gamma_\E,\theta)$ is equal to $1$ for $\mathcal{H}^{n -1}$-a.e. $x \in \Gamma_\E$.
\end{lemma}
\begin{proof}
Let $1 \le i < j < \ell \le k$ be arbitrary. Since $\partial^* E \subset \partial E$, the proof follows from the following simple observation: 
\begin{equation}\label{intbound}
\mathcal{H}^{n -1}(\partial E_i\cap \partial E_j\cap \partial E_\ell) = 0, \qquad \mbox{and} \qquad \mathcal{H}^{n -1}\left(\partial E_i\cap \partial E_j\cap  \partial\left[\left(\bigcup_{r = 1}^k\overline{E_r}\right)^c\right]\right) = 0.
\end{equation}

\end{proof}

Proposition \ref{equiva} and Lemma \ref{lele} immediately imply the following crucial results, which will be the only information we will use in the rest of the paper concerning stationary $k$-bubbles.
\begin{Cor}\label{corstat}
Let $E_i$ be convex for all $1\le i \le k$ and let $V_{\E}$ be stationary for the $k$-bubble problem. Then, the mean curvature of $V_\E$ is bounded. Moreover, the varifolds $\llbracket\partial E_i\setminus \bigcup_{j= 1, j \neq i}^k \partial E_j\rrbracket$ are of constant mean curvature $\lambda_i$ for all $i$. Finally if $\mathcal{H}^{n-1}(\partial E_i \cap \partial E_j) \neq 0$, then $\lambda_i=\lambda_j$.
\end{Cor}
\begin{proof}
The fact that $V_\E$ has bounded mean curvature and that $\llbracket\partial E_i\setminus \bigcup_{j= 1, j \neq i}^k \partial E_j\rrbracket$ are of constant mean curvature $\lambda_i$ for all $i$ easily follows from \eqref{boundedmean} and \eqref{doubprob}. To show the last statement, assume $\mathcal{H}^{n-1}(\partial E_i \cap \partial E_j) \neq 0$. As $E_i$ and $E_j$ are convex, $\partial E_i \cap \partial E_j$ is an $(n-1)$-dimensional convex subset of an $(n-1)$-dimensional plane $\pi$. Take a point $x_0 \in \intt_{\pi}(\partial E_i \cap \partial E_j)$. Then, convexity easily yields the existence of a small $\delta > 0$ such that
\begin{equation}\label{onlyij}
B_\delta(x_0) \cap \spt(\|V_{\E}\|) = \partial^*E_i \cap \partial^*E_j \cap B_\delta(x_0) \subset \pi.
\end{equation}
Taking test vector fields $g$ with $\spt(g) \subset B_\delta(x_0)$, by \eqref{intbound} and the divergence theorem we see that \eqref{doubprob} reads
\[
0 = \lambda_i\int_{\partial^*E_i}(n_{E_i},g) d\mathcal{H}^{n-1} + \lambda_j\int_{\partial^*E_j}(n_{E_j},g) d\mathcal{H}^{n-1} = (\lambda_i - \lambda_j)\int_{\pi}(n_{E_i},g) d\mathcal{H}^{n-1},
\]
the last equality being true since $n_{E_j} = -n_{E_i}$ on $\pi$. The arbitrarity of $g$ yields $\lambda_i = \lambda_j$.
\end{proof}
Moreover, Corollary \ref{corstat} together with \cite[Chapter 4]{SIM} implies the following:
\begin{Cor}\label{blowcone}
Let $V_{\E}$ be stationary for the $k$-bubble problem. Then, at all $x \in \spt(V_\E)$, there exists at least one blow-up $W$ of $V_\E$. $W$ is a cone with vertex at $x$, i.e. $(\eta_{x,r})^{\#}W = W, \forall r > 0$, and $W$ is stationary.
\end{Cor}

It is well known, see for instance \cite[Corollary 1.2.2.3]{GRIS}, that
\begin{prop}\label{graph}
Let $C\subset \R^d$ be a bounded convex set of dimension $d$. Then, $C$ is a Lipschitz domain, and in particular for all points $x \in \partial C$, there exist $\delta = \delta(x) > 0$, a supporting hyperplane $\pi$ to $\partial C$ with normal $\nu$ and a convex function $\varphi : B^{\pi}_{\delta}(x) \to \R$ such that:
\[
\partial C \cap B_{\delta, \Lip(\varphi)\delta}^\pi(x) = \{y + \varphi(y)\nu: y \in B_\delta^\pi(x)\} 
\]
and
\[
 C \cap  B_{\delta,\Lip(\varphi)\delta}^\pi(x) = \{y + a\nu \in B_{\delta,\Lip(\varphi)\delta}^\pi(x): a > \varphi(y), y \in B_\delta^\pi(x) \}.
\]
\end{prop}

\begin{prop}\label{uniquecone}
Let $C\subset \R^d$ be a bounded convex set of dimension $d$, and let $V =\llbracket \partial C\rrbracket$. Then, the blow-up $\llbracket K \rrbracket$ at every $x \in \partial C$ of $V$ is unique, and it is a cone that bounds a convex set. In particular, with the notation of Proposition \ref{graph} at $x$, then the blow-up of $\llbracket K \rrbracket$ at $x$ is the graph over $\pi':=\pi-x$ of the convex and positively one-homogeneous function 
\begin{equation}\label{homogeneous}
H(w) := \lim_{r \to 0^+}\frac{\varphi(x +rw) - \varphi(x)}{r}, \quad \forall w \in \pi'.
\end{equation}
The convergence in the previous limit is uniform and in $W^{1,p}$ for all $p <+ \infty$ on every fixed ball $B^{\pi'}_s \subset \pi'$. Finally, $C$ lies in the open, convex set $\Epi_{\pi'}(H)+x$.
\end{prop}
\begin{proof}
The limit in \eqref{homogeneous} is well defined because the left and right derivatives of a $1$-variable convex function always exist. We just need to show that the limit is uniform and in $W^{1,p}$ for all $p < + \infty$ on every fixed ball $B^{\pi'}_s \subset \pi$. It is enough to show that the sequence $H_r:=\frac{\varphi(x +r \cdot) - \varphi(x)}{r}$ satisfies the property that $DH_r$ is equibounded in $BV$. Indeed $BV$ compactly embeds in $L^p$ for every $p < d/(d-1)$ and, since $DH_r$ are equibounded in $L^\infty$, we would deduce the strong convergence in $L^p$ for all $p<+\infty$ by standard interpolation. The fact that $DH_r$ is equibounded in $BV$ follows from the following Lemma \ref{lemmainter}, observing that $H_r$ is convex. The strong convergence of the gradients, together with \cite[Theorem 5.2]{AR}, implies that
$
\llbracket K_{x,r}\rrbracket \weak \llbracket \Gamma_{H,\pi'}\rrbracket
$
in the sense of varifolds.
\end{proof}

\begin{lemma}\label{lemmainter}
There exists a constant $C(d)>0$ such that, for every convex function $f:B_2\subset \R^d \to \R$, it holds
$$\|D^2f\|(B_1)\leq C\|f\|_{L^\infty(B_2)},$$
where $\|D^2f\|$ denotes the total variation of the measure $D^2f$.
\end{lemma}
\begin{proof}
By standard mollification, it is enough to assume $f \in C^\infty$.
We compute
\begin{equation}\label{primo}
\int_{B_1}\Delta f=\int_{\partial B_1}( Df,n_{B_1} ) d\mathcal{H}^{n-1}\leq \|f\|_{\Lip(B_1)}\Haus^{d-1}(\partial B_1)\leq C\|f\|_{L^\infty(B_2)},
\end{equation}
where the last inequality holds because $f$ is convex.
Since $D^2f$ is valued in the set of positive semidefinite matrices, then the pointwise Frobenius norm $|D^2f|$ of  $D^2f$ is controlled as follows:
$$|D^2f|\leq C(d)\Delta  f,$$
which, combined with \eqref{primo}, concludes the claim.
\end{proof}


\begin{Cor}\label{regularconv}
Let $C\subset \R^d$ be a bounded convex set of dimension $d$, and let $V =\llbracket \partial C\rrbracket$. Then, $\partial C$ is a $C^1$ domain if and only if the blow-up at every $x \in \partial C$ is contained in a $(d-1)$-dimensional linear subspace of $\R^d$. In this case, the function $\varphi$ of Proposition \ref{graph} can be chosen to be $C^1$.
\end{Cor}
\begin{proof}
If $\partial C$ is a $C^1$ domain, then trivially the blow-up at every $x \in \partial C$ is contained in a $(d-1)$-dimensional linear subspace of $\R^d$. For the reverse implication, assume that at every $x$ the blow-up $W_x$ of $\partial C$ is contained in a hyperplane $\pi$. By Proposition \ref{uniquecone}, $W_x$  is a cone such that $W_x=\llbracket \partial U\rrbracket$, where $U \subset \R^d$ is a convex non-empty open set. Assume by contradiction that $\partial U \neq  \pi $, then there exists $y \in \pi \setminus \partial U$. Moreover there exists $z\in U\setminus \pi$. Hence, the segment joining $y$ with $z$ intersects $\partial U$ at least once, contradicting the fact that $\partial U\subset \pi$. Whence $W_x = \llbracket\pi \rrbracket$. The conclusion follows from the fact that a convex function that is differentiable at every point of an open set must be $C^1$ in the open set.
\end{proof}

\begin{Cor}\label{corstruct}
Let $E_1$ and $E_2$ be convex sets of $\R^n$ separated by the hyperplane $\pi$. Assume that $\overline{E_1}\cap \pi = \overline{E_2}\cap \pi$  is an $(n-1)$-dimensional convex set.  Let $L$ be any supporting $(n-2)$-dimensional plane $L \subset \pi$ to $\partial_\pi(\overline{E_1}\cap \pi)$. Let $\pi_1$ and $\pi_2$ be the two half hyperplanes bounded by $L$ inside $\pi$, with $\overline{E_1}\cap \pi \subset \pi_1$. Eventually, let $\E = \{E_1,E_2\}$. Then, the blow-up $W$ of $V_\E$ at $x \in \partial_\pi(\overline{E_1}\cap \pi)$ is unique and can be characterized as follows:
\[
W = \llbracket K\rrbracket + \llbracket K_1\rrbracket + \llbracket K_2\rrbracket,
\]
where
\begin{itemize}
\item  $\llbracket K\rrbracket$, $\llbracket K \cup K_1\rrbracket$ and $\llbracket K\cup K_2\rrbracket$ are the blow-ups respectively of $\llbracket\overline{E_1}\cap \pi\rrbracket$, $\llbracket\partial E_1\rrbracket$ and $\llbracket\partial E_2\rrbracket$. Moreover $K$ is a convex cone of dimension $n - 1$ with vertex at $0$ and entirely contained in $\pi_1$;
\item $K_1$ and $K_2$ are two cones with $K_i\cap K = \partial_\pi K, i=1,2,$ $K_1$ and $K_2$ are entirely contained respectively in each of the two half-spaces bounded by $\pi$. Furthermore,  for $i = 1,2$, $K_i$ is the graph of the restriction of a positively one-homogeneous convex function to a subset $A_i$ of a supporting hyperplane $\pi_i$ of $\overline{E_i}$ at $x$. $A_i$ need not to be convex, but it is a connected cone and $A_i^c$ is a convex cone inside $\pi_i$.
 \end{itemize}
\end{Cor}

\begin{proof}
In a small ball $B_\delta(x)$, we have that
\[
V_\E\llcorner B_\delta(x) = \llbracket \overline{E_1}\cap \pi \rrbracket\llcorner B_\delta(x) + \llbracket \partial{E_1}\setminus \pi\rrbracket\llcorner B_\delta(x) + \llbracket \partial{E_2}\setminus \pi\rrbracket\llcorner B_\delta(x) =: V_1 + V_2 + V_3.
\]
Since clearly a blow-up of a varifold is local and linear, we can study separately the blow-ups of the three addenda of the previous expression. In particular, the study of the blow-up of $V_3$ is entirely analogous to the one of $V_2$, hence we can reduce ourselves to study the blow-ups of the first two addenda. The key point here is that we can express these objects as graphs of convex functions through Proposition \ref{graph}.

In particular, take any supporting hyperplane $\pi_1$ for $\overline{E_1}$ at $x \in \partial_\pi(\overline{E_1}\cap \pi)$ for which we can apply Proposition \ref{graph} at $x$ with the plane $\pi_1$. Hence we find a convex function $\varphi$ as in the statement of Proposition \ref{graph}, defined on $B_\delta^{\pi_1}(x)$. Notice that, possibly taking a smaller ball, $\varphi$ can be chosen to be Lipschitz in $B_\delta^{\pi_1}(x)$. Moreover:
\begin{equation}\label{phizer}
\varphi(y) \ge 0, \quad \forall y \in B_\delta^{\pi_1}(x), \qquad \text{ and } \qquad \varphi(x) = 0.
\end{equation}
We further simplify the study of the blow-ups in the following way. We write
\begin{equation}\label{sum}
\llbracket \partial E_1 \rrbracket\llcorner B_\delta(x) = \llbracket \overline{E_1}\cap \pi \rrbracket\llcorner B_\delta(x) + \llbracket \partial{E_1}\setminus \pi\rrbracket\llcorner B_\delta(x).
\end{equation}
If we show separately that the blow-up of $\llbracket \partial E_1 \rrbracket\llcorner B_\delta(x)$ at $x$ is given by the graph of the positively one-homogeneous convex function $H_1$ defined on $\pi_1$ as in \eqref{homogeneous}, and that the blow-up of $\llbracket \overline{E_1}\cap \pi \rrbracket\llcorner B_\delta(x)$ is given by the graph of $H_1$ on a subset $A$ of $\pi_1$, then it readily follows that also the blow-up of $ \llbracket \partial{E_1}\setminus \pi\rrbracket\llcorner B_\delta(x)$ is unique and is given by the graph of $H_1$ on $\pi_1 \setminus A$.
\\
\\
\fbox{Blow-up of $\llbracket \partial E_1 \rrbracket\llcorner B_\delta(x)$.} Without loss of generality, we can assume $\pi_1 = \{y \in \R^n: y_n = 0\}$. Let $\eps > 0$ be sufficiently small such that
\[
\partial{E_1}\cap B_{\eps,\Lip(\varphi)\eps}^{\pi_1}(x) = \{(y,\varphi(y)): y \in B_\eps^{\pi_1}(x)\}.
\]
The blow-up of $\llbracket\partial E_1\rrbracket$ at $x$ coincides with that of $\llbracket\partial{E_1}\cap B_{\eps,\Lip(\varphi)\eps}^{\pi_1}(x)\rrbracket$.  Recalling the notation of Subsection \ref{introvar}, by the area formula it readily follows that for $r$ sufficiently small
\begin{equation}\label{areaspec}
(\llbracket \partial{E_1}\cap B_{\eps,\Lip(\varphi)\eps}^{\pi_1}(x) \rrbracket)_{x,r} = \llbracket \Gamma_{\varphi_r,B_{\eps /r}^{\pi_1}(0)}\rrbracket,
\qquad \mbox{ 
where } \qquad 
\varphi_r(w) := \frac{\varphi(x + rw) - \varphi(x)}{r}.
\end{equation}
Since $\varphi$ is convex, by Proposition \ref{uniquecone}, we infer that for every fixed $s$,
$
\llbracket \Gamma_{\varphi_r,B_{s}^{\pi_1}(x)}\rrbracket \weak \llbracket \Gamma_{H_1,B_{s}^{\pi_1}(0)}\rrbracket
$
in the sense of varifolds. It then follows that the blow-up of $\llbracket \partial E_1 \rrbracket\llcorner B_\delta(x)$ coincides with $\llbracket\Gamma_{H_1}\rrbracket$.
\\
\\
\fbox{Blow-up of $\llbracket \overline{E_1}\cap \pi \rrbracket\llcorner B_\delta(x)$.} With the same notation of the previous case, the convex set $A \subset B_{\eps}^{\pi_1}(x)$ 
\[
A := p_{\pi_1}(\overline{E_1}\cap \pi \cap B_{\eps,\Lip(\varphi)\eps}^{\pi_1}(x)),
\]
satisfies
\[
\overline{E_1}\cap \pi\cap B_{\eps,\Lip(\varphi)\eps}^{\pi_1}(x) = \{(y,\varphi(y)): y \in A\}.
\]
The blow-up of $\llbracket \overline{E_1}\cap \pi \rrbracket$ at $x$ coincides with that of $\llbracket \overline{E_1}\cap \pi\cap B_{\eps,\Lip(\varphi)\eps}^{\pi_1}(x)\rrbracket$.  As in \eqref{areaspec}, we have
\[
\llbracket \overline{E_1}\cap \pi\cap B_{\eps,\Lip(\varphi)\eps}^{\pi_1}(x) \rrbracket_{x,r} = \left\llbracket \Gamma_{\varphi_r,\frac{A - x}{r}}\right\rrbracket.
\]
Since $A$ is a convex set inside $\pi_1$, we can find a supporting $(n-2)$-dimensional plane $L'$ for $A$ at $x$, on which $\partial_{\pi_1}A\cap B_{\rho}(x)$ can be written as the graph of some convex function $\psi$ defined on $B_\rho^{L'}(x)$. The existence of a sufficiently small $\rho > 0$ with this property is again guaranteed by Proposition \ref{graph}. Moreover, $A \cap B_\rho(x)$ is given by the open epigraph of $\psi$. Since the rescaling family $\psi_r$ of $\psi$, defined analogously to \eqref{areaspec}, are equi-Lipschitz and convex, we deduce that they uniformly converge to a convex one-homogeneous function $G$. This yields the strong local convergence in $L^1$ of the characteristic functions
$
\chi_{\frac{A - x}{r}} \to \chi_{B},
$
where $B \subset \pi_1$ is the epigraph of $G$. Now we will show that
\[
\left\llbracket \Gamma_{\varphi_r,\frac{A - x}{r}}\right\rrbracket \weak \left\llbracket \Gamma_{H_1,B}\right\rrbracket.
\]
Indeed for every $f \in E_c(\R^n \times \mathbb G(n, n-1))$ we compute
\begin{equation}
\begin{split}
\left\llbracket \Gamma_{\varphi_r,\frac{A - x}{r}}\right\rrbracket (f) &= \int_{ \Gamma_{\varphi_r,\frac{A - x}{r}}}f(y, T_y \Gamma_{\varphi_r,\frac{A - x}{r}}) d\Haus^{n-1}(y)=\int_{ \frac{A - x}{r}}f(y',\varphi_r(y'), h(D\varphi_r(y'))) \A (D\varphi_r(y')) dy'\\
&=\int_{\pi_1}(\chi_{ \frac{A - x}{r}}(y')-\chi_{B}(y'))f(y',\varphi_r(y'), h(D\varphi_r(y'))) \A (D\varphi_r(y')) dy' \\
&\quad + \int_{\pi_1}\chi_{B}(y')f(y',\varphi_r(y'), h(D\varphi_r(y'))) (\A (D\varphi_r(y'))- \A (DH_1(y')))dy'\\
&\quad + \int_{\pi_1}\chi_{B}(y')f(y',\varphi_r(y'), h(D\varphi_r(y')))  \A (DH_1(y'))dy'.
\end{split}
\end{equation}
The first term in the last equality converges to zero, as $f(y',\varphi_r(y'), h(D\varphi_r(y'))) \A (D\varphi_r(y'))$ is locally equibounded in $L^2$, $f$ is compactly supported,  and $\chi_{ \frac{A - x}{r}}\to \chi_{B}$ strongly in $L^2_{\textrm{loc}}$. 
The second term converges to zero, as $\chi_{B}(y')f(y',\varphi_r(y'), h(D\varphi_r(y')))$ is compactly supported and equibounded in $L^\infty$ and $\A (D\varphi_r(y'))\to \A (DH_1(y'))$ strongly in $L^1_{\textrm{loc}}$.
Hence, by dominated convergence and reapplying the area formula
$$
\left\llbracket \Gamma_{\varphi_r,\frac{A - x}{r}}\right\rrbracket(f) \to \int_{\pi_1}\chi_{B}(y')f(y',H_1(y'), h(DH_1(y')))  \A (DH_1(y'))dy'=\left\llbracket \Gamma_{H_1,B}\right\rrbracket(f).
$$
This concludes the proof of the Corollary.
\end{proof}

\begin{Cor}\label{corstruct31}
Let $E_1,E_2$ and $E_3$ be convex disjoint sets of $\R^3$, and let as usual $\E = \{E_1,E_2,E_3\}$. Define the planes $\pi_{ij}$ to be the planes separating $E_i$ from $E_j$ for $1\le i \neq j \le 3$. Assume:
\begin{enumerate}[\quad(H1)]
\item $\pi_{12}\cap \pi_{13} = L$, where $L$ is a line;
\item $\overline{E_1}\cap \pi_{12} = \overline{E_2}\cap \pi_{12}$, and these are two-dimensional sets;
\item $\overline{E_1}\cap \pi_{13} = \overline{E_3} \cap \pi_{13}$, and these are two-dimensional sets;
\item $\Haus^{2}(\partial E_2\cap \partial E_3) = 0$;
\item $ \partial E_1\cap L =\partial E_2\cap L = \partial E_3 \cap L\neq \emptyset$.
\end{enumerate}
Let $\Sigma_j$, $j \in \{1,2,3,4\}$, be the four open convex sectors in which $\pi_{12}$ and $\pi_{13}$ subdivide $\R^3$, with
\[
E_1 \subset \Sigma_1, \quad E_2 \subset \Sigma_2 \cup \Sigma_4\quad \text{ and } E_3 \subset \Sigma_3 \cup \Sigma_4.
\]
Fix $x \in  \partial E_1\cap L$. Then, the unique blow-up $W$ of $V_\E$ at $x$ is a $2$-dimensional rectifiable varifold and can be characterized as follows:
\[
W = \llbracket K\rrbracket + \llbracket K_2\rrbracket + \llbracket K_3\rrbracket,
\]
where
\begin{itemize}
\item $\llbracket K\rrbracket$ is the blow-up of $\llbracket \partial E_1\rrbracket$ at $x$ and $K$ is the $2$-dimensional boundary of a convex cone with vertex at $0$ and entirely contained in $\overline{\Sigma_1}-x$;
\item For $i = 2,3$, $K_i$ is a cone with $K_i \cap K = \partial_{\pi_{1i}}(K\cap \pi_{1i})$, and is the blow-up of $\llbracket\partial E_i \setminus (\partial E_j\cup \partial E_k)\rrbracket$, with $\{i,j,k\} =\{1,2,3\}$. Moreover, $K_i$ is the graph of the restriction of a positively one-homogeneous convex function to a subset $A_i$ of a supporting hyperplane $\pi_i$ of $\partial E_i$ at $x$. $A_i$ needs not to be convex, but it is a connected cone and $A_i^c$ is a convex cone inside $\pi_i$. Finally, $K_i \subset \overline{\Sigma_i\cup \Sigma_4}-x$, for $i = 2,3$.
 \end{itemize}
\end{Cor}
\begin{proof}
The proof is analogous to the one of Corollary \ref{corstruct}, once we notice that in a sufficiently small ball $B_\delta(x)$,
\[
V_\E\llcorner B_\delta(x) = \llbracket \partial{E_1}\rrbracket\llcorner B_\delta(x) + \llbracket \partial{E_2}\setminus (\partial E_1\cup \partial E_3)\rrbracket\llcorner B_\delta(x) + \llbracket \partial{E_3}\setminus (\partial E_1\cup \partial E_2)\rrbracket\llcorner B_\delta(x),
\]
and hence that the three blow-ups can be computed separately. In particular, $\llbracket K\rrbracket$ is the blow-up of $ \llbracket \partial{E_1}\rrbracket\llcorner B_\delta(x)$, $\llbracket K_2\rrbracket$ is the blow-up of $\llbracket \partial{E_2}\setminus (\partial E_1\cup \partial E_3)\rrbracket\llcorner B_\delta(x)$ and $\llbracket K_3\rrbracket$ is the blow-up of $\llbracket \partial{E_3}\setminus (\partial E_1\cup \partial E_2)\rrbracket\llcorner B_\delta(x)$.
\end{proof}

\subsection{Classification of surfaces}

 In a few instances, we will need to know \emph{a priori} that the surface under consideration is, roughly speaking, equivalent to a disk in $\R^2$ or to an annulus in $\R^2$. Below we give sufficient conditions to infer this.  We will use the terminology from \cite{FMC,PRT,MCC}.
 \begin{Def}\label{def:DTRT}
 Let $S \subset \R^3$ be a smooth surface with piecewise smooth boundary. Then, $S$ is \emph{of disk-type} (or topologically a disk) if there exists a homeomorphism $r : \overline{B_1(0)} \subset \R^2 \to S$. $S$ is \emph{of ring-type} if it is a compact, connected, orientable surface with two boundary components and Euler-Poincar\'e characteristic zero.
 \end{Def}

\subsubsection{Ring-type surfaces}

\begin{prop}\label{RT}
Let $E \subset \R^3$ be a convex subset. Let $\pi_1$, $\pi_2\subset \R^3$ be two affine planes and $C_1:=\pi_1\cap \partial E$, $C_2:=\pi_2\cap \partial E$ be disjoint and 2-dimensional. Then $\overline{\partial E\setminus C_1\cup C_2}$ is of ring-type according to Definition \ref{def:DTRT}.
\end{prop}
\begin{proof}
Let $\Sigma := \overline{\partial E\setminus C_1\cup C_2}$. $\Sigma$ is clearly compact, it is orientable since $\partial E$ is orientable, and it has two boundary components, that is the disjoint Lipschitz curves $\partial C_1$ and $\partial C_2$, where with a slight abuse of notation $\partial$ is denoting the boundary in the induced topology on ${\partial E}$. The Euler characteristic $\chi(\Sigma)$ is:
\begin{align*}
\chi(\Sigma) &= \chi(\partial E) - \chi(C_1 \cup C_2) + \chi(\Sigma \cap (C_1\cup C_2)) = \chi(\partial E) - \chi(C_1) -  \chi(C_2) + \chi(\partial C_1) + \chi(\partial C_2).
\end{align*}
In this chain of equalities we have used that $C_1\cap C_2 = \emptyset$ and $\Sigma \cap (C_1\cup C_2) = \partial C_1\cup \partial C_2$. Since $\partial E$ is homeomorphic to a sphere, we have $\chi(\partial E) = 2$. Moreover, since $C_1$ and $C_2$ are convex two-dimensional sets, they are homeomorphic to disks, and hence $\chi(C_1) = \chi(C_2) = 1$. Finally, $\partial C_1$ and $\partial C_2$ are simple closed loops homeomorphic to a circle, and hence $\chi(\partial C_1) = \chi(\partial C_2) = 0$. Hence $\chi(\Sigma) = 0$. 

We are just left to prove that $\Sigma$ is connected. To this aim, consider $x,y \in \Sigma$. Since $\partial E$ is path connected, we find $\gamma : [0,1]\to \partial E$ such that $\gamma(0) = x, \gamma(1) =y$. If $\im(\gamma) \cap (C_1\cup C_2) = \emptyset$, then there is nothing to prove. Otherwise, assume without loss of generality that $\im(\gamma) \cap C_1 \neq \emptyset$. In this case let
\[
t_1 = \min\{t \in [0,1]: \dist(\gamma(t),C_1) = 0\}, \quad t_2 = \max\{t \in [0,1]: \dist(\gamma(t),C_1) = 0\}.
\]
Notice that these are well-defined since $\gamma$ is continuous and $C_1$ is closed, and that $\gamma(t_1),\gamma(t_2) \in \partial C_1$. Since $\partial C_1$ is path connected, we may find $\eta: [t_1,t_2] \to \partial C_1$ such that $\eta(t_i) = \gamma(t_i), \forall i =1,2$. By replacing $\gamma$ with $\eta$ in $[t_1,t_2]$, we can define a new path $\bar \gamma$ connecting $x$ and $y$ which does not intersect the interior of $C_1$ in the relative topology of $\partial E$. Now we can further consider the two cases in which $\im(\bar \gamma) \cap C_2 = \emptyset$ or $\im(\bar \gamma) \cap C_2 \neq \emptyset$. In the first case, there is nothing to prove. In the second, we can reason as above to replace $\bar \gamma$ with a continuous curve which connects $x,y$ and that does not intersect the interior of $C_2$.
\end{proof}

\subsection{Properties of stationary convex sets}
%

\begin{prop}\label{rigconv}
Let $\mathcal{L} : \R^n \to \R$ be $C^2$ and uniformly convex on compact sets. If $D \subset \R^n$ is an open and connected domain and $u: D \to \R$ is a locally convex function that satisfies in the sense of distributions
\begin{equation}\label{weakform}
\dv(D\mathcal{L}(Du)) = 0,
\end{equation}
then $u$ must be affine on $D$.
\end{prop}
\begin{proof}
Take any convex subdomain $D' \subset D$ where $u$ is Lipschitz. On this subdomain, by \eqref{weakform} we deduce that $u \in W^{2,2}$, see for instance \cite[Proposition 8.6]{GM}. Therefore, we can write \eqref{weakform} in the strong form a.e. in $D'$:
\begin{equation}\label{strongform}
\langle D^2\mathcal{L}(Du), D^2u\rangle = 0.
\end{equation}
Given two matrices $A$ being positive definite and $B$ being positive semidefinite, if $A = \sum_{i = 1}^n \lambda_ie_i\otimes e_i$ with $\lambda_i > 0$ and $\{e_i\}$ an orthonormal basis of eigenvectors for $A$, then
\[
\langle A, B\rangle = \sum_i\lambda_i \langle e_i\otimes e_i, B\rangle = \sum_i\lambda_i (Be_i,e_i) \ge 0,
\]
with equality if and only if $B = 0$. This observation together with \eqref{strongform} yields the thesis.
\end{proof}

\begin{Cor}\label{statvarif}
Let $\pi$ be a hyperplane in $\R^n$ and $\varphi:\pi \mapsto \R$ be a convex function. If $\|\delta\llbracket \Gamma_{\varphi} \rrbracket \|(B_\delta(x))=0$ for some $x \in \Gamma_{\varphi}$ and $\delta > 0$, then $\llbracket \Gamma_{\varphi} \rrbracket \llcorner B_\delta(x)$ must be a plane.
\end{Cor}
\begin{proof}
By \cite[Proposition 5.8]{HRT}, we find that the function $u$ is stationary for the parametric area functional, that is uniformly convex on compact sets. In particular \eqref{weakform} holds for $\varphi$, choosing $\mathcal{L}=\A$. Hence Proposition \ref{rigconv} concludes the proof.
\end{proof}


\begin{prop}\label{reg}
Let $C$ be a bounded convex set. Assume that for $x \in \partial C$ there exists a ball $B_\delta(x)$ such that $\partial C\cap B_\delta(x)$ has constant mean curvature, i.e. that there exists $\lambda \in \R$ such that: 
\begin{equation}\label{integral}
\int_{\partial^* C}\langle T_x\partial C, Dg\rangle d\Haus^{n-1} = \lambda\int_{\partial^* C}(n_{C},g)d\Haus^{n-1},
\end{equation}
for all $g \in C^\infty_c(B_\delta(x),\R^n)$. Then, $\partial C \cap B_\delta(x)$ is a smooth submanifold of $\R^n$.
\end{prop}
\begin{proof}
$\partial C$ is locally the graph of a convex function $\varphi$, hence locally Lipschitz. Hence, analogously to \cite[Lemma 7.3]{DLDPKT}, we can write equation \eqref{integral} as the elliptic PDE $\dv(D\A(D\varphi)) = \lambda$ for $\varphi$.
 Standard elliptic regularity, see \cite[Chapter 8]{GM}, yields the claim.
\end{proof}

\begin{prop}\label{varigraph}
Let $C \subset \R^n$ be a bounded convex set of dimension $n$. Let $x \in \partial C$ and $S$ a (possibly empty) closed set of dimension $d < n$.  If there exist an hyperplane $\pi\ni x$, $\delta > 0$ and a convex function $\varphi: B_\delta^\pi(x)\to \R$ such that $\partial C \cap B_{\delta,\Lip(\varphi)\delta}^\pi(x)= \Gamma_{\varphi,B_\delta^\pi(x)}^\pi$, and $\varphi$ satisfies in the sense of distributions
\begin{equation}\label{alm}
\dv(D\A(D\varphi))(y') = \lambda, \quad \forall y' \in B_\delta^\pi(x)\setminus p_\pi(S),
\end{equation}
then $\partial C$ has constant mean curvature $\lambda \in \R$ in a neighborhood of $x$ except for $S$.
Conversely, if $\partial C$ is of constant mean curvature $\lambda \in \R$ in a neighborhood of $x$ except for $S$, then for every convex function $\varphi$ for which Proposition \ref{graph} holds, $\varphi$ satisfies $\eqref{alm}$ on $B_\eps^\pi(x)\setminus p_\pi(S)$ for some sufficiently small $\eps > 0$.
\end{prop}
\begin{proof}
This readily follows from \cite[Lemma 7.3]{DLDPKT}.
\end{proof}

\begin{lemma}\label{div}
Let $\Phi : B_1\subset\R^d \to \R^d$ be continuous. Suppose that there exists a convex set $S \subset B_1$ with $\Haus^{d}(S) = 0$ and $f \in L^1_{\loc}(B_1)$ such that
$\dv(\Phi) = f$
in the sense of distributions in $B_1\setminus S$. Then, $\dv(\Phi) = f$ in the sense of distributions in $B_1$.
\end{lemma}
\begin{proof}
Since $S$ is convex and $\Haus^{d}(S) = 0$, $S$ is contained in a hyperplane $\pi$, that up to a rotation we can suppose to be $\pi = \{y \in \R^d: y_d = 0\}$.  Let $\eps \in \left(0,\frac{1}{2}\right)$, consider $\rho_\eps \in C^\infty_c(B_{\eps/2})$ to be the usual radial mollifier. Denote $\Phi_\eps=\Phi \ast \rho_\eps$, $f_\eps=f \ast \rho_\eps$ and observe that $\dv(\Phi_\eps)(x) = f_\eps(x)$ for every  $x \in B_{1-\eps} \setminus B_\eps(\pi)$. Denote $D_\eps^+ = \{y \in B_1: y_d \ge \eps\}$, $D_\eps^- = \{y \in B_1: y_d \le -\eps\}$, $\sigma_\eps^+ = \{y \in B_1: y_d = \eps\}$ and $\sigma_\eps^- = \{y \in B_1: y_d = -\eps\}$. Fix any $\psi \in C^\infty_c(B_1)$. Then, by the divergence theorem, for all $\eps$ small enough it holds:
\begin{align*}
&\int_{B_{1-\eps}\setminus B_\eps(\pi)}(\Phi_\eps(x),D\psi(x))dx +  \int_{B_{1-\eps}\setminus B_\eps(\pi)}f_\eps(x)\psi(x)dx=\int_{B_{1-\eps}\setminus B_\eps(\pi)}\dv(\psi(x)\Phi_\eps(x))dx\\
&= \int_{D_\eps^+}(\dv(\psi(x)\Phi_\eps(x))dx + \int_{D_\eps^-}(\dv(\psi(x)\Phi_\eps(x))dx= -\int_{\sigma_\eps^+}(\Phi_\eps,e_d)\psi d\Haus^{n-1} +  \int_{\sigma_\eps^-}(\Phi_\eps,e_d)\psi d\Haus^{n-1},
\end{align*}
where $e_d$ is the $d$-th element of the canonical basis of $\R^d$. By the strong convergence in  $L^1_{\loc}(B_1)$ of $\Phi_\eps$ and $f_\eps$ respectively to $\Phi$ and $f$, we see that
\[
\int_{B_{1-\eps}\setminus B_\eps(\pi)}(\Phi_\eps(x),D\psi(x))dx +  \int_{B_{1-\eps}\setminus B_\eps(\pi)}f_\eps(x)\psi(x)dx
\to
\int_{B_1}(\Phi(x),D\psi(x))dx +  \int_{B_1}f(x)\psi(x)dx.
\]
Moreover by continuity of $\Phi$,
\[
-\int_{\sigma_\eps^+}(\Phi_\eps,e_d)\psi d\Haus^{n-1} +  \int_{\sigma_\eps^-}(\Phi_\eps,e_d)\psi d\Haus^{n-1} \to 0, \text{ as }\eps \to 0^+,
\]
which concludes the proof.
\end{proof}

\subsection{A capillarity theorem}\label{s:captem}
The main result of this section is Theorem \ref{pointrigid}, which is a rigidity result for capillary surfaces lying in a wedge. To prove Theorem \ref{pointrigid}, we will obtain some intermediate results: Proposition \ref{pointsym}, Lemma \ref{grafo} and Lemma \ref{MIERSE}.

\begin{prop} \label{pointsym}
Consider two planes in $\R^3$, $\pi$ and $\pi'$, intersecting in a line $L$. Let $\Sigma$ be one of the open, convex cylindrical sectors of $\R^3 \setminus (\pi \cup \pi')$ with opening angle $\alpha < \pi$. Suppose $E$ is an open, non-empty, bounded and convex set with $E \subset \Sigma$ with the following properties:
\begin{enumerate}
\item $\partial E \cap L = \{p_0\}$;
\item $\mathcal{H}^2(\partial E \cap \pi) \neq 0 \neq \mathcal{H}^2(\partial E \cap \pi')$ and  $\partial E$ intersects $\pi\setminus L$ and $\pi' \setminus L$ with constant angle of $120$ degrees;
\label{anglecond}
\item $\partial E \cap \Sigma$ has constant mean curvature.\label{cmccond}
\end{enumerate}
Then, $E$ is symmetric with respect to the plane $\pi''$ orthogonal to $L$ and passing through $p_0$.
\end{prop}
We will first need the following technical result.
\begin{lemma}\label{grafo}
Consider the assumptions of Proposition \ref{pointsym}, and suppose without loss of generality that $p_0 = (0,0,0)$, $L = \{(x,0,0): x \in \R\}$, $\pi'' = \{(0,y,z): y,z \in \R\}$. Let $S := \overline{\partial E\cap \Sigma}$ and for every $r>0$ let
\[
S_r := S \cap \{(x,y,z) \in \R^3: y^2 + z^2 \le r^2\} \qquad \mbox{and} \qquad T_r := B_r^{\pi''}(0)\cap \overline \Sigma,
\]
Then, there exists $\delta > 0$ such that $S_\delta$ is the union of the graphs of two functions $f,g$, defined on $T_\delta$ with $f(0,0) = g(0,0) = 0$, $f \le g$, $f$ convex and $g$ concave.
\end{lemma}
\begin{proof}
We will be using the fact that $S$ is a smooth manifold with smooth boundary, excluding $p_0 = 0$. Indeed, it is a smooth surface in the interior by Proposition \ref{reg}, while smoothness up to the boundary (except for $p_0$) will be shown in Subsection \ref{subsubreg}. Define $ d_1 = \max_{p \in \partial E}\dist(p, L)$, and let $\bar p \in \partial E$ be a point that realizes the maximum. Next, parametrize the sets $\partial E\cap \pi$,  $\partial E\cap \pi'$ respectively with parametrizations $\sigma_i:[0,1] \to \R^3$ which are simple, Lipschitz, smooth in $(0,1)$ and $\sigma_i(0) =\sigma_i(1) = p_0 = 0$, for all $i=1,2$. For every $r>0$, we denote $D_{r} := \{p \in \overline \Sigma: \dist(p,L) \le r\}$. By convexity of $\partial E \cap \pi$ and $\partial E \cap \pi'$, there exists $d_2>0$ such that $\sigma_i'(t)$ is not parallel to $L$ for every $t \in (0,1)$ such that $\sigma_i(t) \in D_{d_2}$. 
We claim we can choose $0 <\delta \le \min\{\frac{d_1}{2}, \frac{d_2}{2}\}$ such that 
\begin{equation}\label{equalTdelta}
p_{\pi''}(S_\delta) = D_{\delta}\cap \pi'' = T_\delta.
\end{equation}
Indeed, $p_{\pi''}(S_{\delta}) \subseteq T_\delta$ is immediate from the definitions. To see the reverse inclusion, simply notice that by our choice of $d_2$,
$
p_{\pi''}(\sigma_i\cap D_{\delta}) =T_{\delta}^i,
$
for $i=1,2$, 
where $T_{\delta}^1:=T_{\delta}\cap \pi$ and $T_{\delta}^2:=T_{\delta}\cap \pi'$.
 Furthermore, by boundedness and convexity of $E$ and the linearity of $p_{\pi''}$, we deduce that $p_{\pi''}(S_{\delta}) = p_{\pi''}(\overline E \cap D_{\delta})$ is convex.
Thus, the triangle $\co(T_{\delta}^1\cup T_{\delta}^2) \subseteq p_{\pi''}(S_\delta)$. We can choose $\delta' < \delta$ such that
$
T_{\delta'} \subset \co(T_{\delta}^1\cup T_{\delta}^2).
$
Then, up to replace $\delta$ with $\delta'$, \eqref{equalTdelta} holds. Finally, we show that, if $\delta$ is sufficiently small, for all $x \in T_\delta$, we have that 
\begin{equation}\label{ab}
p_{\pi''}^{-1}(x) \cap S_{\delta} = \{a(x),b(x)\}
\end{equation}
with the first components satisfying $a_1(x) \le b_1(x)$. Notice that we are not excluding that $a(x) = b(x)$. We already know that  $p_{\pi''}^{-1}(x) \cap S_{\delta}$ is non-empty, and, due to the convexity of $E$, to prove \eqref{ab} we only need to show that it is not a non-trivial segment. If $x \in T_\delta^1\cup T_\delta^2$ and $p_{\pi''}^{-1}(x) \cap S_{\delta}$ were a non-trivial segment, it would be parallel to $L$, against our choice of $d_2$. Now suppose by contradiction that for a sequence $\delta_n$ with $\delta_n < \delta$ and  $\delta_n \to 0$ we have $x \in T_{\delta_n} \setminus T_{\delta_n}^1\cup T_{\delta_n}^2$ and $p_{\pi''}^{-1}(x) \cap S_{\delta_n}$ is a non-trivial segment. Then $\partial E \cap \Sigma \cap D_{\delta_n}$ contains a non-trivial segment parallel to $L$. Let $y_n$ be any point in the interior of that segment. Due to the smoothness of the surface, the blow-up of $\partial E$ at $y_n$ must be a plane $\pi_n\supset L$. Furthermore, by convexity, $E$ must lie in one of the two open half-spaces $C_n^1, C_n^2$ of $\R^3 \setminus \{\pi_n + y_n\}$. We may always suppose $E \subset C_n^1$. Notice that $0 \in C_n^1$, indeed if $0 \in \partial C_n^1 = \pi_n + y_n$, then $\pi_n + y_n$ would cut $\overline\Sigma$ into two halfs, one containing $\pi\cap \overline \Sigma$ and one containing $\pi'\cap \overline \Sigma$. As $E$ needs to be contained in one of the two halfs, this would go against Assumption \eqref{anglecond}. Thus, $0 \in C^1_n$. Let 
$
L_n^1 := (\pi_n + y_n)\cap \pi \subset \partial \Sigma$ and $L_n^2 := (\pi_n + y_n)\cap \pi' \subset \partial \Sigma.
$
Then, either
\[
\liminf_{n \to \infty}\dist(L_n^1,L) = 0 \text{ or } \liminf_{n \to \infty}\dist(L_n^2,L) = 0.
\]
Since
\[
E \subset \bigcap_{n} C_n^1,
\]
we deduce that either $\partial E\cap \pi=\{p_0\}$ or $\partial E\cap \pi'=\{p_0\}$. This yields once again a contradiction with Assumption \eqref{anglecond}. This shows \eqref{ab} up to choosing $\delta > 0$ small enough. We can then define $f$ and $g$ in the following way. Given $(0,y,z) \in T_\delta$, consider the only (possibily identical) vectors $a = (a_1,y,z), b = (b_1,y,z)$ with
\[
p_{\pi''}^{-1}((0,y,z)) \cap S_{\delta} = \{a,b\}
\]
and $a_1 \le b_1$. Then, define $f(y,z) = a_1$ and $g(y,z) = b_1$. Of course, $f \le g$ on $T_\delta$. To see that $f$ is convex, take any point $p \in \intt T_\delta$ with $B_\rho(p) \subset \intt T_\delta$ and notice that $\Epi(f)\cap (B_\rho(p)\times L)$ is a convex set since $E$ is convex. The proof that $g$ is concave is analogous. 
\end{proof}
\begin{proof}[Proof of Proposition \ref{pointsym}]
We exploit Alexandrov's moving plane method as in \cite[Theorems 4.1.1-4.1.16]{LOP}. After a rotation and a translation, we can suppose $L = \{(x,0,0): x \in \R\}$, $p_0 = 0$ and thus $\pi'' = \{(0,y,z): y,z \in \R\}$. For $t \in \R$, define $\pi''_t = \{(x,y,z) \in \R^3: x = t\}$, so that $\pi''_0 = \pi''$. For large $|t|$, $\pi''_t$ does not intersect $\overline E$. Let $t_2 < 0 < t_1$ be the minimum and the maximum $t$ for which $\pi''_t\cap \overline E\neq \emptyset$. Then, either $p_0\not \in \pi''_{t_1}$ or $p_0\not \in \pi''_{t_2}$, as otherwise $E$ would not be open. Suppose, without loss of generality, that $p_0\not \in \pi''_{t_1}$. Let $S := \overline{\partial E \cap \Sigma}$. With the notation of \cite{LOP}, we define for $A \subset \R^3$ 
\[
A^+_t := A \cap\{(x,y,z) \in \R^3: x \ge t\}, \quad A^-_t := A \cap\{(x,y,z) \in \R^3: x \le t\}
\]
and finally $A_t^*$ as the reflection of $A_t^+$ with respect to $\pi''_t$. First we claim that there exists $\eps >0$ such that for $t \in (t_1-\eps, t_1)$, the following two properties hold:
\begin{enumerate}[\quad (a)]
\item $S_t^*\setminus \partial S_t^* \subset E$, where here $\partial$ is the boundary in the manifold sense;\label{firlast}
\item $S_t^+$ is a graph on $\pi_t''$.\label{selast}
\end{enumerate}
We define the validity of both \eqref{firlast}-\eqref{selast} as Property $(P_t)$. To show the claim, we will use the fact that $S$ is a smooth manifold with smooth boundary, excluding $p_0 = 0$, as in Lemma \ref{grafo}. With this observation at hand, the claim is easily proved. Indeed, at any point $p' \in \pi''_{t_1}\cap S$, the tangent plane to $S$ must be $\pi''_{t_1}$ by the Lagrange Multiplier Theorem. In particular, we deduce that $\pi''_{t_1}\cap S$ does not contain any point of $\partial S \setminus \{p_0\}$ due to the constant angle condition \eqref{anglecond}. For a small $\eps > 0$ and $t \in (t_1-\eps,t_1)$, Property $(P_t)$ is now a consequence of embeddedness of $S$, which in turn is a consequence of the convexity of $E$, see \cite[Theorems 4.1.1]{LOP} for details. Now we can define
$
t_0 := \inf\{t: (P_s) \text{ holds } \forall s \in (t,t_1)\}.
$
Notice that this set is non-empty by the claim above, hence $t_0 < t_1$. We wish to show that $t_0 = 0$. We observe that $0 \le t_0$, as otherwise $\partial E \cap L$ would contain more elements than just $p_0$. Let us now exclude by contradiction the case $t_0 > 0$. If $t_0 > 0$, then either \eqref{firlast} does not hold for a sequence of values $T_n \in (0, t_0)$, with $T_n\to t_0$, and it holds for every $t\in (t_0,t_1)$, and hence 
\begin{equation}\label{interior}
S_{t_0}^* \text{ intersects } S_{t_0}^- \text{ at an interior point}
\end{equation}
or \eqref{selast} does not hold for a sequence of values $T_n \in (0, t_0)$, with $T_n\to t_0$, and it holds for every $t\in (t_0,t_1)$, and hence
\begin{equation}\label{boundary}
\partial (S_{t_0}^*) \text{ intersects } \partial (S_{t_0}^-)\setminus \{p_0\} \text{ tangentially}.
\end{equation}
In case \eqref{interior} holds, we see as in \cite[Theorem 4.1.1]{LOP} that, since $t_0$ is the first intersection time, $S_{t_0}^*$ is tangent to $S_{t_0}^-$ at an intersection point $p$. Moreover, the two surfaces have the same mean curvature and $S_{t_0}^* \le S_{t_0}^-$ in a neighborhood of $p$, i.e. writing the two surfaces in a neighborhood of $p$ as the graphs of $u$ and $v$ respectively over the common tangent plane, we have $u \le v$. A similar situation happens in case \eqref{selast} occurs. In any case we may apply \cite[Corollary 3.2.5]{LOP} to conclude that $S_{t_0}^*$ and $S_{t_0}^-$ agree on some neighborhood. As in \cite[Theorem 4.1.1]{LOP}, one can then prove that $S_{t_0}^* = S_{t_0}^-$, which anyway is in contradiction with the fact that $p_0 \in S$. Thus, $t_0 = 0$. We then have that $S_0^+$ is a graph over $\pi_0''=\pi''$. Now, if $S_0^- = S_0^*$, then the proof is concluded. We may suppose, by  contradiction, that this is not the case. By the definition of $t_0$, $S_t^*$ does not intersect $S_t^-$ for any $t > 0$. Thus, if $S_0^- \neq S_0^*$, we see that $S_0^-$ cannot intersect $S_0^*$ except at $p_0 = 0$. Otherwise, we may use again \cite[Corollary 3.2.5]{LOP} to conclude $S_0^- = S_0^*$ as above. We apply Lemma \ref{grafo} to infer the existence of $\delta > 0$ such that 
\begin{equation}\label{claim:graph}
S_0^-\cap \{(x,y,z) \in \overline \Sigma: y^2 + z^2 \le \delta^2\}
\end{equation}
is the graph over $
T_\delta \subset \pi''$ 
of a convex function $f: T_\delta \to \R$. By \eqref{anglecond}-\eqref{cmccond} and Proposition \ref{varigraph}, we find that, for some $\lambda, \gamma \in \R$, $f$ satisfies in the weak sense the system 
\begin{equation}\label{sysun}
\begin{cases}
\dv(D\mathcal{A}(Df)) = \lambda,&\quad \text{ in } \intt_{\pi''} T_\delta,\\
(D\A(Df),\nu_{\pi}) = -\cos \gamma, &\quad \text{ on } T_\delta^1,\\
(D\A(Df),\nu_{\pi'}) = -\cos \gamma, &\quad \text{ on }T_\delta^2,
\end{cases}
\end{equation}
where $T_\delta^1, T_\delta^2$ are the two segments contained in $\partial T_\delta$, $\nu_\pi$ and $\nu_{\pi'}$ are the outer normals of $T_\delta^1$ and $T_\delta^2$ respectively. Furthermore, since $0 \in S_0^-$, we find that $f(0,0) = 0$. Since $S_0^+$ is a graph over $\pi''$, also $S^*_0$ is. Let $v$ be the function that describes $S^*_0$. $v$ satisfies the same system \eqref{sysun} as $f$ on $T_\delta$, and fulfills $v(0,0) = 0$, since $0 \in S_0^+$. As $S_0^*$ does not intersect $S_0^-$ except at $0$, we further get $f < v$ in $T_\delta \setminus \{0\}$. By compactness, for a small $\eps > 0$, $f_\eps := f + \eps$ still solves system \eqref{sysun} and is such that  
\[
f_\eps < v \text{ on } \{(0,y,z) \cap \overline \Sigma: y^2 + z^2 = \delta^2\} \subset T_\delta.
\]
Furthermore, $f_\eps(0,0) = \eps > v(0,0) = 0$. This is in contradiction with \cite[Theorem 5.1]{FINB}, which implies that $f_\eps < v$ throughout $T_\delta$. Thus, $f = v$, and hence $S_0^* = S_0^-$ and $\pi''$ is a plane of symmetry for $\partial E$, as wanted.
\end{proof}

We need one last lemma before proving Theorem \ref{pointrigid}.

\begin{lemma}\label{MIERSE}
Consider the assumptions and conclusion of Lemma \ref{grafo}. Then, if the functions $f$, $g : T_\delta \to \R$ of Lemma \ref{grafo} additionally satisfy
\begin{equation}\label{boundfg}
|Df(x)| + |Dg(x)| \le C, \quad \forall x \in T_\delta \setminus \{0\},
\end{equation}
then
\begin{equation}\label{C11regprel}
f,g \in C^{1,1}(T_\delta) \cap C^\infty(T_\delta \setminus \{0\})
\end{equation}
and for every $\eta\in (0,1)$ there exists $C > 0$ and $R_0>0$ such that for every $R\in (0,R_0)$,
\begin{equation}\label{c2etalem}
[D^2f]_{C^{\eta}(T_\delta\cap B_R^c(0))} + [D^2g]_{C^{\eta}(T_\delta\cap B_R^c(0))}\le \frac{C}{R^\eta}.
\end{equation}
\end{lemma}
\begin{proof}
We only need to show the assertion for $f$, the conclusion for $g$ is analogous. \eqref{C11regprel} will be shown in Subsection \ref{subsubreg}. It is enough to show \eqref{c2etalem}. To this end, recall that $f$ solves \eqref{sysun}.
As in \cite[Section 2]{MIE1}, we introduce the stream function $\psi$ such that
\begin{equation}\label{streamf}
\partial_y\psi = - (D\mathcal{A}(Df))_2 + \frac{\lambda}{2}z, \quad \partial_z\psi = (D\mathcal{A}(Df))_1 - \frac{\lambda}{2}y.
\end{equation}
The regularity of $f$ yields $\psi \in C^{1,1}(\overline T_\delta)\cap C^\infty(\overline T_\delta\setminus \{(0,0)\})$. Again in the notation of Lemma \ref{grafo}, we have $T_\delta \subset \pi''$. Assuming that $\Sigma$ has opening angle $\alpha \in (0,\pi)$, we can assume after a rotation that $T_\delta^i$ is parametrized in $\pi'' \sim\R^2$ by $sa^i$ for $s$ sufficiently small and
\[
a^i = \left(\cos\left( \frac{\alpha}{2}\right),(-1)^{i+1}\sin\left(\frac{\alpha}{2}\right)\right).
\]
By \eqref{streamf} we find that
\[
\frac{d}{ds}(\psi(sa^i)) = (-1)^{i + 1}\cos \gamma.
\]
Since $a_{2}^1 = -a_2^2$ and since we can assume without loss of generality that $\psi(0,0) = 0$, it follows that
\[
\psi(y,z) = \frac{\cos \gamma}{\sin\left(\frac{\alpha}{2}\right)}z = kz \text{ on }T_\delta^1\cup T_\delta^2.
\]
As in \cite{MIE2}, we invert \eqref{streamf} to get
\[
\partial_yf = \frac{F_2}{\sqrt{1- F^2}}, \quad \partial_z f = \frac{F_1}{\sqrt{1- F^2}}
\]
where
\[
F_1 = - \partial_y\psi + \frac{\lambda}{2}z, \quad F_2 = \partial_z \psi + \frac{\lambda}{2}y, \quad F^2 = F_1^2 + F_2^2 = \frac{|Df|^2}{1 + |Df|^2}.
\]
Thus,
\begin{equation}\label{dva}
\dv(B(y,z,D\psi)) = 0, \qquad \mbox{where} \qquad B:=(-\partial_z f,\partial_yf )= \left(- \frac{F_1}{\sqrt{1- F^2}},\frac{F_2}{\sqrt{1- F^2}}\right).
\end{equation}
It is easy to see that $B$ is smooth in all variable in its domain of definition. Moreover, for every $(y,z)$, a direct computation shows that $D_{p}B(y,z,p)$ is a symmetric matrix, which is uniformly positive definite since
\[
\max_{T_\delta} F^2= \max_{T_\delta}\frac{|Df|^2}{1 + |Df|^2} \overset{\eqref{boundfg}}{<} 1.
\]
We rewrite  \eqref{dva} as
\[
\langle A(y,z), D^2\psi(y,z)\rangle = h(y,z) \text{ in }\intt T_\delta, \quad \text{ and } \quad \psi = kz \text{ on }T_\delta^1\cup T_\delta^2,
\]
where $A$ and $f$ are Lipschitz in $T_\delta$ and $A$ is uniformly bounded from below. Now $w := \psi - kz$ solves
\begin{equation}\label{eqw}
\langle A(y,z), D^2w(y,z)\rangle = h(y,z) \text{ in }\intt T_\delta, \quad \text{ and } \quad w = 0  \text{ on }T_\delta^1\cup T_\delta^2.
\end{equation}
Notice that, since $w \in C^{1,1}(T_\delta)$ and $w = 0$ on $T_\delta^1\cup T_\delta^2$, then $Dw(0) = 0$ and hence
\begin{equation}\label{estw2}
|w(p)| \le C|p|^2 \text{ and } |Dw(p)| \le C|p|, \quad \forall p \in T_\delta.
\end{equation}
To conclude, we only need to show \eqref{c2etalem} for $w$, which would imply the same estimate for $f$. To do so, we fix $\eta\in(0,1)$ and we first consider annular sectors of the form
\[
K_R = \{(r\cos \theta ,r\sin \theta): R \le r \le 2R\} \text{ and } K'_R = \{(r\cos \theta ,r\sin \theta): \frac{1}{2}R \le r \le 4R\}.
\]
Fix $R > 0$ and much smaller than $\delta > 0$. Let $\varphi = \varphi(t)$ be a smooth non-negative cut-off function of $[1,2]$ inside $\left[\frac{1}{2},4\right]$ and set $u(x) := \varphi(|x|)$. Instead of studying $w$ directly, consider $w_0(x) := u \left(\frac{x}{R}\right) w(x)$, which solves
\[
\langle A(x), D^2w_0(x)\rangle = h_{1,R}(x) \quad\text{ in } T_\delta,
\]
with 
\[
h_{1,R}(x) := u \left(\frac{x}{R}\right)h(x) + \frac{1}{R}\left\langle A(x),Dw(x)\otimes Du\left(\frac{x}{R}\right) + Du \left(\frac{x}{R}\right)\otimes Dw(x)\right\rangle + \frac{w(x)}{R^2}\left\langle A(x),D^2u \left(\frac{x}{R}\right)\right\rangle.
\]
Due to the definition of $K'_R$ and \eqref{estw2}:
\begin{equation}\label{acca1r}
\|h_{1,R}\|_{C^0(K'_R)} \le C, \quad [h_{1,R}]_{C^\eta(K'_R)} \le \frac{C}{R^\eta}.
\end{equation}
This step is crucial to obtain the additional information that $w_0$ is zero near the curvilinear boundaries of $K'_R$. Further, recall that $w_0$ is zero on $T_\delta^1\cup T_\delta^2$. Now consider $w_{0,R} := w_0(Rx)$, which solves
\[
\langle A_R(x), D^2w_{0,R}(x) \rangle = h_{0,R}(x) \; \text{ in }K'_1,
\quad \mbox{ 
where
}
\quad
A_R(x) := A(Rx) \text{ and } h_{0,R}(x) := R^2{h_{1,R}}(Rx).
\]
Using the rescalings and Schauder estimates of \cite[Theorem 6.6]{GT}, we find
\begin{align*}
R^2&\|D^2w_0\|_{L^\infty(K'_R)} + R^{2 + \eta}[D^2w_0]_{C^\eta(K'_R)} = \|D^2w_{0,R}\|_{C^{\eta}(K'_1)} \\
&\le C(\|w_{0,R}\|_{C^0(K_1')} + \|h_{0,R}\|_{C^\eta(K_1')}) \overset{\eqref{estw2}}{\le} C(R^2 + R^2\|h_{1,R}\|_{C^0(K'_R)} + R^{2 + \eta}[h_{1,R}]_{C^\eta(K'_R)})\overset{\eqref{acca1r}}{\le} CR^2.
\end{align*}
This shows \eqref{c2etalem} for $w_0$ in $K'_R$. Since $w_0(x) = u\left(\frac{x}{R}\right) w(x)$, we see that then \eqref{c2etalem} holds for $w$ in $K_R$. We conclude the same for $f$. Thus, \eqref{c2etalem} is proved.
\end{proof}

\begin{Teo} \label{pointrigid}
Consider two planes in $\R^3$, $\pi$ and $\pi'$, intersecting in a line $L$. Let $\Sigma$ be one of the open, convex cylindrical sectors with opening angle $0 < \alpha < \pi$ of $\R^3 \setminus (\pi \cup \pi')$. Suppose $E$ is an open, non-empty, bounded and convex set with $E \subset \Sigma$, with the following properties:
\begin{enumerate}
\item $\partial E \cap L = \{p_0\}$;
\item\label{assu:rig} $\mathcal{H}^2(\partial E \cap \pi) \neq 0 \neq \mathcal{H}^2(\partial E \cap \pi')$ and  $\partial E$ intersects $\pi\setminus L$ and $\pi' \setminus L$ with constant angle of $120$ degrees;
\item\label{assu:cmc} $\partial E \cap \Sigma$ has constant mean curvature.
\end{enumerate}
If $\alpha \neq \frac{\pi}{3}$, no such surface exists. If $\alpha = \frac{\pi}{3}$, the two curves $\Gamma_1:=\partial_\pi(\partial E\cap \pi)$ and $\Gamma_2:=\partial_{\pi'}(\partial E\cap \pi')$ intersect $L$ tangentially at $p_0$, and more precisely the blow-up of $\llbracket \partial E\rrbracket$ at $p_0$ is given by
\begin{equation}\label{blowver}
\llbracket\pi \cap \overline \Sigma\rrbracket + \llbracket \pi'\cap \overline \Sigma \rrbracket.
\end{equation}
\end{Teo}
\begin{proof}
Rotating and translating, we suppose that $p_0 = 0$, $L = \{(x,0,0) \in \R^3: x \in \R\}$, $\pi'' = \{(0,y,z): y,z \in \R\}$ and that the lines $\pi\cap \pi''$ and $\pi'\cap \pi''$ form angles of opening $\frac{\alpha}{2}$ and $-\frac{\alpha}{2}$ with the line $\{(0,y,0): y \in \R\}$. By Proposition \ref{pointsym}, we will focus just on
\[
S := \overline{\partial E \cap \Sigma} \cap \{(x,y,z): x \ge 0\}.
\]
By Lemma \ref{grafo} we can write $S$ in a neighborhood of $p_0$ as the graph of a bounded concave function $u$ over the triangle 
$
T_\delta = B_\delta^{\pi''}(0) \cap \overline \Sigma.
$
This function $u$ solves system \eqref{sysun}. 
We divide the proof in three cases:
\subsection*{The case $\alpha < \frac{\pi}{3}$}
We can simply apply \cite{CONFIN}, see also \cite[Section 2]{CFR}, to infer that no such $u$ exists.

\subsection*{The case $\alpha = \frac{\pi}{3}$}
We can employ \cite{TAM} to deduce that the normal to $S$ is continuous up to $0$, and furthermore
\[
\lim_{(0,y,z) \to 0,\newline (0,y,z) \in T_\delta\setminus\{0\}}\frac{(-1,u_y,u_z)}{\sqrt{1 + u_y^2 + u_z^2}} = (0,-1,0).
\]
It follows that the blow-up of $S$ at $p_0$ is $L$. Hence, by Proposition \ref{uniquecone}, \eqref{blowver} readily follows. 
\subsection*{The case $\alpha > \frac{\pi}{3}$}

\begin{remark}
We are grateful to the anonymous referee for pointing out that we could study this case by means of the Heintze-Karcher inequality of the very recent work \cite{JWXZ} to conclude that $S$ is part of a sphere. Indeed we can apply here \cite[Theorem 1.5]{JWXZ}. The smoothness assumptions of $S$ are satisfied everywhere, except (a priori) at the point $x$. However, the smoothness at $x$ was only used in the proof of \cite[Theorem 1.5]{JWXZ} to show that, using the notation therein, $S_{r_y}(y+r_y \mathbf{k}_0)$ cannot have the first touching point at $x$. This fact is in our setting an easy consequence of the convexity of $E$. This observation can be used to replace the following argument that we originally devised to study this case. As both arguments are of independent interest, we decided to keep both of them.
\end{remark}

We first analyze the regularity of $S$, and then use it to show that $S$ is a subset of the sphere. The third and last step is to prove the technical claim \eqref{claimfi}.
\\
\\
\fbox{\emph{The regularity of S:}} By \eqref{assu:rig}-\eqref{assu:cmc}, $\overline{\partial E \cap \Sigma}$ is a smooth manifold with smooth boundary outside $0$, see Subsection \ref{subsubreg}. $S$ is bounded by the three curves $\gamma_1,\gamma_2$ and $\gamma_3$, where $\gamma_i := \Gamma_i \cap \{(x,y,z): x \geq 0\}$ for $i=1,2$ and $\gamma_3 := \overline{ \partial E \cap \pi''\cap \Sigma}$. The concatenation of $\gamma_1,\gamma_2$ and $\gamma_3$ is the boundary in the sense of manifolds of $S$. $\gamma_1$ intersects $\gamma_2$ at $0$ and $\gamma_3$ at $x_1$, while $\gamma_2$ and $\gamma_3$ intersect at $x_2$. These three points will be named vertices. From the regularity of $\overline{\partial E \cap \Sigma}$, we infer that $S$ is smooth up to the boundary (including the vertices $x_1,x_2$) except for the vertex $0$, i.e., at all points but $0$ the surface $S$ can be written as the graph of a smooth function defined on a disk (if at an interior point), on the complement of a smooth convex set (if at a boundary non-vertex point) as in Subsection \ref{subsubreg} and on a sector-like domain (if at $x_1,x_2$). Concerning the vertex $0$, since $\alpha > \frac{\pi}{3}$ we can employ \cite[Theorem 1]{SIMREG} to infer that the function $u$ representing $S$ in $T_\delta$ as above is such that
\begin{equation}\label{sim}
u \in C^1(T_\delta).
\end{equation}
Thus, using Lemma \ref{MIERSE}, we infer that $u$ is actually $C^{1,1}(T_\delta) \cap C^\infty(T_\delta \setminus\{0\})$, and, fixed any $\eta \in (0,1)$, we have the following controlled blow-up of $[u]_{C^{2,\eta}}$ for all small enough $R > 0$: 
\begin{equation}\label{c2eta}
[D^2u]_{C^{\eta}(T_\delta\cap B_R^c(0))} \le \frac{C}{R^\eta}.
\end{equation}
From now on, we fix $\eta\in (0,1)$. In conclusion, $S$ is globally a $C^{1,1}$ manifold with boundary, smooth except at $0$, where we have the additional information \eqref{c2eta}. Moreover, due to its smoothness and Proposition \ref{pointsym},
\begin{equation}\label{secangle}
S \text{ intersects the curve $\gamma_3$ with an angle of $90$ degrees,}
\end{equation}
\begin{equation}\label{novanta}
\text{ $\gamma_i$ intersects $\gamma_3$ with an angle of 90 degrees at $x_i$, for $i=1,2$,}
\end{equation}
and, from \eqref{sim}, we also have that
\begin{equation}\label{rig:beta}
\text{the curves $\gamma_1$ and $\gamma_2$ intersect at $0$ forming an angle $0 < \beta < \pi$.}
\end{equation}
Notice that the curves enjoy the following property
\begin{equation}\label{reggamma}
\gamma_i \text{ is a }C^{1,\eta} \text{ curve for all }i=1,2,3.
\end{equation}
\fbox{\emph{$S$ must be part of a sphere and contradiction:}} Our aim is to show that combining \eqref{secangle} and assumption \eqref{assu:rig} we obtain that $S$ is actually part of a sphere. If we manage to do so, we get a contradiction. Indeed, denoting with $0$ and $p$ the points given by $S \cap \{(0,y,0),y \in \R\}$, we see from \eqref{secangle} that $S$ must intersect $\gamma_3$ at $p$ with an angle of $90$ degrees, and from \eqref{sim} that $S$ must intersect $L$ at $0$ with an angle of less than $90$ degrees. Thus, $S \cap \{(x,y,0): x,y\in \R\}$ cannot be part of a circle and $S$ cannot be part of a sphere. \\
\\
To show that $S$ is part of a sphere, we adapt \cite[Section 2]{FMC}. Recalling that $\beta$ is defined in \eqref{rig:beta}, let
\begin{equation}\label{set:Q}
Q := \left\{(x,y) \in \R^2: x \le 0, y \in \left[\frac{-\beta}{2},\frac{\beta}{2}\right]\right\}
\end{equation}
and
$
p_1 := (0,-\beta/2)$, $p_2 := (0,\beta/2)
$.
We claim that:
\begin{equation}\label{claimfi}
\mbox{there exists a conformal parametrization  $\Phi: Q \to S \subset \R^3$ satisfying the following properties:}
\end{equation}
\[
\Phi(p_i) = x_i, \forall i=1,2, \quad \lim_{z \to \infty, z \in Q}\Phi(x) = 0;
\]
additionally, for some $\delta \in (0,1)$, we have that $\Phi \in C^{1,\delta}(Q)$, $\Phi \in C^{2,\delta}(Q\setminus\{p_1,p_2\})$ and, crucially,
\begin{equation}\label{control}
\lim_{z \to \infty, z \in Q}|D^2\Phi(z)| \to 0, \quad |D^2\Phi(z)| \le \frac{C}{\dist(z,p_i)^{1-\delta}}, \text{ for all $z$ close to $p_i$}.
\end{equation} 
Recall that $\Phi$ is said to be conformal if, everywhere in $\intt Q$,
\begin{equation}\label{conformal}
(\partial_x \Phi,\partial_y\Phi) = 0 \text{ and } \quad |\partial_x\Phi| = |\partial_y \Phi|.
\end{equation}
We will prove claim \eqref{claimfi} later. Assuming its validity, we now conclude the proof of Theorem \ref{pointrigid}.
Consider the second fundamental form of $S$ in these coordinates
\[
II = \left(\begin{array}{cc} L & M \\ M & N\end{array}\right).
\]
It is useful now to introduce the complex coordinates $z = x + iy$. The computations of \cite[Section 1]{TRIF1} show that in conformal coordinates, the so-called \emph{Hopf differential}
\[
f(z) := L - N - 2iM
\]
is holomorphic in $\intt Q$, due to the fact that $S$ has constant mean curvature. Since $\Phi$ is $C^{2,\delta}$ in $Q \setminus \{p_1,p_2\}$, $f$ is continuous in the same domain. Recall in fact that $L,M,N$ are given by the scalar product of second derivatives of $\Phi$ with the unitary normal of $S$. We apply the Terquem-Joachimsthal Theorem \cite[Theorem 9, Chapter 4]{MSp}, which states that, if a smooth curve $\gamma$ is a line of curvature for $M_1$, i.e. if $\gamma'$ is at all times an eigenvector of the shape operator of $M_1$, then another surface $M_2$ intersects $M_1$ along $\gamma$ with a constant angle if and only if $\gamma$ is also a line of curvature for $M_2$. Since for a plane the shape operator is the zero matrix, every curve on the plane is a line of curvature. Thus, the constant angle condition of \eqref{assu:rig}-\eqref{novanta} yields that $\gamma_i$ is a line of curvature for $S$, for all $i =1,2,3$. Due to our choice of $Q$ and the fact that the shape operator is given by $I^{-1}II$, denoting $e_1, e_2$ to be the canonical basis of $\R^2$, this amounts to ask that:
\begin{equation}\label{rig:I}
I^{-1}II(te_1 + p_i)e_1 = k_i(t)e_1, \quad \forall i=1,2, \, \forall t \in (-\infty,0); \qquad I^{-1}II\left(\frac{\beta}{2}se_2\right)e_2 = k(s)e_2, \quad \forall s \in (-1,1),
\end{equation}
for some functions $k_i, k$ representing the eigenvalues of the shape operator. Equivalently, since \eqref{conformal} implies that the first fundamental form is diagonal, we can rewrite \eqref{rig:I} as
\begin{equation}\label{rig:II}
(II(te_1 + p_i)e_1,e_2) = 0, \quad \forall i=1,2, \, \forall t \in (-\infty,0); \qquad \left(e_1, II\left(\frac{\beta}{2}se_2\right)e_2\right) = 0, \quad \forall s \in (-1,1).
\end{equation}
This implies that $M = 0$ on $\partial Q \setminus \{p_1,p_2\}$. Suppose for a moment that $M$ extends continuously on $\{p_1,p_2\}$ and hence that $M = 0$ on $\partial Q$. Then, by \eqref{regphi3} and the maximum principle for harmonic functions, we would find that $M \equiv 0$ on $Q$. This shows that, since $f$ is holomorphic, $L-N \equiv c \in \R$. Using once again \eqref{regphi3}, we find that $c = 0$ and hence that $L = N$ on $Q$, $M = 0$ on $Q$, i.e. that $S$ is part of a sphere, as wanted. Hence we are only left to show that $M$ extends continuously to $p_i$. 
\\
\\
Upon translating, suppose $p_i = 0$. Then, we have an harmonic function $M$ defined on a subset of a quadrant of $\R^2$, say on
$
\Omega := B_r(0) \cap \{(x,y): x \le 0, y \le 0\}\setminus\{0\},
$
with the properties that $M \in C^\infty(\intt \Omega)\cap C^0(\Omega)$ and
$
M(0,y) = 0, M(x,0) = 0$, for $|x|, |y| \le r, x<0, y < 0$. Moreover, by \eqref{control},
$
|M(z)| \le C|z|^{\delta-1}
$
for some $\delta \in (0,1)$.
We extend $M$ to $B_r$ in the following way, which is a simple variant of Schwarz Reflection Principle:
\begin{equation}\label{schwa}
v(x,y) := \begin{cases}
M(x,y), \text{ if } x, y \le 0,\\
-M(-x,y), \text{ if } x \ge 0, y \le 0,\\
-M(x,-y), \text{ if } x \le 0, y \ge 0,\\
M(-x,-y), \text{ if } x \ge 0, y \ge 0.
\end{cases}
\end{equation}
With this extension, $v$ is continuous in $B_r \setminus \{0\}$, and fulfills the local mean value property, hence we can use \cite[Theorem 1.24]{ABR} to infer that $v$ is also harmonic in $B_r\setminus\{0\}$. By construction, we still have that
\begin{equation}\label{expv}
|v(z)| \le \frac{C}{|z|^{1-\delta}}
\end{equation}
for some $\delta \in (0,1)$, and
$
v(0,y) = 0, v(x,0) = 0,$ for  $|x|, |y| \le r, (x,y) \neq 0$.
To show that $0$ is a removable singularity, we use \cite[Theorem 9.7]{ABR} to write $v = v_1 + v_2$,
where $v_1$ is harmonic in $B_r$ and $v_2$ is harmonic in $\R^2\setminus\{0\}$ with
\[
\lim_{|z| \to \infty}(v_2(z) - b\log(|z|)) = 0,
\]
for some constant $b$. It is therefore convenient to rewrite
$
v = v_1 + v_3 + b\log(|z|),
$
where $v_3 := v_2 - b\log(|z|)$ is such that $\lim_{z \to \infty}v_3 = 0$. Our aim is to show that $v_3 \equiv 0$ and $b = 0$. Let us start by showing $v_3 \equiv 0$. First, considering $v_4(z) := v_3\left(\frac{z}{|z|^2}\right)$, we see that $v_4$ is harmonic in $\R^2\setminus \{0\}$ with $\lim_{|z| \to 0}v_4(z) = 0$. Thus, $v_4$ extends to an harmonic function in $\R^2$. Moreover, from \eqref{expv}, we see that $v_4$ satisfies
$
|v_4(z)| \le C|z|^{1-\delta}
$
for all $z$ sufficiently large. The mean value property shows that $Dv_4 \equiv 0$ in $\R^2$ and, since $v_4(0) = 0$, we obtain $v_4 \equiv 0$, which implies $v_3 \equiv 0$. Thus,
$
v = v_1 + b\log(|z|),
$
and hence
\[
0 = v(x,0) = v_1(x,0) + b\log(|x|), \quad \forall 0 < |x| < r.
\]
Since $v_1$ is harmonic inside $B_r$, then $v_1$ is bounded inside $B_r$. This yields $b= 0$ and concludes the proof.
\\
\\
\fbox{\emph{Proof of claim \eqref{claimfi}:}}
We need to build $\Phi$ as in \eqref{claimfi}. The map $\Phi$ will be obtained in \eqref{opop} as a composition of maps with suitable properties. We start from the following parametrization of $S$. Using the convexity of $E$, we write $S$ as the graph $G_F$ of a function $F$ over a subset $S'$ of a sphere $\mathbb{S}$ contained in $E$. Let $\tilde p$ be the point of $S'$ such that $G_F(\tilde p) = 0$. Then, due to the regularity of $S$, $F: S' \to (0,+ \infty)$ is a smooth function outside of $\tilde p$, it is globally $C^{1,1}$ and its second derivatives fulfill estimate \eqref{c2etalem} near $\tilde p$. Now we extend $F$ to $F': \mathbb{S} \to S$ in such a way that, for some $C > 0$:
\begin{enumerate}[(a)]
\item $F' \in C^{1,1}(\mathbb{S})$;\label{AAA'}
\item $F'$ is smooth in $S'\setminus\{\tilde p\}$ and $C^{2,\eta}(\mathbb{S}\setminus \{\tilde p\})$; \label{BBB'}
\item for every small $R > 0$, \label{CCC'}
\[
[D^2F']_{C^\eta(\mathbb{S} \setminus B_R(\tilde p))} \le \frac{C}{R^\eta}.
\]
\end{enumerate}
One way to do so is to consider the stereographical projection $\pi_q: \mathbb{S}\setminus \{q\} \to \R^2$ for some $q \notin S'$ and extend $F'' := F \circ \pi_q^{-1}$ from $\pi_q(S')$ to the whole $\R^2$ in such a way that the extension $F'''$ fulfills
\begin{itemize}
\item $F''' = F''$ in $\pi_q(S')$;
\item $F''' \in C_{\loc}^{1,1}(\R^2)$;
\item $F'''$ is smooth in $\R^2\setminus \pi_q(\tilde p)$; 
\item for every small $R > 0$, 
\[
[D^2F''']_{C^\eta(\R^2 \setminus B_R(\pi_q(\tilde p)))} \le \frac{C}{R^\eta}.
\]
\end{itemize}
Once this is achieved, we can set $F':= \varphi F''' + (1-\varphi)$, where $\varphi$ is a suitable non-negative cut-off function of the compact set $\pi_q(S')$. This gives the required extension $F'$. The precise assumptions to do so and a sketch of the proof is given in Lemma \ref{lem:extension} below, see also Remark \ref{remarkone}. 
\\
\\
Let $M := G_{F'}(\SSS)$. Consider any point $q_0 \in \SSS \setminus S'$ and define 
\[
G' := G_{F'}\circ \pi_{q_0}^{-1}: \R^2 \to M \subset \R^3,
\]
where $\pi_{q_0}: \SSS \to \R^2$, is the stereographical projection based in $q_0$, as above. In these charts, we introduce the metric tensor given by
\begin{equation}\label{gij2}
 g_{ij}(y) = (\partial_iG'( (G')^{-1}(y)), \partial_jG'((G')^{-1}(y))), \quad \forall y \in M, \quad \forall \; 1\le i,j \le 2,
\end{equation}
Due to \eqref{AAA'} and the fact that $G_{F'}$ is a (radial) graph, $(g_{ij})$ is a Lipschitz tensor which is bounded from below and from above by two positive constants in the sense of quadratic forms. We employ \cite[Theorem 3.1.1]{JOST} to find a new $C^{1,\eta}$ parametrization
$
h: \SSS \to M
$
which is a conformal diffeomorphism. We introduce $\Psi := h \circ \pi_q^{-1}$, for any (fixed) $q \notin h^{-1}(S)$. Since $\pi_q$ is a smooth conformal diffeomorphism, $\Psi \in C^{1,\eta}(\R^2,M\setminus \{h(q)\})$ is a conformal diffeomorphism. Define the compact set
\[
\tilde T := \Psi^{-1}(S).
\]
Furthermore, notice that, after a translation of the domain, we can further suppose that
\begin{equation}\label{zeroes}
0_{\R^2} \in \tilde T \text{ and } \Psi(0_{\R^2}) = 0_{\R^3} \in S. 
\end{equation}
From now on, we will not denote differently $0 \in \R^2$ and $0 \in \R^3$. Notice that, since the stereographic projection is smooth and due to our choice of $q \in \SSS \setminus h^{-1}(S)$, \eqref{AAA'}-\eqref{BBB'}-\eqref{CCC'} imply the following properties of $G'$:
\begin{itemize}
\item $G'\in C^{1,1}(\R^2,M\setminus\{G_{F'}(q)\})$ is a $C_{\loc}^{1,1}$ diffeomorphism;
\item $G' \in C_{\loc}^{2,\eta}(\R^2 \setminus \{0\},M\setminus\{G_{F'}(q)\})$ and $(G')^{-1}\in C_{\loc}^{2,\eta}(M\setminus\{G_{F'}(q)\},\R^2 \setminus \{0\})$;
\item given a ball $B'\subset \R^2$ containing $(G')^{-1}(S)$ in its interior and for every small $R > 0$, if $s_0 := (G')^{-1}(0)$,
\[
[D^2 G']_{C^\eta(B' \setminus B_R(s_0))} + [D^2(G'^{-1})]_{C^\eta(B' \setminus B_R(s_0))} \le \frac{C}{R^\eta}.
\]
\end{itemize}
Finally, let us define
\begin{equation}\label{vv}
v_1 := (G'^{-1}\circ \Psi)_1,\quad  v_2 := (G'^{-1}\circ \Psi)_{2}.
\end{equation}
Since $h$ is conformal, then $\Psi$ satisfies the following conformality relations, see \cite[(3.1.9)]{JOST},
\[
\sum_{i,j}g_{ij}(\Psi(x))\partial_{x}v_i\partial_xv_j = \sum_{i,j}g_{ij}(\Psi(x))\partial_{y}v_i\partial_yv_j,\quad \sum_{i,j}g_{ij}(\Psi(x))\partial_xv_i\partial_yv_j = 0.
\]
By \eqref{gij2}, the metric is chosen to preserve the standard scalar product in $\R^3$, hence $\Psi$ is conformal in the sense of \eqref{conformal}. Now we wish to show the following properties of $\Psi$:
\begin{enumerate}[(A)]
\item $\Psi\in C^{1,\eta}(\R^2,M\setminus\{h(q)\})$ is a $C^{1,\eta}$ diffeomorphism;\label{AAAAA}
\item $\Psi \in C^{2,\eta}$ in $\R^2 \setminus \{0\}$ and $\Psi^{-1} \in C^{2,\eta}$ in $M \setminus \{0\}$; \label{BBBBB}
\item given an open ball $B \subset \R^2$ containing $\tilde T$ and for every small $R > 0$, \label{CCCCC}
\[
\|D^2\Psi\|_{L^\infty(B \setminus B_R(0))} + \|D^2(\Psi^{-1})\|_{L^\infty(B \setminus B_R(0))} \le \frac{C}{R^{1-\eta}}, \quad 
[D^2\Psi]_{C^\eta(B \setminus B_R(0))} + [D^2(\Psi^{-1})]_{C^\eta(B \setminus B_R(0))} \le \frac{C}{R}.
\]
\end{enumerate}
We already know \eqref{AAAAA}. We only need to show \eqref{BBBBB}-\eqref{CCCCC}. To this aim, by \eqref{vv}, we only need to prove the analogous properties for $v_1, v_2$ and their inverse $(w_1,w_2) := (v_1,v_2)^{-1}$. However, $v_1,v_2$ and $w_1,w_2$ solve the elliptic systems of \cite[(3.1.14) and (3.1.19)]{JOST} respectively, and hence the required regularity follows from classical elliptic estimates, see \cite[Theorem 6.2, Corollary 6.3]{GT}. Near the singular point $0$, one can use the same reasoning we employed at the end of Lemma \ref{MIERSE}.
\\
\\
So far, we have obtained a conformal, $C^{1,\eta}$ map
$
\Psi: B \to M \subset \R^3,
$
which is a diffeomorphism onto its image. Recall that $B$ is an open ball which contains $\tilde T= (\pi_q\circ h^{-1})(S) = \Psi^{-1}(S)$.
Roughly speaking, we wish to substitute the domain $\tilde T$ with the reference triangular domain
\[
T := \left\{(R\cos\theta,R\sin\theta): R \in [0,1], \theta \in \left[\frac{-\beta}{2},\frac{\beta}{2}\right]\right\}.
\]
To do so, we will use the Riemann Mapping theorem, which preserves conformality of $\Psi$. First, we need to study the regularity of the domain $\tilde T$. In particular, we wish to show that $\tilde T$ is a curvilinear triangle according to the definition below.

\begin{Def}\label{curvtrian}
Let $\eta \in (0,1)$. A compact set $D \subset \R^2$ is called a \emph{curvilinear triangle} if the following conditions are fulfilled. $D$ is a connected and simply connected Lipschitz set. Moreover, its boundary is the concatenation of three simple curves, $\sigma_i\in C^{1,\eta}([0,1],\R^2)$, $i=1,2,3$, such that
$
\sigma_i((0,1))\cap \sigma_j((0,1)) = \emptyset,$ for $i\neq j$.
Moreover, $\sigma_i$ intersects $\sigma_j$ in exactly one point forming an opening angle in $(0,\pi)$. The three distinct points $x_{1},x_{2},x_{3}$ at which the curves intersect are called \emph{vertices}, and the corresponding angles are denoted by $\gamma_i \in (0,\pi)$. Furthermore, every point $y$ of $\partial D$ fulfills one of the following two conditions.
\begin{enumerate}[(I)]
\item \label{intx} If $y \in \partial D \setminus \{x_{1},x_{2},x_{3}\}$, then there exists an open neighborhood $V = V(y)$, and a $C^{2,\eta}$ diffeomorphism 
$
F_y : V \to B_1(0) \subset \R^2
$
such that
\[
F_y(D \cap V) = B_1(0)\cap\{(a,b): a \ge 0\}, \quad F_y((\partial D)\cap V) = B_1(0) \cap \{(0,b): b \in \R\}.
\]
\item \label{badv} If $y\in \{x_{1},x_{2},x_{3}\}$, say $y = x_{i}$, then there exists an open neighborhood $V = V(y)$, and a $C^{1,\eta}$ diffeomorphism 
$
F_y :V \to B_1(0) \subset \R^2
$
with $DF_y(y)$ being a positive multiple of a rotation, such that
\[
F_y(D \cap V) = B_1(0)\cap \left\{(R\cos\theta,R\sin\theta): R>0, \theta \in \left[\frac{-\gamma_i}{2},\frac{\gamma_i}{2}\right]\right\},
\]
\[
F_y((\partial D) \cap V) = B_1(0)\cap \partial  \left\{(R\cos\theta,R\sin\theta): R>0, \theta \in \left[\frac{-\gamma_i}{2},\frac{\gamma_i}{2}\right]\right\}.
\]
and there exists $R_0>0$ such that
\[
\|D^2F_y\|_{L^\infty(V\setminus B_R(0))} \le \frac{C}{R^{1-\eta}},\quad [D^2F_y]_{C^\eta(V\setminus B_R(0))} \le \frac{C}{R}, \qquad \forall R\leq R_0.
\] 
\end{enumerate}
\end{Def}

\begin{remark}
Let us comment on Definition \ref{curvtrian}. First, the regularity we asked on the boundary is tailored to our needs, but of course different requirements are possible. In particular, notice that some vertices of $\tilde T$, namely $y_i:= \Psi^{-1}(x_i)$ for $i =1,2$, see \eqref{novanta}, admit charts of $C^{2,\eta}$ regularity. However, we will treat all vertices $0,x_1,x_2$ in the same way and hence we do not need to give different definitions. Secondly, in case $y \in \{x_{1},x_{2},x_{3}\}$, we added to the definition that $DF_y(y)$ is a multiple of a rotation, which will simplify some future computations. This is not restrictive. Indeed, denote $v_1 = e^{-i\gamma_i/2}$, $v_2 =e^{i\gamma_{i}/2}$ in the complex variables $\R^2\sim \mathbb C$. After a rotation of the domain,  we have that $DF_y(y)v_1 =\lambda v_i$ and $DF_y(y) v_2= \mu v_j$ with $\lambda, \mu > 0$ and $\{v_i,v_j\} = \{v_1,v_2\}$, due to the requirement that
\[
F_y((\partial D) \cap V) \subset \{a v_1: a > 0\}\cup \{b v_2: b > 0\}.
\]
We consider the new diffeomorphism $\tilde F(z) := kDF_y(y)^{-1}F_y(z)$, where $k > 0$ is a constant. It is not hard to see that $\tilde F$ fulfills all the required properties, once we consider a possibly smaller neighborhood of the point $y$.
\end{remark}

Since $\tilde T = \Psi^{-1}(S)$, $\Psi$ is a homeomorphism and $S$ is a connected, simply connected set, then the topological conditions of Definition \ref{curvtrian} are verified. Moreover, its boundary is given by the concatenation of $\alpha_1 := \Psi^{-1}\circ \gamma_1$, $\alpha_2 := \Psi^{-1}\circ \gamma_2$, and $\alpha_3 := \Psi^{-1}\circ \gamma_3$. The regularity of $\Psi^{-1}$ expressed in \eqref{AAAAA}-\eqref{BBBBB}-\eqref{CCCCC} and \eqref{reggamma} imply that $\alpha_i$ are $C^{1,\eta}$ curves. Moreover, their images only intersect at the distinct vertices $y_i = \Psi^{-1}(x_i)$ and $0 = \Psi^{-1}(0)$. Since $\Psi$ is conformal and \eqref{novanta}-\eqref{rig:beta} hold, we see that, for $i=1,2$,
\begin{equation}\label{novantaa}
\text{ $\alpha_i$ intersects $\alpha_3$ at an angle of 90 degrees at $y_i:= \Psi^{-1}(x_i)$}
\end{equation}
and also
\begin{equation}\label{rig:betaa}
\text{the curves $\alpha_1$ and $\alpha_2$ intersect at $0 \in \R^2$ forming an opening angle $\beta \in (0, \pi)$.}
\end{equation} 
From the regularity of $S$ that we discussed and the regularity of $\Psi$ expressed in \eqref{AAAAA}-\eqref{BBBBB}-\eqref{CCCCC}, we see that $\tilde T$ fulfills \eqref{intx} and \eqref{badv}. Hence $\tilde T$ is a curvilinear triangle in the sense of Definition \ref{curvtrian}. It can be checked, by the explicit definition of $T$, that $T$ is a curvilinear triangle as well.
\\
\\
By the Riemann Mapping Theorem, there exist $\delta,c,C>0$ and a map 
$
g : T \to \tilde T
$
such that:
\begin{enumerate}[(i)]
\item $g$ is biholomorphic from the interior of $T$ to the interior of $\tilde T$, and is a homeomorphism of $T$ to $\tilde T$; \label{A1}
\item $g \in C^{2,\delta}(T\setminus \{0,g^{-1}(y_1),g^{-1}(y_2)\})$; \label{A2}
\item \label{A3} $g(0) = 0$ and for all small $R > 0$ and all $z \in B_R(0)\cap T$,
$
c|z| \le |g(z)| \le C|z|.
$
Analogously, for all $i = 1,2$ and all $z\in T$ sufficiently close to $z_i$ it holds
$
c|z - z_i| \le |g(z) - g(z_i)| \le C|z-z_i|;
$
\item 
$
|g'(z)| \le C
$
for all $z \in T\setminus \{0,g^{-1}(y_1),g^{-1}(y_2)\}$, and for all small $R > 0$,
\[
\|g''\|_{L^\infty(T\setminus [{B_R(0)\cup B_R(g^{-1}(y_1)) \cup B_R(g^{-1}(y_2))}])} \le \frac{C}{R^{1-\delta}}.
\]
\label{A4}
\end{enumerate}
Indeed, since $\tilde T$ is a compact simply connected domain of $\R^2$, \eqref{A1} follows from the Riemann Mapping Theorem, see \cite[Theorem 3.2.1]{JOST} or \cite[Theorem 4.0.1, 5.1.1]{KRA}. As written in \cite[Theorem 3.2.1]{JOST}, we can prescribe the values of three points of the boundary of $T$ for $g$:
\begin{equation}\label{punti}
g(0) = 0, \, g(z_1) = y_1, \, g(z_2) = y_2,
\mbox{
where }
z_1 := \left(\cos\left(\frac{\beta}{2}\right),\sin\left(\frac{\beta}{2}\right)\right), \,  z_2 := \left(\cos\left(\frac{\beta}{2}\right),-\sin\left(\frac{\beta}{2}\right)\right).
\end{equation}
A map $g$ fulfilling \eqref{A1} and \eqref{punti} is unique, see \cite[Corollary 3.2.1]{JOST}. We will sketch the proof of  \eqref{A2}-\eqref{A3}-\eqref{A4} in Lemma \ref{lem:regg} below.
\\
\\
To conclude the present proof, we need one last map, this time explicit. Let $Q$ be as in \eqref{set:Q}. Then, the exponential map $e^z$ maps $Q$ biholomorphically inside $T$, once we consider $\mathbb{C}\cup \{\infty\}$ and we write, with a small abuse of notation, $e^z|_{Q}(\infty) = 0$. We can finally set 
\begin{equation}\label{opop}
\Phi := \Psi \circ g \circ e^z.
\end{equation}
By the regularity of $\Psi$ in \eqref{BBBBB} and $g$ in \eqref{A2}, we immediately obtain
$
\Phi \in C^{2,\delta}(Q\setminus \left\{p_1,p_2\right\}),
$
for some $\delta > 0$. Furthermore, we are composing a conformal map $\Psi$ in the sense of \eqref{conformal} with the holomorphic map $g\circ e^z$, thus $\Phi$ is still conformal, as a direct computation shows. In order to show \eqref{control}, we recall that
\[
p_1 = \left(0,-\frac{\beta}{2}\right), \qquad  p_2 = \left(0,\frac{\beta}{2}\right),
\]
and we claim the more precise bounds:
\begin{equation}\label{regphi3}
\text{$|D^2\Phi(z)| \le C|e^z|$ if $z \in Q$ and $|z|$ is sufficiently large}
\end{equation}
and
\begin{equation}\label{regphi4}
\text{$|D^2\Phi(z)| \le \frac{C}{\dist(z,p_i)^{1-\delta}}$, if $z \in Q$ is sufficiently close to $p_i$.}
\end{equation}
Write, for $f := e^z$ and any $a =1,2,3$, $i =1,2$, $j= 1,2$:
\begin{align*}
\partial_i \Phi^a = \sum_{k,\ell}\partial_k\Psi^a(g\circ f)\partial_\ell g^k(f) \partial_if^\ell,
\end{align*}
and
\begin{align*}
\partial_{ij} \Phi^a = \sum_{r,s,k,\ell}\partial_{kr}\Psi^a(g\circ f)\partial_sg^r(f)\partial_jf^s \partial_\ell g^k(f) \partial_if^\ell &+ \sum_{k,\ell,r}\partial_k\Psi^a(g\circ f)\partial_{\ell r} g^k(f)\partial_j f^r \partial_if^\ell \\
&+ \sum_{k,\ell}\partial_k\Psi^a(g\circ f)\partial_\ell g^k(f) \partial_{ij}f^\ell.
\end{align*}
Thus, noticing that $|\partial_if|,|\partial_{ij} f| \le |e^{z}|$ for all $i,j$, $z \in Q$ and using \eqref{AAAAA} and \eqref{A3}-\eqref{A4}, we can first estimate:
\begin{align*}
|\partial_{ij} \Phi^a| \le C\sum_{r,k}|\partial_{kr}\Psi^a(g\circ f)||e^z|^2 &+C|g''(e^z)||e^z|^2 + C|e^z|.
\end{align*}
Now \eqref{regphi3}-\eqref{regphi4} readily follow again from \eqref{CCCCC} and \eqref{A3}-\eqref{A4}. This concludes the proof of claim \eqref{claimfi} and hence the proof of Theorem \ref{pointrigid}.
\end{proof}

\begin{lemma}\label{lem:extension}
Let $K$ be a compact subset of $\R^2$ with non-empty interior, and let $\eta \in (0,1]$, $a \in \partial K$. Assume that for all couple of points $x,y\in K$ there exist differentiable curves $\gamma_i:[0,1] \to K$ such that $\gamma_i(0) = y, \gamma_i(1) = x$, and, for all $t,s\in[0,1]$, 
\begin{equation}\label{curve}
\begin{split}
&|\gamma_1'(t)| \le C|x-y|; \quad  |\gamma_2'(t)| \le C|x-y|; \quad |\gamma_2'(t)-\gamma_2'(s)| \le C|x- y|^{1+\eta};\\
&|\gamma_1'(t)-\gamma_1'(s)| \le \frac{C}{R^\eta}|x-y|^{1+\eta}\quad  \mbox{ and } \quad |\gamma_1(t) - a|\ge cR, \text{ if } |x-a|, |y-a| \ge R \geq 0.
\end{split}
\end{equation}
Assume that $F \in C^{2}(\intt K) \cap C^{1,\eta}(K)$ and that the following conditions are satisfied:
\begin{itemize}
\item if $\eta < 1$,  for all  $R > 0$ sufficiently small
\begin{equation}\label{Fexpagain}
\|D^2F\|_{L^\infty(K\setminus B_R(a))} \le \frac{C}{R^{1-\eta}}, \quad [D^2F]_{C^\eta(K\setminus B_R(a))} \le \frac{C}{R},
\end{equation}
\item if $\eta = 1$, there exists  $\alpha \in (0,1)$ such that
\begin{equation}\label{Fexpagain1}
\|D^2F\|_{L^\infty(K\setminus B_R(a))} \le C, \quad [D^2F]_{C^\alpha(K\setminus B_R(a))} \le \frac{C}{R^\alpha},
\end{equation}
and the curve of \eqref{curve} fulfills
\[
|\gamma_1'(t)-\gamma_1'(s)| \le \frac{C}{R^\alpha}|x-y|^{1+\alpha}.
\]
\end{itemize}
Then, there exists an extension $F' \in C^{2,\eta}_{\loc}(\R^2\setminus\{a\})\cap C^{1,\eta}_{\loc}(\R^2)$ with the same properties \eqref{Fexpagain} or \eqref{Fexpagain1} where the estimates are in $B_M(a)$ for any $M>0$, rather than in $K$, and the constant $C$ depends on $M$.
\end{lemma}

\begin{proof}[(Sketch of) Proof] The complete proof of this result is technical and lengthy, but is based on the classical Whitney extension Theorem, see \cite{WHI}. We will follow the proof given in \cite[Theorem 6.10]{EVG}. We will define a candidate extension $F'$, which has to be chosen with some care in order to take into account the singular point $a$. Thus, we will omit the proof that $F'$ is an extension with the properties above, as all the computations are rather standard, see \cite[Theorem 6.10]{EVG} for the first order estimates.
\\
\\
We start by noticing that assumption \eqref{curve} implies the following: for every couple $x,y \in K$,
\begin{align}
|F(x) - F(y)| &\le C\|DF\|_{\infty}|x-y|, \label{FLIP}\\
|F(x) - F(y) - (DF(y),x-y)| &\le C\|F\|_{C^{1,\eta}}|x-y|^{1+\eta},\label{C1}
\end{align}
and, if $|x-a|,|y-a|\ge R$:
\begin{align}
|DF(x) - DF(y)| &\le C\min\left\{|x-y|^\eta,\frac{|x-y|}{R^{1-\eta}}\right\},\label{DFDF}\\
|DF(x) - DF(y) - D^2F(y)(x-y)| &\le \frac{C}{R}|x-y|^{1+\eta}, \label{DFDF2}\\
|F(x) - F(y) - (DF(y),x-y) - \frac{1}{2}D^2F(y)[x-y,x-y]| &\le \frac{C}{R}|x-y|^{2+\eta}.\label{C2}
\end{align}
\eqref{FLIP}-\eqref{C1}-\eqref{DFDF}-\eqref{DFDF2} are immediate. \eqref{C2} easily follows from the following equalities for $\gamma_1$ as in \eqref{curve}:
\begin{align*}
F(x)& - F(y) - (DF(y),x-y) - \frac{1}{2}D^2F(y)[x-y,x-y] \\
&= \int_0^1\int_0^t(D^2F(\gamma_1(s))\gamma_1'(s),\gamma_1'(t))dsdt - \frac{1}{2}D^2F(y)[x-y,x-y]\\
&= \int_0^1\int_0^t(D^2F(\gamma_1(s))[\gamma_1'(s)-(x-y)],\gamma_1'(t))dsdt  + \int_0^1\int_0^t(D^2F(\gamma_1(s))(x-y),\gamma_1'(t)-(x-y))dsdt \\ &\qquad \qquad + \int_0^1\int_0^t [D^2F(\gamma_1(s))-D^2F(y)][x-y,x-y]dsdt.
\end{align*}
As already mentioned, we follow \cite[Theorem 6.10]{EVG}. Let $U := \R^2 \setminus K$. As in the proof of \cite[Theorem 6.10]{EVG}, we can consider a family of disjoint balls $\{B_{r_j}(x_j)\}_{j\in \N}$ such that $\bigcup_{j}B_{5r_j}(x_j) = U$. For all $x\in U$, let $r(x) := \frac{1}{20}\min\{1,\dist(x,K)\}$. By \cite[Theorem 6.10]{EVG}, we know that, if
\[
S_x = \{x_j: B_{10 r(x_j)}(x_j) \cap B_{10r(x)}(x)\neq \emptyset\},
\]
then $\mathcal H^0(S_x)$ is uniformly bounded and
\[
\frac{1}{3} \le \frac{r(x)}{r(x_j)} \le 3, \quad \forall x_j \in S_x.
\]
With the same procedure of \cite[Theorem 6.10]{EVG}, we can consider a partition of unity $\{\chi_j\}_{j\in \N}$ subordinate to $\{B_{r_j}(x_j)\}_{j\in \N}$ with the properties that for all $j,i \in \N, x \in U$
\[
\spt(\chi_j) \subset B_{10r(x_j)}(x_j), \quad\sum_j \chi_j(x) = 1,\quad \sum_j D^{(i)}\chi_j (x) = 0,\quad |D^{(i)}\chi_j(x)| \le \frac{C}{r^i(x)}.
\]
Now define $F'$ by
$
F' = F \text{ on } K
$
and, for $x \in U$,
\[
F'(x) = \sum_j(F(s_j) + (DF(s_j),x-s_j) + \frac{1}{2}D^2F(s_j)[x-s_j,x-s_j])\chi_j(x),
\]
where $s_j \in K$ is chosen in the following way. Let $\rho = \max_{y \in K}|y-a|$. Given any $x \in U$, find $b(x) \in \partial K$ such that
\begin{equation}\label{propsj}
|x - b(x)| \le C\dist(x,\partial K) \qquad \text{ and } \qquad |b(x)-a| \ge \frac{\min\{\rho,|x-a|\}}{10}.
\end{equation}
This is possible with a constant independent of $x$. Indeed, consider a point $y \in \partial K$ such that $\dist(x,\partial K) = |x- y|$. If $|y-a| \ge \frac{\min\{\rho,|x-a|\}}{10}$, then set $b(x) := y$. Otherwise, take $b \in \partial K\cap \partial B_{\min\{\rho,|x-a|\}}(a)$, which exists since $ K$ is connected, due to \eqref{curve}. Thus, $b \in \partial K$ and $|b-a| = \min\{\rho,|x-a|\}$ by construction. Furthermore, we have
\[
\min\{\rho,|x-a|\} \le |x -a | \le |x-y| + |y-a| \le \dist(x,\partial K)+ \frac{\min\{\rho,|x-a|\}}{10}.
\]
Thus,
$
\min\{\rho,|x-a|\} \le C\dist(x,\partial K),
$
and hence
\[
|x - b| \le |x-y| + |y-a| + |a-b| \le \dist(x,\partial K) + \frac{11}{10}\min\{\rho,|x-a|\} \le C\dist(x,\partial K).
\]
We can thus set $b(x) := b$. Now we define $s_j := b(x_j)$.
\\
\\
Using this setup, namely \eqref{FLIP}-\eqref{C1}-\eqref{DFDF}-\eqref{DFDF2}-\eqref{C2} and the definition of $F',x_j$ and $s_j$, it is possible to show that $F'$ is the required extension of $F$. We will omit the details.
\end{proof}

\begin{remark}\label{remarkone}
Let us briefly explain how to use Lemma \ref{lem:extension} to extend $F''$ from $\pi_q(S')$ to $\R^2$ as we did in the proof of Theorem \ref{pointrigid}. In particular, let us explain how to construct the curves $\gamma_i$ satisfying \eqref{curve}. Due to the regularity of $S$, and hence of $S'$, the only delicate region to construct the curves is near the vertex $a=\pi_q(\tilde p)$. However, the regularity of $S$ and in particular Lemma \ref{MIERSE} easily provides us with a diffeomorphism $P:T_\delta \to V \cap \tilde T$, where $V$ is a neighborhood of $\pi_q(\tilde p)$, and $P \in C^{1,1}(T_\delta) \cap C^{2,\eta}(T_\delta \setminus \{0\})$ with the usual estimates on $[D^2P]_{C^\eta(T_\delta \setminus B_R(0))}$. Having fixed $x,y \in \tilde T$, let $p = P^{-1}(x), q = P^{-1}(y)$. If $p = re^{i\theta_1}$, $q = se^{i\theta_2}$, $\theta_i \in (-\pi,\pi)$ for all $i$, then the curve $\gamma_1$ of \eqref{curve} is given by the image through $P$ of the curve
\[
\alpha_1(t) = (tr + (1-t)s)e^{i(t\theta_1+(1-t)\theta_2)}, \text{ for } t \in [0,1].
\]
The curve $\gamma_2$ is simply the image through $P$ of the segment connecting $p$ and $q$.
\end{remark}

\begin{lemma}\label{lem:regg}
Let $T$ and $\tilde T$ be curvilinear triangles in the sense of Definition \ref{curvtrian}. Let $x_1,x_2,x_3$ and $y_1,y_2,y_3$ be the vertices of $T$ and $\tilde T$ respectively, and suppose that the angles $\gamma_{i}$ are the same at $x_i$ and $y_i$, for all $i$. Let $g$ be the unique biholomorphism given by the Riemann Mapping theorem between $\intt T$ and $\intt \tilde T$ which maps $x_i$ to $y_i$ for all $i$. Then, there exists $\delta > 0$ such that
\begin{equation}\label{gprop1}
g \in C^{2,\delta}(T\setminus \{x_1,x_2,x_3\},\tilde T), \quad g' \in L^\infty(T,\tilde T),
\end{equation}
and for all $i = 1,2$ and all $z$ sufficiently close to $x_i$
\begin{equation}\label{gprop2}
c|z - x_i| \le |g(z) - g(x_i)| \le C|z-x_i|
\end{equation}
and for all sufficiently small $R > 0$,
\begin{equation}\label{gprop3}
\|g''\|_{L^\infty(T\setminus [{B_R(x_1)\cup B_R(x_2)\cup B_R(x_3)}])} \le \frac{C}{R^{1-\delta}}.
\end{equation}
\end{lemma}
\begin{proof}[(Sketch of) Proof]
Since this result is rather classical and the details are quite lengthy, we will only sketch its proof. We need to handle separately the case $z_0 \in \partial T\setminus\{x_1,x_2,x_3\}$ and the case $z_0 = x_i$ for some $i$. We note that $g$ is a homeomorphism up to the boundary by \cite[Theorem 5.1.1]{KRA}. Let us start with the first case.
\\
\\
\fbox{$z_0 \in \partial T\setminus\{x_1,x_2,x_3\}$:} We set $y_0 = g(z_0)$. We employ the proof suggested in \cite[Problem 2, Section 5]{KRA}. Take a smooth closed Jordan curve $\Gamma$ in $\intt \tilde T$, and define $\Omega \subset \intt \tilde T$ to be the domain bounded by $\partial \tilde T$ and $\Gamma$. Observe that $\Omega$ is a $C^{2,\eta}$ domain at all points of $\partial \Omega = \partial \tilde T \cup \Gamma$ except for $\{y_1,y_2,y_3\}$, due to Definition \ref{curvtrian}\eqref{intx}. Furthermore, since $g$ is biholomorphic from $\intt T$ to $\intt \tilde T$ and $g$ is a homeomorphism that maps vertices into vertices, we have that
$
\Omega' = g^{-1}(\Omega)
$
is again a $C^{2,\eta}$ set except for the points $\{x_1,x_2,x_3\}$. After a rotation, we can suppose that the normal $n(y_0)$ to $ \partial \tilde T$ at $y_0 \in \partial \tilde T \subset \partial \Omega$ is given by $n(y_0) = -e_1 = (-1,0)$. Thus, since $\tilde T$ is a $C^{2,\eta}$ set near $y_0$, we can suppose that $\rho$ is sufficiently small such that 
\begin{equation}\label{ne1}
|n(y)+e_1| \le \eps, \quad \forall y \in \partial \Omega \cap B_\rho(y_0) = (\partial \tilde T)\cap B_\rho(y_0),
\end{equation}
for some $\eps > 0$ to be chosen later. Now we solve the Dirichlet problem
\[
\begin{cases}
\Delta f = 0, \text{ in }\Omega,\\
f = \varphi, \text{ in }\partial \Omega,
\end{cases}
\]
where $\varphi$ is a smooth non-negative function with compact support in $\R^2$ with $\varphi \equiv 0$ on $\Gamma$ and $\varphi \equiv 1$ on $\partial \tilde T$. Since $\Omega$ is a Lipschitz domain due to Definition \ref{curvtrian} and our choice of $\Gamma$, this problem is solvable for some $f \in W^{1,2}(\Omega)$ via classical variational methods, with the boundary datum attained in the sense of traces. We have that $f$ is smooth in the interior and $C^{2,\eta}$ up to the boundary except possibly at the points $\{y_1,y_2,y_3\}$ by classical regularity theory. Furthermore, since $\varphi \equiv 1$ on $\partial \tilde T$ and $\varphi \equiv 0$ on $\Gamma$, the classical maximum principle and Hopf Lemma \cite[Section 6.4.2]{EVAPDE} imply the existence of $c = c(z_0)> 0$ such that:
\begin{equation}\label{fnormal}
|Df(y)| = |(Df(y),n(y))| \ge c, \quad \forall y \in \partial \Omega \cap \partial \tilde T \cap B_\rho(g(z_0)).
\end{equation}
Consider now $\tilde f:= f \circ g$, which, since $g$ fulfills \eqref{A1}, solves
\[
\begin{cases}
\Delta \tilde f = 0, \text{ in }\Omega'\\
\tilde f = 1 \text{ on } \partial T, \quad \tilde f = 0 \text{ on } g^{-1}(\Gamma).
\end{cases}
\]
The same reasoning as above yields that $\tilde f$ is $C^{2,\eta}$ up to the boundary at all points of $T$ except possibly at $\{x_1,x_2,x_3\}$. Thus, for all $z \in \intt T$ near $z_0$ we have for some constant $C = C(z_0) > 0$ and for all $i =1,2$:
\begin{equation}\label{trick}
|(\partial_ig(z),Df\circ g(z))| =|\partial_i(f\circ g)| =  |\partial_i\tilde f|(z) \le C.
\end{equation}
The regularity of $f$ and the fact that $g$ is a homeomorphism up to the boundary allows us to combine \eqref{ne1} and \eqref{fnormal} to deduce that, if $\eps$ and $\rho$ are chosen sufficiently small,
\[
|(\partial_i g(z),e_1)| \le 2C, \quad \forall i =1,2.
\]
Since $g$ solves the Cauchy-Riemann equations in the interior of $T$, the latter is enough to show that $Dg$ stays bounded near $z_0$. Similar computations can be performed for every  higher order derivative, through which we deduce that $g$ is $C^{2,\eta}$ up to the boundary near $z_0$.
\\
\\
\fbox{$z_0 = x_i$ for some $i$:} As above, let $y_0 = g(z_0)$. Upon translating, we can assume $z_0 = 0, y_0= 0$. The corresponding angle is $\gamma_i = \frac{\pi}{\gamma}$, with $\gamma > 1$. The idea is to reduce ourselves to a case analogous to the previous one by \emph{opening up} the sets $T$ and $\tilde T$ near the relevant vertices. Let $F_0: V \to B_1(0)$ and $\tilde F_0: \tilde V \to B_1(0)$ be the diffeomorphisms provided by \eqref{badv} of Definition \ref{curvtrian}. From \eqref{badv} of Definition \ref{curvtrian}, we know that the differentials at $0$ of the maps are positive multiples of rotations. Thus, we rotate the domains in such a way that the differentials at $0$ are positive multiples of the identity. It is convenient to define, for all $\alpha > 0$, the cone
\[
C_\alpha := \left\{R(\cos \theta ,\sin \theta): \theta \in \left[-\frac{\pi}{2\alpha},\frac{\pi}{2\alpha}\right]\right\}.
\]
In particular, we find that, for all $\gamma' < \gamma$, there exists $r > 0$ such that
\begin{equation}\label{asymdom}
T\cap B_r(0) \subset C_{\gamma'}\quad \text{ and } \quad \tilde T\cap B_r(0) \subset C_{\gamma'}.
\end{equation}
Furthermore, in the complex notation $\R^2\sim \mathbb C$, we consider the \emph{opening} maps: for $\alpha>0$ define $f_\alpha := z^\alpha$, for which we choose the following determination. In polar coordinates, if $\theta \in (-\pi,\pi)$, $R \ge 0$,
\[
f_{\alpha}(R\cos\theta+iR\sin\theta) := R^\alpha\cos(\alpha\theta)+iR^\alpha\sin(\alpha\theta).
\]
In particular, $f_\alpha$ is defined in $O := \mathbb C \setminus\{x+iy: x < 0\}$, is holomorphic in $\mathbb C \setminus\{x+iy: x \le 0\}$ and continuous in its domain of definition. Employing \eqref{asymdom}, we choose $r>0$ sufficiently small so that 
$
T \cap B_r(0) \subset V$ and $D_r := f_{\gamma}(T\cap B_{r}(0))\subset O.
$
We also set
$
\tilde D_r := f_{\gamma}(g(T\cap B_r(0))).
$
Again, if $r$ is small enough, we can assume that
$
\tilde D_r \subset \tilde V \text{ and }\tilde D_r \subset O.
$
We \emph{open} the vertices by considering
$
G := f_{\gamma}\circ F_0\circ f_\frac{1}{\gamma} \text{ and } \tilde G := f_{\gamma}\circ \tilde F_0\circ f_\frac{1}{\gamma},
$
defined in $D_r$ and $\tilde D_r$ respectively. Lengthy but direct computations show that, using the regularity properties of $F_0$, $\tilde F_0$ expressed in \eqref{badv} of Definition \ref{curvtrian}, if $r>0$ is chosen sufficiently small, then 
\[
G \in C^{1,\frac{\eta}{\gamma}}(D_r,\R^2) \cap C^{2,\frac{\eta}{\gamma}}(D_r\setminus \{0\},\R^2), \quad \tilde G \in C^{1,\frac{\eta}{\gamma}}(\tilde D_r,\R^2) \cap C^{2,\frac{\eta}{\gamma}}(\tilde D_r\setminus \{0\},\R^2)
\]
and for all $R>0$ sufficiently small,
\begin{equation}\label{Gexp}
\|D^2G\|_{L^\infty(D_r \setminus B_R(0))},\|D^2\tilde G\|_{L^\infty(\tilde D_r \setminus B_R(0))} \le \frac{C}{R^{1-\frac{\eta}{\gamma}}}; \qquad [D^2G]_{C^\frac{\eta}{\gamma}(D_r \setminus B_R(0))}, [D^2\tilde G]_{C^\frac{\eta}{\gamma}(\tilde D_r \setminus B_R(0))} \le \frac{C}{R}.
\end{equation}
Moreover,
\begin{equation}\label{derG}
DG(0) = a^\gamma \id \qquad \mbox{and} \qquad D\tilde G(0) = (\tilde a)^\gamma \id,
\end{equation}
provided $DF_0(0) = a\id$, $D\tilde F_0(0) = \tilde a \id$. Using Lemma \ref{lem:extension}, we can extend $G$ and $\tilde G$ to maps defined in a full neighborhood of $0$ and having the same regularity properties. We will not denote these extensions differently. Finally, due to \eqref{derG}, we may restrict these neighborhoods and still enforce that $G$ and $\tilde G$ are diffeomorphisms between neighborhoods of $0 \in \R^2$. This shows that $D_r$ and $\tilde D_r$ are classical $C^{1,\frac{\eta}{\gamma}}$ sets close to $0$. To conclude, define
$
h := f_{\gamma}\circ g \circ f_{\frac{1}{\gamma}}
$
on $D_r$. Notice that this is still a biholomorphic map from the interior of $D_r$ to the interior of $\tilde D_r$, which is a homeomorphism up to the boundary. By using arguments similar to the ones of the previous step, we can thus show that, for all $R > 0$ sufficiently small,
\begin{equation*}
h \in C^{1,\frac{\eta}{\gamma}}(D_r,\tilde D_r),\quad \Re(h'(0))\neq 0,\quad \text{ and } \quad  \|h''\|_{L^\infty(D_r\setminus B_R(0))} \le \frac{C}{R^{1-\frac{\eta}{\gamma}}},
\end{equation*}
where $\Re$ denotes the real part of a complex number.
Finally, we can transform back the map as
$
g = f_{\frac{1}{\gamma}}\circ h \circ f_\gamma,
$
and again direct computations show that, for some $\delta > 0$, $g \in C^{1,\delta}(T\cap B_r(0))\cap C^{2,\delta}(T\cap B_r(0)\setminus\{0\})$ and $g$ enjoys properties \eqref{gprop2}-\eqref{gprop3}, as wanted. This concludes the proof.
\end{proof}

\section{The double bubble: $k = 2$}\label{main_double}

Throughout this section, we assume that $k = 2$. This corresponds to considering the critical points of the \emph{double bubble} problem.

\begin{Teo}\label{doublethm}
Assume $E_1,E_2$ are convex sets such that $V_{\E}$ for $\E = \{E_1,E_2\}$ satisfies \eqref{doubprob}. Then either $E_1,E_2$ are disjoint (possibly tangent) balls of equal volume or they are the standard double bubble.
\end{Teo}

\begin{proof}
We start by noticing that, as $E_1,E_2$ are convex, then by Hahn-Banach there exists a hyperplane $\pi$ separating the two convex sets. We can, up to rotating and translating, assume that $\pi=e_n^\perp$. Furthermore, $\partial E_1 \cap \partial E_2 \subset \pi$, $E_1 \subset \{x_n>0\}$ and $E_2 \subset \{x_n<0\}$. In particular $\partial E_1 \cap \partial E_2$ is convex. We divide the proof in two cases: $\Haus^{n-1}(\partial E_1\cap \partial E_2) = 0$ and $\Haus^{n-1}(\partial E_1\cap \partial E_2) \neq 0$.

\vspace{0.1cm}
\noindent
\fbox{\emph{Case 1: $\Haus^{n-1}(\partial E_1\cap \partial E_2) = 0$.}} We plug in \eqref{doubprob} any vector field $g$ that coincides with the identity in a neighborhood of $\overline{E_1\cup E_2}$. By the divergence theorem and Lemma \ref{lele}, we compute
\begin{align*}
\lambda_1n|E_1| + \lambda_2n|E_2|&= \sum_{i = 1}^2\lambda_i\int_{\partial^*E_i}(n_{E_i},g)d\mathcal{H}^{n-1}\overset{\eqref{doubprob}}{=} \int_{\partial^* E_1 \cup \partial^* E_2}\langle T_x\Gamma_{\E}, Dg\rangle d\Haus^{n -1}\\
&= \int_{\partial^* E_1 \cup \partial^* E_2}\langle T_x\Gamma_{\E}, \id_n\rangle d\Haus^{n -1}= \int_{\partial^* E_1 }\langle T_x\Gamma_{E_1}, \id_n\rangle d\Haus^{n -1} + \int_{ \partial^* E_2}\langle T_x\Gamma_{E_2},\id_n\rangle d\Haus^{n -1}\\
&= (n-1)\Haus^{n -1}(\partial^* E_1) + (n-1)\Haus^{n -1}(\partial^* E_2),
\end{align*}
where the fourth equality follows from $\mathcal{H}^{n-1}(\partial E_1 \cap \partial E_2)= 0$. In particular, it follows that
\begin{equation}\label{eq12}
\lambda_1|E_1| + \lambda_2|E_2| = \frac{n - 1}{n}\left(\Haus^{n -1}(\partial^* E_1) +\Haus^{n -1}(\partial^* E_2)\right).
\end{equation}
Since $\partial E_1$ and $\partial E_2$ have constant mean curvature in $\{x_n > 0\}$ and $\{x_n < 0\}$ respectively, then by Proposition \ref{reg} they are smooth submanifolds when restricted respectively to $\{x_n > 0\}$ and $\{x_n < 0\}$. 
In particular, denoting by $H_i$ the pointwise mean curvature for the convex set $E_i$ as introduced in \cite[Section 2]{SAN1}, by \eqref{doubprob} we deduce that $H_i = \lambda_i$ holds $(\mathcal{H}^{n-1}\llcorner \partial E_i)$-a.e., for $i = 1,2$. By the Heintze-Karcher inequality \cite[Theorem 1.2]{SAN1}
\begin{align*}
\lambda_1|E_1| + \lambda_2|E_2| &\le \frac{n-1}{n}\lambda_1\int_{\partial^*E_1}\frac{d\Haus^{n-1}}{H_1} +  \frac{n-1}{n}\lambda_2\int_{\partial^*E_2}\frac{d\Haus^{n-1}}{H_2}\\
& = \frac{n-1}{n}\Haus^{n-1}(\partial E_1) +  \frac{n-1}{n}\Haus^{n-1}(\partial E_1).
\end{align*}
The latter, combined with \eqref{eq12}, implies that equality holds in \cite[Theorem 1.2]{SAN1}, and by the rigidity part of \cite[Theorem 1.2]{SAN1} applied separately to $E_1$ and $E_2$, we find that $E_1$ and $E_2$ must be balls. Since they are disjoint, either $\overline{E_1}$ and $\overline{E_2}$ are disjoint or they touch tangentially. This concludes the proof of this case.

\vspace{0.1cm}
\noindent
\fbox{\emph{Case 2: $\Haus^{n-1}(\partial E_1\cap \partial E_2) \neq 0$.}}
We make a list of claims and show how these imply that $V_\E$ must be the standard double bubble. We defer the proof of these claims to the end of the section. First, we claim that
\begin{equation}\label{claim1}\tag{Claim 1}
\partial E_1 \cap \pi = \partial E_2 \cap \pi.
\end{equation}
If \eqref{claim1} holds, we can apply Corollary \ref{corstruct}. Using this, we claim that at every point $x \in \partial_\pi(\overline{E_1}\cap \pi)$,
\begin{equation}\label{claim2}\tag{Claim 2}
\text{the blow-up of }V_\E \text{ at $x$ is a $Y$ cone containing a half hyperplane contained in $\pi - x$.}
\end{equation}
and
\begin{equation}\label{claim3}\tag{Claim 3}
\partial E_i \cap \pi \text{ is a $C^\infty$ domain in $\pi$ and $\overline{\partial E_i \setminus \pi}$ is a smooth manifold with boundary for every $i=1,2$}.
\end{equation}

From \eqref{claim2} and Corollary \ref{corstruct} we infer that the blow-up $W_x$ of $V_\E$ at $x \in \partial_\pi(\overline{E_1}\cap \pi)$ is a $Y$ cone:
\[
W_x =  \llbracket K\rrbracket + \llbracket K_1\rrbracket + \llbracket K_2\rrbracket,
\]
where $K$, $K_1$, $K_2$ are half hyperplanes intersecting at the $(n - 2)$-dimensional tangent plane $T_x(\partial_\pi(\overline{E_1}\cap \pi))$, with $K_1 \subset \{x_n \ge 0\}$ and $K_2 \subset \{x_n \le 0\}$. Since $\llbracket K\cup K_i\rrbracket$ is the blow-up of $\llbracket \partial E_i\rrbracket$ at $x$, we infer that $\partial E_i$ forms an angle of $120$ degrees with $\pi$ at every $x \in \partial_\pi(\overline{E_i}\cap \pi)$, for all $i = 1,2$. This, combined with the fact that $\partial E_1\cap \{x_n > 0\}$ and $\partial E_2\cap \{x_n < 0\}$ have constant mean curvature and the regularity of $\partial E_i\cap \pi$ of \eqref{claim3}, allows us to apply Alexandrov moving plane method \cite[Theorem 4.1.16]{LOP}. We remark that, although \cite[Theorem 4.1.16]{LOP} is stated in $\R^3$, its proof generalizes to $\R^n$ with no major changes. We deduce that $E_1$ and $E_2$ are balls intersected respectively with $\{x_n > 0\}$  and $\{x_n < 0\}$. Since $\overline{E_1}\cap \pi=\overline{E_2}\cap \pi$ and $\partial E_i$ forms an angle of $120$ degrees with $\pi$, then we deduce that  $V_\E$ is the standard double bubble.
\end{proof}

\subsubsection{Proof of \ref{claim1}:}
By contradiction, suppose \eqref{claim1} does not hold. Since $\partial E_1 \cap \pi$ is a convex set and $\Haus^{n-1}(\partial E_1\cap \pi) \neq 0$, then it has non-empty interior in $\pi$. Then, without loss of generality, we can assume that there exists a point $x \in \intt_\pi(\partial E_1\cap \pi)$ such that $x \notin \partial E_2$. This implies that there exists a ball $B^\pi_\delta(x)$ such that $B^\pi_\delta(x)\subset \partial E_1\setminus \partial E_2$. Hence, by Corollary \ref{corstat}, $\llbracket\partial{E_1}\setminus \partial E_2\rrbracket$ is a varifold with zero mean curvature. As $\partial E_1$ is the boundary of a convex set, Corollary \ref{statvarif} shows that each connected component $T$ of $\partial{E_1}\setminus \pi$ is contained in a hyperplane $\pi_T$. We observe that $\partial{E_1}\setminus \pi$ has only one connected component: indeed, if we fix two points $x,y\in \partial{E_1}$, and the segment $L_{x,y}\subset \overline{E_1}$ joining them, it is enough to radially project $L_{x,y}$ into $\partial{E_1}\setminus \pi$ from any point $z \in \partial{E_1}\cap \pi$, to obtain a path connecting $x$ and $y$ in $\partial{E_1}\setminus \pi$. This in turn implies that $\partial{E_1}$ is contained in the union of two hyperplanes, which contradicts the fact that $E_1$ is an open bounded nonempty set.

\subsubsection{Proof of \ref{claim2}:}\label{proofofclaim2} We assume without loss of generality that $x = 0$. By Corollary \ref{corstruct}, the blow-up $W_x$ of $V_\E$ at $x \in \partial_\pi(\overline{E_1}\cap \pi)$ is unique and can be obtained as
 \[
W_x = \llbracket K\rrbracket + \llbracket K_1\rrbracket + \llbracket K_2\rrbracket.
\]
Since $V_\E$ is stationary for the double bubble problem, the varifold $W_x$ must be stationary for the area functional, see Corollary \ref{blowcone}.
 We wish to say that $K$, $K_1$, $K_2$ are three half hyperplanes intersecting at $120$ degrees on an $(n-2)$-dimensional plane $L$ contained in $\pi$ and passing through to $0$. We need to consider different cases, depending on whether $K_i$ is contained in $\pi$ or not.
\\
\\
\fbox{$K_1$ and $K_2$ are contained in $\pi$.} 
By Proposition \ref{uniquecone} and Corollary \ref{corstruct} we deduce that $K\cup K_1=\pi $ and $K\cup K_2=\pi $. In particular $K_1=K_2$. It follows that $W_x=2\mathcal{H}^{n-1}\llcorner (\pi \setminus K) +\mathcal{H}^{n-1}\llcorner K$. This is in contradiction with the constancy theorem for stationary varifolds, see \cite[Theorem 8.4.1]{SIM}. Hence this case is not possible.
\\
\\
\fbox{Only one between $K_1$ and $K_2$ is contained in $\pi$.} Without loss of generality, we can suppose $K_1 \subset \pi$. By Proposition \ref{uniquecone} and Corollary \ref{corstruct} we deduce that $K\cup K_1=\pi $. We first observe that $K_2\cap \{x_n < 0\}$ is connected, otherwise $K_2\cup K$ could not be the boundary of a convex cone. By Corollary \ref{corstruct}, we have that $K_2$ is a graph of a convex function $u$ over a subset of a hyperplane $\pi'$. Since $W_x\llcorner \{x_n < 0\} = \llbracket K_2\cap \{x_n < 0\}\rrbracket$ and $W_x$ is stationary, we then find that $\llbracket K_2\rrbracket$ is stationary in $\{x_n < 0\}$. By Proposition \ref{rigconv} and the connectedness of $K_2\cap \{x_n < 0\}$, it follows that $u$ is affine, and hence $K_2$ is subset of a hyperplane. Since  $K_i\cap K = \partial_\pi K$, for $i=1,2$, see Corollary \ref{corstruct}, it then follows that $K,K_1$ an $K_2$ are all half hyperplanes meeting at a common $(n-2)$-dimensional plane. Denoting with $T,T_i$ the tangent hyperplanes to $K$ and $K_i$ respectively, then the stationarity of $W_x$ yields
\begin{equation}\label{zeroeq}
\begin{split}
0 =  \left\langle T,\int_{K}Dg(x) d\Haus^{n-1}(x)\right\rangle + \left\langle T_1,\int_{K_1}Dg(x) d\Haus^{n-1}(x)\right\rangle + \left\langle T_2,\int_{K_2}Dg(x) d\Haus^{n-1}(x)\right\rangle,
\end{split}
\end{equation}
for all $g \in C^\infty(\R^n,\R^n)$. If $T = \id_n - v\otimes v$, $T_i = \id_n - v_i\otimes v_i$ with $|v|,|v_i| = 1$ for $i = 1,2$, one can see that \eqref{zeroeq} is equivalent to $v + v_1 + v_2 = 0$, up to changing orientation to the vectors $v,v_1,v_2$. The only solution to this equation among unit vectors is precisely given by triples of vectors lying on a two-dimensional plane that pairwise form an angle of $120$ degrees. Since $K\cup K_1=\pi$, though, their normals $v$ and $v_1$ must coincide (up to a change of orientation). Hence we discard also this case.
\\
\\
\fbox{$K_1$ and $K_2$ are not contained in $\pi$.} Arguing as in the previous case, it follows that $K_1$ and $K_2$ are half hyperplanes meeting at a common $(n-2)$-dimensional plane, and hence also $K$ is a half hyperplane contained in $\pi$ intersecting the same $(n-2)$-dimensional plane. As in \eqref{zeroeq}, by stationarity we infer that $W_x$ is a $Y$ cone.

\subsubsection{Proof of \ref{claim3}}\label{subsubreg}
By Corollary \ref{corstruct}, $K$ coincides with the blow-up at $x$ of $\overline{E_1}\cap \pi$. Since $K$ must be a half hyperplane (as proved above), it follows that the blow-up at any $x \in \partial_\pi(\overline{E_1}\cap \pi)$ is an $(n-2)$-dimensional plane. From Corollary \ref{regularconv} we deduce that $\partial E_i \cap \pi \text{ is a $C^1$ domain in $\pi$}$. In order to prove \eqref{claim3}, let us show how to upgrade this regularity to $C^\infty$. To this aim, we fix $x \in \partial E_1\cap \pi$. Since $W_x$ is a $Y$ cone and hence $\partial E_i$ intersects the plane $\pi$ at an angle of 120 degrees, we can parametrize $\partial E_i$ over $\pi$ in the sense of Proposition \ref{graph} in a small neighborhood of $x$. Thus, we find $\delta > 0$ and a convex, Lipschitz, non-negative function $u: B = B_\delta^{\pi}(x) \to \R$ that fulfills:
\begin{equation}\label{sys:reg}
\begin{cases}
u = 0, &\text{ on }\Omega_0 := \partial E_1 \cap B,\\
\dv(D\A(Du)) = \lambda, &\text{ in }\Omega := B \setminus \partial E_1,\\
(D\A(Du),n) = \cos\left(\frac{\pi}{6}\right) = \frac{\sqrt{3}}{2}, &\text{ on } \Gamma :=\partial_{\pi}(\partial E_1 \cap \pi)\cap B,
\end{cases}
\end{equation}
where $n$ is the unit outer normal of $\Omega_0$, which is defined at every point of $\Gamma$ since $\Gamma$ is $C^1$. We choose $\delta > 0$ sufficiently small so that $\Omega$ is connected. To justify the second line of \eqref{sys:reg} we can employ Corollary \ref{corstat} and Proposition \ref{varigraph}, while the third line can be deduced from \eqref{claim2}. Although $u$ is just Lipschitz in $B$, we can compute $Du$ at every point of $\Gamma$ by using directional derivatives along the directions entirely contained in $\Omega$. The latter exist by our blow-up analysis and Proposition \ref{uniquecone}. We wish now to show that $u$ is actually $C^{1,1}$ up to $\Gamma$. Notice that by Proposition \ref{reg}, $u$ is smooth in $\Omega$. Thus we can consider the pointwise Hessian of $u$. We can rewrite the PDE of \eqref{sys:reg} in its strong form
\[
\langle D^2\A(Du(y)), D^2u(y)\rangle = \lambda.
\]
As $Du \in L^\infty(\Omega)$, we see that there exists a constant $\rho > 0$ such that $\rho \id \le D^2\A(Du(y))$ for all $y \in \Omega$ in the sense of quadratic forms. Since $u$ is convex, $D^2u(y)$ is non-negative definite in $\Omega$, and thus there exists $C= C(\rho) > 0$ such that
\[
C|D^2u(y)| \le \langle D^2\A(Du(y)), D^2u(y)\rangle = \lambda.
\]
Therefore, $D^2u(y)$ is uniformly bounded in $\Omega$. Since $\Omega$ is a $C^1$ domain in $B$, we conclude that $u \in C^{1,1}(\overline \Omega)$.
In particular, $\Gamma$ is a $C^{1,1}$ $(n-2)$-dimensional manifold and $\Omega$ is a $C^{1,1}$ domain in $B$. Indeed, since $u = 0$ on $\Gamma$ and $\Gamma$ is a $C^1$ manifold, we have that its unit normal is given by
\begin{equation}\label{normal}
n = \frac{Du}{|Du|}, \qquad \mbox{at all points of $\Gamma$ where $|Du|\neq 0$.}
\end{equation}
By the third line of \eqref{sys:reg}, we deduce that $|Du|$ is bounded from below in a neighborhood of $\Gamma$, thus \eqref{normal} implies that $\Gamma$ is as regular as $u$ is. The procedure to conclude that $u$ is actually smooth in $\overline{\Omega}$ and that $\Gamma$ is a smooth manifold is well-known and employs the hodograph transform and the classical boundary regularity of \cite{ADN}. For details, see \cite[Corollary 5.6]{MNE} and \cite[Theorem 5.2]{KNS}. Thus, $\partial E_i \cap \pi \text{ is a $C^\infty$ domain in $\pi$}$ and $\overline{\partial E_i \setminus \pi}$ is a smooth manifold with boundary.

\section{The triple bubble: $k = 3$}\label{sec_3}

Throughout this section, we assume that $k = 3$ and $n = 3$. This corresponds to considering the critical points of the \emph{triple bubble} problem in $\R^3$. For convenience, let us introduce the following terminology. Let $C_1,C_2$ be two disjoint, nonempty, open and bounded convex sets. We say that $C_1$ \emph{interacts with } $C_2$ if $\Haus^{2}(\partial C_1\cap \partial C_2) \neq 0$. Given three pairwise disjoint open convex sets, $E_1$, $E_2$, $E_3$ we only have four possibilities, up to relabeling of the indices:

\begin{enumerate}[(\text{Case} 1)]
\item None of the sets interacts with the others;
\item $E_2$ interacts with $E_1$, and $E_3$ does not interact with $E_1$ and $E_2$;
\item $E_1$ interacts with both $E_2$ and $E_3$, but $E_2$ and $E_3$ do not interact with each other;
\item $E_i$ interacts with $E_j$, for all $1\le i \neq j\le 3$.
\end{enumerate}

 As in the proof of Theorem \ref{doublethm}, we apply Hahn-Banach Theorem to find planes $\pi_{ij}$ separating $E_i$ from $E_j$, given $1\le i \neq j \le 3$. In particular, $\partial E_i\cap \partial E_j $ is a convex set contained in $\pi_{ij}$.

\begin{Def}\label{LU}
We say that $E_1,E_2, E_3 \subset \R^3$ form a chain triple bubble if, up to a relabeling of the indices, $E_i$ is an open ball intersected with an open half space for all $i \in \{2,3\}$,  $\partial E_1 \cap \partial E_2$ and $\partial E_1\cap \partial E_3$ are disks contained in two distinct planes $\pi_{12}$ and $\pi_{13}$ respectively. If $\pi_{12}$ is not parallel to $\pi_{13}$, then $E_1$ is an open ball intersected with a connected component of the open set between $\pi_{12}$ and $\pi_{13}$; while if $\pi_{12}$ is parallel to $\pi_{13}$, then $\partial E_1$ is a surface of revolution intersected with the open set between $\pi_{12}$ and $\pi_{13}$. Furthermore, $\partial E_2$ intersects $\pi_{12}$ forming an angle of $120$ degrees,  $\partial E_1$ intersects $\pi_{12}$ and $\pi_{13}$ forming an angle of $120$ degrees and $\partial E_3$ intersects $\pi_{13}$ forming an angle of $120$ degrees.
\end{Def}

\begin{Def}\label{STTB}
We say that $E_1,E_2, E_3 \subset \R^3$ form a standard triple bubble if $E_1$, $E_2$ and $E_3$ are open balls with equal radii intersected with wedges generated respectively by the couples of planes $(\pi_{13},\pi_{12})$, $(\pi_{23},\pi_{12})$ and $(\pi_{13},\pi_{23})$. Furthermore, the planes intersect along a common line $\ell$ at an angle of $120$ degrees, $\partial E_i \cap \ell = \partial E_1 \cap \ell \neq \emptyset, \forall i =1,2,3$, and $E_i$ is equal to $E_{i+1}$ up to a rotation of $120$ degrees with axis $\ell$.
\end{Def}


\begin{Teo}\label{triplethm}
Assume $E_1,E_2,E_3 \subset \R^3$ are convex sets such that $V_{\E}$ for $\E = \{E_1,E_2,E_3\}$ satisfies \eqref{doubprob}. Then only one of the following configurations holds (up to relabeling indices):
\begin{enumerate}[(\text{Configuration} 1)]
\item $E_1,E_2,E_3$ are disjoint, possibly tangent, balls;
\item $E_1,E_2$ form a standard double bubble, while $E_3$ is a ball, possibly tangent to $E_1$ or $E_2$;
\item $E_1,E_2,E_3$ form a chain triple bubble;
\item $E_1,E_2,E_3$ form a standard triple bubble.
\end{enumerate}
In particular, (Configuration $i$) happens in (Case $i$) and only in (Case $i$), for every $i\in \{1,2,3,4\}$.
\end{Teo}

\subsection{Proof of (Case 1) if and only if (Configuration 1).} In this case, the proof is the same as in Case 1 of Theorem \ref{doublethm}, i.e. plugging into \eqref{doubprob} $g = \id$ and then applying \cite[Theorem 1.2]{SAN1} separately on $E_1$, $E_2$ and $E_3$. This gives us precisely (Configuration 1).

\subsection{Proof of (Case 2) if and only if (Configuration 2).} We first show that $E_3$ must be a ball. 

\subsubsection{$E_3$ is a ball.}\label{aball} By Corollary \ref{corstat}, we know that $\llbracket \partial E_3 \setminus (\partial E_1 \cup \partial E_2)\rrbracket$ is a varifold with constant mean curvature. Our aim is to show that $\llbracket\partial E_3\rrbracket$ is a varifold with constant mean curvature. By \cite[Theorem 1]{DEM}, this would imply that $E_3$ is a ball. By assumption, $\Haus^{2}(\partial E_3\cap \partial E_1) = \Haus^{2}(\partial E_3\cap \partial E_2) = 0$, and hence we aim to show that $E_3$ is also of constant mean curvature \emph{across} $\partial E_3\cap \partial E_1$ and $\partial E_3 \cap \partial E_2$. We can suppose that either $\partial E_3\cap \partial E_1 \neq \emptyset$ or $\partial E_3\cap \partial E_2 \neq \emptyset$, as otherwise Corollary \ref{corstat} concludes the proof. Without loss of generality, let
$
\partial E_3\cap \partial E_1 \neq \emptyset.
$
\\
\\
\fbox{\emph{Step 1: $\partial E_3$ is $C^1$.}}
By Corollary \ref{regularconv}, we just need to show that the blow-up at every point of $\partial E_3$ is a plane. We have $\partial E_3 \cap \partial E_1 \subset \pi_{13}$. Since $\Haus^{2}(\partial E_3\cap \partial E_1) =0$, by convexity we can find a line $L \subset \pi_{13}$ such that $\partial E_3 \cap \partial E_1\subset L$. Consider a point $x \in \partial_{L}(\partial E_3 \cap \partial E_1)$. We translate and rotate to have that $\pi_{13} = \{y \in \R^3: y_3 = 0\}$ and $x=0$. By Proposition \ref{graph}, we can write $\partial E_3$ as the graph of a convex function $\varphi$ defined on $B_\delta^{\pi}(x)$, where $\pi$ is a supporting plane of $\partial E_3$ at $x$. We use Proposition \ref{uniquecone} to find the unique blow-up cone $\llbracket K\rrbracket$ of $\llbracket \partial E_3 \rrbracket $ at $x$, that is the graph of some positively one-homogeneous convex function $H$ defined on $\pi$. Since $\partial E_3$ has constant mean curvature $\lambda$ outside of $\partial E_3\cap \partial E_1$, by Proposition $\ref{varigraph}$
\[
\dv(D\A(D\varphi)) = \lambda,\quad \text{in the sense of distributions on $B_\delta^\pi(x)\setminus p_\pi(\partial E_3 \cap \partial E_1)$}.
\]
By the latter and the fact that $x \in \partial_{L}(\partial E_3 \cap \partial E_1)$, we can exploit the definition of $H$ in \eqref{homogeneous} to see that there exists a half line $\ell \subset \pi $ such that
\[
\dv(D\A(DH)) = 0,\quad \text{in the sense of distributions on $\pi \setminus \ell$}.
\]
From Proposition $\ref{rigconv}$, we find that $H$ is affine on $\pi\setminus \ell$. By continuity, we infer that $H$ is affine on $\pi$. This shows that $\llbracket K \rrbracket$ is a plane. Since $\pi_{13} = \{y \in \R^3: y_3 = 0\}$ separates $E_3$ from $E_1$, we have that without loss of generality $z_3 > 0$ for every $z = (z_1,z_2, z_3) \in E_3$. This imposes that $z_3 \ge 0$ for every $z \in K$ and, since $K$ is a plane, $K= \pi_{13}$. Up to now, we have shown that the blow-up of $\partial E_3$ is $\pi_{13}$ at every $x \in \partial_L(\partial E_1\cap\partial E_3)$, and in particular this settles the proof of the claim in case $\partial E_1 \cap \partial E_3$ consists of a single point. Hence we can assume there exists $x \in \intt_L(\partial E_1\cap \partial E_3)$. Write $x = s x_1 + (1-s) x_2$, with $x_1,x_2 \in \partial_L(\partial E_1\cap\partial E_3)$. By the first part of the proof, the blow-up at $x_1$ and $x_2$ of $\partial E_3$ is the plane $\pi_{13} = \{y \in \R^3: y_3 = 0\}$. By the uniform convergence of Proposition \ref{uniquecone}, we infer that for all $\alpha > 0$, there exists $\beta> 0$ such that
\begin{equation}\label{closedness}
\partial E_3\cap B_\beta(x_i) \subset \{z \in \R^3: 0\le z_3 \le \alpha\beta\}, \quad\forall i =1,2.
\end{equation}
Suppose by contradiction that the blow-up $\llbracket K_x\rrbracket$ at $x$ is not a plane. Since $x \in \intt_L(\partial E_1\cap \partial E_3)$, then $(L-x) \subset K_x$. Up to a rotation inside $\pi_{13}$, we denote $(L-x)=\{y \in \R^3:y_2=0,y_3=0\}$. Since by assumption $\llbracket K_x\rrbracket$ is not a plane and it is a cone and the graph of a convex function over a supporting plane of $E_3$, then without loss of generality there exists $\theta >0$ such that $y_3=\theta y_2$ for every $y=(y_1,y_2,y_3)\in K_x$ such that $y_1=0$ and $y_2>0$. 
%
%
%
%
 Again by the uniform convergence of Proposition \ref{uniquecone}, we infer that there exists $\rho > 0$ such that
\begin{equation}\label{too}
\partial E_3 \cap \{y \in \R^3:y_1=0, y_2>0\} \cap B_\rho(x) \subset \left\{y \in \R^3: 0 < \frac{\theta}{2}\rho \le y_3 \le 2\theta \rho\right\}.
\end{equation}
We now choose $\alpha = \frac{\theta}{4}$ in \eqref{closedness}, to find a corresponding $0 < \beta \le \rho$. Since $E_3$ is convex, there exists $y \in \co(\partial E_3\cap (B_\beta(x_1)\cup B_\beta(x_2)))\cap B_\rho(x)$ such that
\[
y_1=0, \quad y_2>0, \quad 0 \le y_3 \le \beta\frac{\theta}{4} \le \rho\frac{\theta}{4}.
\]
In particular, there exists $z \in \partial E_3 \cap \{y \in \R^3:y_1=0, y_2>0\} \cap B_\rho(x)$ satisfying $z_3\in \left[0,\frac{\theta}{4}\rho\right]$, that is in contradiction with \eqref{too}.
\\
\\
\fbox{\emph{Step 2: Conclusion of the proof.}} We claim that for every $x\in\partial  E_3$, there exists a neighborhood of $x$ such that in that neighborhood $\partial  E_3$ has constant mean curvature $\lambda$ in the sense of varifolds. If $x\in \partial E_3 \setminus (\partial E_1 \cup \partial E_2)$, the claim follows by Corollary \ref{corstat}. Let $x \in \partial E_3 \cap \partial E_1$, the case $x \in \partial E_3\cap \partial E_2$ being analogous. Then, by Corollary \ref{regularconv}, we can write $\partial E_3$ as the graph over its unique tangent plane $\pi=\pi_{13}$ of a convex $C^1$ function $\varphi$ defined on $B^\pi_\delta(x)$ for some small $\delta > 0$. By our assumption $\mathcal{H}^{2}(\partial E_3 \cap \partial E_1) = 0$, we infer the existence of a line $L \subset \R^3$ such that $\partial E_3 \cap \partial E_1 \subset L$. By Proposition \ref{varigraph} we deduce that $\varphi$ satisfies
\begin{equation}\label{-S}
\dv(D\A(D\varphi)) = \lambda,\quad \text{in the sense of distributions on $B^\pi_\delta(x)\setminus L$}.
\end{equation}
Since $\varphi \in C^1$, by Lemma \ref{div} we deduce that the same holds in $B_\delta^\pi(x)$. We infer that the mean curvature of $\partial E_3$ in a neighborhood of $x$ is $\lambda$ again by Proposition \ref{varigraph}. Therefore, for every $x\in\partial  E_3$ there exists a neighborhood of $x$ such that in that neighborhood $\llbracket \partial  E_3\rrbracket$ has constant mean curvature $\lambda$ in the sense of varifolds. It readily follows that $\llbracket \partial  E_3\rrbracket$ has constant mean curvature $\lambda$ in the sense of varifolds. Applying \cite[Theorem 1.2]{SAN1} or \cite[Theorem 1]{DEM} we deduce that $E_3$ is a ball.

\subsubsection{Conclusion}

We just proved that $E_3$ is an open ball. We wish to prove that $\E' = \{E_1,E_2\}$ is such that 
\begin{equation}\label{conc32}
V_{\E'} \text{ is stationary for the double bubble problem,}
\end{equation}
i.e. \eqref{doubprob} holds for $k = 2$. An application of Theorem \ref{doublethm} will then conclude the proof, as it will tell us that we are exactly in the situation described by (Configuration 2). If 
$
\partial E_3 \cap \partial E_i = \emptyset \text{ for every } i =1,2,
$
we immediately conclude \eqref{conc32}. Hence we can suppose, without loss of generality, that $\partial E_3 \cap \partial E_1 \neq \emptyset$. Since an Euclidean sphere does not contain any segment, by convexity of $E_1,E_2,E_3$ it follows that $\partial E_3 \cap \partial E_1 = \{x_1\}$ and $\partial E_3\cap \partial E_2= \{x_2\}$, if it is nonempty. Notice that, as in \eqref{claim1}, we find that $\partial E_1\cap \pi_{12} = \partial E_2 \cap \pi_{12}$, and this is a $2$-dimensional convex set inside $\pi_{12}$ by assumption. If 
$
\partial E_3 \cap \partial E_1 = \{x_1\} \text{ and } x_1 \notin \pi_{12},
$
then we also have that $x_1 \not \in \partial E_3\cap \partial E_2$. In this situation, we can use a procedure similar to the one of Section \ref{aball} to show that $\partial E_1$ is $C^1$ across $x_1$ and that $\partial E_1$ has constant mean curvature around $x_1$. The same would hold for $E_2$ provided $\partial E_3\cap \partial E_2 = \{x_2\}$. This yields \eqref{conc32}. 
We conclude the proof showing by a blow-up analysis that 
\begin{equation}\label{utiles}
x_1\not\in \pi_{12} \mbox{ and }x_2\not\in \pi_{12}.
\end{equation}
Assume by contradiction, without loss of generality, that $x_1\in \pi_{12}$. Since $\partial E_1\cap \pi_{12} = \partial E_2 \cap \pi_{12}$, then $x_1 \in \partial E_2$. Hence the tangent plane $\pi$ to $\partial E_3$ at $x_1$ coincides with the separating hyperplanes $\pi_{13}$ and $\pi_{23}$. After a rigid motion, we suppose that $\pi = \{y \in \R^3: y_3 = 0\}$, that $E_3 \subset \{y \in \R^3: y_3 > 0\}$, and that $x_1=0$. Notice that $E_1$ and $E_2$ are in the same configuration as in Corollary \ref{corstruct}, with the additional sphere $\partial E_3$ attached to the point $x_1$ of blow-up. Hence, it is not difficult to see that the blow-up of $V_\E$ at $x_1$ is given by four pieces:
\begin{equation}\label{conocontr}
V_{x_1} = \llbracket K\rrbracket + \llbracket K_1\rrbracket + \llbracket K_2\rrbracket + \llbracket \pi\rrbracket,
\end{equation}
which are respectively the blow-ups at $x_1$ of $\llbracket \partial E_1 \cap \pi_{12}\rrbracket=\llbracket \partial E_2 \cap \pi_{12}\rrbracket$, $\llbracket \partial E_1 \setminus \pi_{12}\rrbracket$, $\llbracket \partial E_2 \setminus \pi_{12}\rrbracket$ and $\llbracket \partial E_3\rrbracket$. By Corollary \ref{blowcone}, $V_{x_1}$ must be stationary. However, this is not possible. Indeed, we would have that $\spt(\|V_{x_1}\|) \subset \{y \in \R^3: y_3 \ge 0\}$, and by the Maximum Principle \cite[Theorem 1.1]{WICK}, we would obtain that $\spt(\|V_{x_1}\|) = \pi$, which contradicts Corollary \ref{corstruct}. Hence we conclude the validity of \eqref{utiles}.

\subsection{Proof of (Case 3) if and only if (Configuration 3).} 
The proof is divided in subcases. 

\subsubsection{ (Subcase 1): $\pi_{12}$ and $\pi_{13}$ are parallel planes.} \label{coincid:pla}
If $\pi_{12} = \pi_{13}$, there exists $x\in \intt_{\pi_{12}}(\partial E_1\cap \pi_{12}) \cap \partial E_2\cap\partial E_3$, otherwise by convexity $\partial E_1 \setminus (\partial E_2\cup \partial E_3)$ would contain a disk, which provides the same contradiction as in the proof of \eqref{claim1}. On the other hand, $x\in \intt_{\pi_{12}}(\partial E_1\cap \pi_{12}) \cap \partial E_2\cap\partial E_3$ gives the same contradiction as for \eqref{conocontr}, considering the blow-up of $V_\E$ at $x$. 
Therefore $\pi_{12} \neq \pi_{13}$ cut $\R^3$ into three open disjoint sectors $\Sigma_j$, for $j = 1,2,3$, with 
\[
\partial \Sigma_2 = \pi_{12},\quad \partial \Sigma_1 = \pi_{12}\cup\pi_{13},\quad \partial \Sigma_3 = \pi_{13}.
\]
Moreover, we have, for $i = 2,3$:
\begin{enumerate}
\item\label{cl13} $\partial E_i\cap \partial E_1 = \partial E_i \cap \pi_{1i} = \partial E_1 \cap \pi_{1i}$;
\item\label{cl23} $E_i$ is a ball intersected with $\Sigma_i$ with $\partial E_i$ intersecting $\pi_{1i}$ with constant angle of $120$ degrees;
\item\label{cl33} $\partial E_1$ intersects $\pi_{1i}$ with constant angle of $120$ degrees;
\item\label{cl34} $\partial E_1 \cap \pi_{1i}$ is a $C^\infty$ domain in $\pi_{1i}$ and $\overline{\partial E_1 \setminus \pi_{1i}}$ is a smooth manifold with boundary.
\end{enumerate}
\eqref{cl13} can be proved in the same way as \eqref{claim1} of Theorem \ref{doublethm}. To show \eqref{cl23} and \eqref{cl33}, one first uses Corollary \ref{corstruct} to study the blow-ups at points of $\partial_{\pi_{1i}}(\partial E_i\cap \partial E_1)$, and then deduces that the angles must be of 120 degrees, exactly as in \eqref{claim2} of Theorem \ref{doublethm}. One can prove \eqref{cl34} of the above list as in \eqref{claim3} of Theorem \ref{doublethm}. Finally, using \cite[Theorem 4.1.16]{LOP}, we find that $E_i$ is a ball intersected with $\Sigma_i$, for $i = 2,3$. Now an application of Alexandrov's method of moving planes shows that $E_1$ is axially symmetric, see for instance \cite{MMRM}. This yields that the two disks $\partial E_1 \cap \pi_{12}$ and $\partial E_1 \cap \pi_{13}$ are coaxial.  Delaunay \cite{De} proved that the only surfaces of revolution with constant mean curvature are the plane, the cylinder, the sphere, the catenoid, the unduloid and the nodoid. The requirement that $E_1$ is a convex set and $\partial E_1$ intersects the two disks at a constant angle of $120$ degrees implies that $\partial E_1\cap \Sigma_1$ is either a sphere or an unduloid or a nodoid intersected with $\Sigma_1$. This shows that $E_1,E_2,E_3$ are a chain triple bubble in the sense of Definition \ref{LU}, as in (Configuration 3).

\subsubsection{ (Subcase 2): $\pi_{12}\cap \pi_{13}=L$ is a line.}\label{sub2}

Notice that $\pi_{12}$ and $\pi_{13}$ divide $\R^3$ in four open convex wedges, that will be denoted by $\Sigma_j$, for $j =1,2,3,4$. Upon relabeling, we can assume that
\[
E_1 \subset \Sigma_1,\quad E_2 \subset \intt(\overline{\Sigma_2\cup\Sigma_4}), \quad E_3 \subset \intt(\overline{\Sigma_3\cup\Sigma_4}).
\]
We start by the following usual observation, see \eqref{claim1},
\begin{equation}\label{intersec}
\partial E_2 \cap \pi_{12}  = \partial E_1 \cap \pi_{12} \quad  \text{ and } \quad \partial E_3 \cap \pi_{13}  = \partial E_1 \cap \pi_{13},
\end{equation}
from which it follows
\begin{equation}\label{intersecL}
\partial E_i \cap L  = \partial E_1 \cap L, \quad \forall i =2,3.
\end{equation}
Now, if
\begin{equation}\label{case331}
\partial E_1 \cap L =\partial E_2 \cap L = \partial E_3 \cap L = \emptyset,
\end{equation}
with the same arguments of (Case 2) one can see that \eqref{cl13}-\eqref{cl23}-\eqref{cl33}-\eqref{cl34} of (Subcase 1) hold in (Subcase 2) as well. Hence, in order to prove that $E_1,E_2,E_3$ form a chain triple bubble, we are just left to show that $E_1$ is a ball intersected $\Sigma_1$. This is anyway an immediate consequence of \cite[Theorem 1]{PRT}. Indeed, by Corollary \ref{corstat}, $\partial E_1 \cap \Sigma_1$ is a surface with constant mean curvature that, by \eqref{cl33} of the list of (Subcase 1), intersects $\pi_{12}$ and $\pi_{13}$ with constant angle. Moreover, Proposition \ref{RT} implies that $\overline{\partial E_1 \setminus (\pi_{12} \cup \pi_{13})}$ is of ring-type in the sense of Definition \ref{def:DTRT}. Thus $\partial E_1$ is \emph{a ring-type spanner in a wedge}, in the terminology of \cite{PRT}. Therefore $E_1$ is a ball intersected with $\Sigma_1$. Thus, if \eqref{case331} holds, we see that $E_1,E_2,E_3$ are in (Configuration 3). 

To conclude the proof of (Subcase 2), we shall consider the case in which \eqref{case331} does not hold, i.e. by \eqref{intersecL}:
\begin{equation}\label{case332}
\partial E_1 \cap L = \partial E_2\cap L = \partial E_3\cap L \neq \emptyset.
\end{equation}
First, we observe that \eqref{cl33}-\eqref{cl34} of (Subcase 1) still hold in this case, at least in $\pi_{1i}\setminus L$ for $i =2,3$. Since $E_i$ is convex for each $i$ and $L$ is a line, there are only two possibilities: either there exists $x\in L$ such that
\begin{equation}\label{case3321}
\partial E_1 \cap L = \partial E_2\cap L = \partial E_3\cap L = \{x\}.
\end{equation}
or there exists $x_1,x_2\in L$ such that
\begin{equation}\label{case3322}
\partial E_1 \cap L = \partial E_2\cap L = \partial E_3\cap L = [x_1,x_2].
\end{equation}
We wish to show that if \eqref{case3321} holds, then $E_1,E_2,E_3$ are in (Configuration 3), while case \eqref{case3322} cannot hold. We start by showing the latter.
\\
\\
\fbox{\eqref{case3322} does not hold.} Assume by contradiction that \eqref{case3322} holds. We start by observing that
\begin{equation}\label{int:line}
\partial E_2\cap \partial E_3 = [x_1,x_2].
\end{equation}
If this were not the case, then by \eqref{case3322} there would exist $y \in \partial E_2\cap \partial E_3\setminus L$. By convexity of $\overline{E_i}$ for $i =2,3$ and since $E_2\cap E_3=\emptyset$, this would imply that the whole triangle with vertices $y, x_1$ and $x_2$ is contained in $\partial E_2\cap \partial E_3$, which contradicts the fact that $E_2$ and $E_3$ do not interact. We now give a list of claims and show how to conclude the proof, before proving our claims. We first claim that the blow-up $V_x$ of $V_\E$ at $x \in (x_1,x_2)$, assuming $x = 0$ without loss of generality, is given by $
\llbracket\pi_{12}\rrbracket + \llbracket \pi_{13}\rrbracket.
$
More precisely, we will show that
\begin{equation}\label{claim1bis:333}
\mbox{the blow-ups of $\llbracket \partial E_2\rrbracket$ and $\llbracket \partial E_3\rrbracket$ at $x$ are $\llbracket\pi_{12} \cap \overline\Sigma_2 \rrbracket + \llbracket\pi_{13} \cap \overline\Sigma_4 \rrbracket \text{ and } \llbracket\pi_{12} \cap \overline\Sigma_4 \rrbracket + \llbracket\pi_{13} \cap \overline\Sigma_3 \rrbracket$}.
\end{equation}
Next, we are going to show that the same happens at $x_i$ for $i=1,2$, i.e. that also the blow-up at $x_i$ of $V_\E$ is given by $\llbracket\pi_{12}\rrbracket + \llbracket \pi_{13}\rrbracket$ and that 
\begin{equation}\label{claim1bis:333i}
\mbox{the blow-ups of $\llbracket \partial E_2\rrbracket$ and $\llbracket \partial E_3\rrbracket$ at $x_i$ are $\llbracket\pi_{12} \cap \overline\Sigma_2 \rrbracket + \llbracket\pi_{13} \cap \overline\Sigma_4 \rrbracket \text{ and } \llbracket\pi_{12} \cap \overline\Sigma_4 \rrbracket + \llbracket\pi_{13} \cap \overline\Sigma_3 \rrbracket$}.
\end{equation}
Furthermore, in the same proof we will deduce that $\Sigma_1$ and $ \Sigma_4$ are wedges with opening angles of $60$ degrees, and consequently $ \Sigma_2$ and $ \Sigma_3$ are wedges of opening angle of 120 degrees. From these two claims, it follows that the blow-up of $\llbracket \partial E_2 \cap \pi_{12}\rrbracket$ is a half-plane at every point of its boundary, and by Corollary \ref{regularconv} we infer that $\Gamma := \partial_{\pi_{12}}(\partial E_2 \cap \pi_{12})$ is a $C^1$ curve. Furthermore, our claims imply that $\overline{\partial E_2\setminus \pi_{12}}$ intersects $\Gamma$ at every point at a constant angle of $120$ degrees. As in Subsection \ref{subsubreg}, it follows that $\overline{\partial E_2 \setminus \pi_{12}}$ is a smooth manifold with boundary and that $\Gamma$ itself is a smooth curve. As in Case 2 of Theorem \ref{doublethm}, we can employ Alexandrov moving plane method to deduce that $E_2$ is a ball intersected $\overline{\Sigma_4\cup\Sigma_2}$. This contradicts the assumption in \eqref{case3322} that $\Gamma$ contains a segment. We now turn to the proof of our claims.
\\
\\
We first prove \eqref{claim1bis:333}. We consider the blow-up $V_x$ of $V_\E$ at $x \in (x_1,x_2)$. Without loss of generality, assume $x = 0$. Employing Corollary \ref{corstruct31}, we write
\begin{equation}\label{lined:V}
V_x = \llbracket K\rrbracket + \llbracket K_2\rrbracket + \llbracket K_3\rrbracket,
\end{equation}
Exploiting the convexity of $E_1$, it is easy to check that the blow-up of $\llbracket \partial E_1\rrbracket$ at $x$ is given by the sum of the half-planes $\llbracket\pi_{12}\cap \overline{\Sigma_1}\rrbracket + \llbracket\pi_{13}\cap \overline{\Sigma_1}\rrbracket$. By Corollary \ref{corstruct31}, we then have
\begin{equation}\label{lined:pi}
\llbracket K\rrbracket = \llbracket\pi_{12}\cap \overline{\Sigma_1}\rrbracket + \llbracket\pi_{13}\cap \overline{\Sigma_1}\rrbracket.
\end{equation}
This implies that $K\cap L = L$, and Corollary \ref{corstruct31} further tells us that $K_i\cap K = K\cap L = L$, for $i = 2,3$. To conclude the proof \eqref{claim1bis:333}, we just need to prove that
\begin{equation}\label{claim3.0}
\text{$K_2 = \pi_{13}\cap \overline{\Sigma_4}$ and $K_3 = \pi_{12}\cap \overline{\Sigma_4}$.}
\end{equation}
To this aim, we show that 
\begin{equation}\label{claim3.1}
\text{$K_2$ and $K_3$ are distinct half-planes containing $L$}.
\end{equation}
Indeed, since $K_i \cap K = L$, for $i =2,3$, both $K_2$ and $K_3$ contain $L$. Moreover, notice that, since $K_2$ and $K_3$ are cones with vertex at $x$, the set
$
I := K_2\cap K_3
$
is itself a cone with vertex at $x$. If $I = L$, then by stationarity and graphicality, compare Corollary \ref{blowcone} and Corollary \ref{corstruct31} respectively, we can employ Corollary \ref{statvarif} to infer that both $K_2$ and $K_3$ are distinct half-planes containing $L$. It cannot happen that $I \neq L$. Otherwise, $I$ contains a half-line $L'$ starting from $x$, with $L'\cap L = \{x\}$ and $L',L\subset \pi_{23}$. Using the notation $A_i\subset \pi_i$ of Corollary \ref{corstruct31}, we deduce that $A_i$ is a half-plane in bounded by $L$. Moreover, since $L',L\subset \pi_{23}$, then $\pi_i$ is not orthogonal to $\pi_{23}$. In particular $\pi_{23}$ is the graph of an affine function $h$ over $\pi_i$. By Corollary \ref{corstruct31} and the fact that $\pi_{23}$ is a separating plane among $E_2$ and $E_3$, we deduce without loss of generality that $K_2$ is the graph of the restriction of a positively one-homogeneous convex function $\varphi$ to the subset $A_i\subset \pi_i$, such that $\varphi\geq h$ on $A_i$ and $\varphi= h$ on $p_{\pi_i}(L'\cup L)$. We conclude that $\varphi$ is affine on $A_i$ and hence that $K_2$ is the half-plane bounded by $L$ in $\pi_{23}$. Analogously, we deduce the same for $K_3$ and in particular that $K_2=K_3$. Since $\partial V_x=0$, we deduce that $\pi_{12}=\pi_{23}$, which contradicts the fact that $\Sigma_1$ is nonempty. This conclude the proof of \eqref{claim3.1}.

We now show \eqref{claim3.0}. Corollary \ref{blowcone} tells us that $V_x$ must be stationary. A direct computation similar to \eqref{zeroeq}, tells us that in order for $V_x$ to be stationary, either $K_2 = \pi_{12}\cap \overline{\Sigma_4}$ and $K_3 = \pi_{13}\cap \overline{\Sigma_4}$, or $K_2 = \pi_{13}\cap \overline{\Sigma_4}$ and $K_3 = \pi_{12}\cap \overline{\Sigma_4}$.  We first show that the first case cannot occur. Indeed, since we are considering the blow-up at $x \in L \subset \pi_{23}$, we notice that the blow-up of $\partial E_2$ and $\partial E_3$ must lie in (the closure of) different connected components of $\R^3\setminus \pi_{23}$. If we had $K_2 = \pi_{12}\cap \overline{\Sigma_4}$ and $K_3 = \pi_{13}\cap \overline{\Sigma_4}$, then the blow-up of $\partial E_i$ at $x$ would be $\pi_{1i}$, for all $i =2,3$. This in turn would imply $\pi_{12} = \pi_{13} = \pi_{23}$, which results in a contradiction. This proves \eqref{claim3.0} and hence \eqref{claim1bis:333}.
By Proposition \ref{uniquecone} and \eqref{claim1bis:333}, we also deduce that 
\begin{equation}\label{E23}
E_i \subset \Sigma_i, \qquad \forall i =2,3.
\end{equation}

To conclude the proof that \eqref{case3322} does not hold, we show our last claim \eqref{claim1bis:333i} for the blow-up $W_i = W_{x_i}$ at $x_i$ of $V_\E$. We also need to show that $\Sigma_1$ and $ \Sigma_4$ are wedges with opening angles of $60$ degrees, and consequently $ \Sigma_2$ and $ \Sigma_3$ are wedges of opening angle of 120 degrees. First, Corollary \ref{corstruct31} gives us, as above, the equality
\[
W_i = \llbracket K\rrbracket + \llbracket K_2\rrbracket + \llbracket K_3\rrbracket.
\]
First assume that
\begin{equation}\label{Kwish}
\llbracket K\rrbracket = \llbracket\pi_{12}\cap \overline{\Sigma_1}\rrbracket + \llbracket\pi_{13}\cap \overline{\Sigma_1}\rrbracket,
\end{equation}
Then, as in the proof of \eqref{claim3.0} it follows that
$
\llbracket K_2 \rrbracket =  \llbracket\pi_{13}\cap \overline{\Sigma_4}\rrbracket$ and $\llbracket K_3 \rrbracket =  \llbracket\pi_{12}\cap \overline{\Sigma_4}\rrbracket.
$
Thus, $\partial E_i$ intersects  $\pi_{1i}$ at a constant angle along the points of $L$. 
By Corollary \ref{regularconv}, we deduce that $\Gamma = \partial_{\pi_{12}}(\partial E_2 \cap \pi_{12})$ is a $C^1$ curve and, since $x_i$ is the limit of points of $\partial_{\pi_{12}}(\partial E_2 \cap \pi_{12})$ on which the contact angle is $120$ degrees, we deduce that the contact angle between $\partial E_2 \setminus \pi_{12}$ and $\pi_{12}$ is $120$ degrees also at $x_i$. In particular, due to \eqref{E23}, $\Sigma_1$ and $ \Sigma_4$ are wedges with opening angles of $60$ degrees, and consequently $ \Sigma_2$ and $ \Sigma_3$ are wedges of opening angle of 120 degrees, which is the desired claim \eqref{claim1bis:333i}. 
Second, we assume that \eqref{Kwish} does not hold, and show how this leads to a contradiction. As $[x_1,x_2] \subset \overline E_1$ and we are blowing-up this set at an endpoint $x_i$, which we will assume to be $0$, by Proposition \ref{uniquecone}, the constant mean curvature of $\llbracket\partial E_1 \cap \Sigma_1\rrbracket$ and Proposition \ref{rigconv}, we deduce that:
\[
\llbracket K\rrbracket = \llbracket S_1 \rrbracket + \llbracket S_2\rrbracket + \llbracket S_3\rrbracket,
\]
where $S_\ell \subset \pi_{1\ell}$ for $\ell =2,3$ is the blow-up of the planar set $\partial E_1 \cap \pi_{1\ell}$. $S_\ell$ is either $\pi_{1\ell} \cap \overline \Sigma_1$ or it is the convex hull of a half line $H \subset L$ (given by the blow-up of $[x_1,x_2]$) starting at $0$ and another half-line $L_i'$ starting at $0$ and contained in $\pi_{1\ell}\cap \overline{\Sigma_1}$. Since \eqref{Kwish} does not hold, we can assume without loss of generality that $S_2$ belongs to the second category. $S_1$ in given by the intersection between $\overline \Sigma_1$ and the unique plane passing through $0$ and containing $L_2'$ and $L_3'$ if also $S_3$ belongs to the second category, or $L_2'$ and  $L$ if $S_3$ belongs to the first one. By Corollary \ref{corstruct31}, we know that $\llbracket K_i\rrbracket$ is the blow-up of $\llbracket \partial E_i \setminus \pi_{1i}\rrbracket$, for $i=1,2$. By \eqref{E23},  $ K_2\subset \overline{\Sigma}_2$ and $ K_3\subset \overline{\Sigma}_3$, hence by Corollary \ref{corstruct31} we deduce that $ K_2 \cap K_3=H$. By Proposition \ref{rigconv}, it follows that $K_2$ is a non-empty piece of plane contained in $\overline{\Sigma}_2$. Since $H,L_2'\subset K_2$, then $K_2$ is a non-empty piece of plane contained in $\pi_{12}$, which contradicts $E_2$ being a non-empty convex open bounded set and concludes the proof of claim \eqref{claim1bis:333i}. 
\\
\\
\fbox{If \eqref{case3321} holds, then $V_\E$ is in (Configuration 3).} 
Assume without loss of generality that $x = 0$. Apply Theorem \ref{pointrigid} to find that the opening angle $\alpha$ of the wedge $\Sigma_1$ is
\begin{equation}\label{pi3}
\alpha = \frac{\pi}{3}
\end{equation}
and that the blow-up at $x$ of $\llbracket \partial E_1\rrbracket$ is given by
\begin{equation}\label{blowe1}
\llbracket \pi_{12}\cap \overline \Sigma_1\rrbracket + \llbracket\pi_{13} \cap \overline \Sigma_1\rrbracket.
\end{equation}
Consider the blow-up $V_x = \llbracket K\rrbracket + \llbracket K_2\rrbracket + \llbracket K_3\rrbracket$ of $V_{\E}$ at $x = 0$ as
in Corollary \ref{corstruct31}. We wish to show that
\begin{equation}\label{Vx}
V_x = \llbracket \pi_{12} \rrbracket + \llbracket \pi_{13} \rrbracket.
\end{equation}
By \eqref{blowe1} and by the definition of $K$ in Corollary \ref{corstruct31}, we have
\[
\llbracket K\rrbracket = \llbracket\pi_{12}\cap \overline{\Sigma_1}\rrbracket + \llbracket\pi_{13}\cap \overline{\Sigma_1}\rrbracket.
\]
Moreover, $K_i \cap K = \partial_{\pi_{1i}}(K\cap \pi_{1i}) = \pi_{12}\cap \pi_{13} = L$. We can reason as in the proof of \eqref{claim3.0} to infer that \eqref{Vx} is the only possible stationary blow-up. In particular by \eqref{case3321} we deduce that
\begin{equation}\label{int:point}
\partial E_2\cap\partial E_3 = \{x\}.
\end{equation}
As already argued before, by \eqref{blowe1} and by Corollary \ref{regularconv}, we deduce that the curves $\Gamma_i := \partial_{\pi_{1i}}(\pi_{1i}\cap \partial E_i)$ are $C^1$ curves for $i=2,3$. Moreover by \eqref{pi3} $\overline{\partial E_i \setminus \pi_{1i}}$ intersects $\Gamma_i$ at a constant angle of 120 degrees. Finally, \eqref{int:point} and Corollary \ref{corstat} imply that $\overline{\partial E_i \setminus \pi_{1i}}$ are surfaces with constant mean curvature $\lambda$. As in the proof of \eqref{claim3} of Section \ref{main_double}, we deduce that they are smooth manifolds with smooth boundary. Hence we can apply the classical Alexandrov moving plane method \cite[Corollary 4.1.3]{LOP} to deduce that $E_i$ is a ball intersected with $\overline{\Sigma_i\cup \Sigma_4}$, for $i =2,3$. This, together with Corollary \ref{corstat}, implies that $\Gamma_2$ and $\Gamma_3$ are circles of the same radius and meeting tangentially at the point $x = 0$. By \eqref{pi3}, there is a unique sphere $S$ of curvature $\lambda$ with center in $\Sigma_1$ containing $\Gamma_i$ for $i = 2,3$ and intersecting these curves at an angle of 120 degrees. Consider finally $S' := \overline{\partial E_1 \cap \Sigma_1}$. Then, $\overline{S\cap\Sigma_1}$ and $S'$ share the same boundary $\Gamma_2 \cup \Gamma_3$ and intersect $\partial \Sigma_1$ at the boundary forming the same angle (outside of $x$). We wish to show that $S'$ and $S$ coincide in $\Sigma_1$. To do so, we consider the surface $S''$ obtained by gluing together $S'$ and $S \cap (\R^3 \setminus \Sigma_1)$. Since $S'$ and $S$ have the same tangent plane at points of $\Gamma_2 \cup \Gamma_3 \setminus \{x\}$, we see by Lemma \ref{div} that $S''$ is a constant mean curvature surface except, possibly, at $x$. Now we write $S$ and $S''$ as graphs of functions $u$ and $v$ near any point of $\Gamma_2\cup \Gamma_3 \setminus \{x\}$. Using Proposition \ref{varigraph} and elliptic regularity theory, $u$ and $v$ are analytic functions which, by construction, coincide on an open set. Thus, $S$ and $S''$ coincide in all such neighborhoods. Now essentially the same argument shows the same for points of $\overline{S'' \cap \Sigma_1} = S'$. This shows that $S' = \overline{S\cap \Sigma_1}$ as wanted. Therefore $V_\E$ is in (Configuration 3).

\subsection{Proof of (Case 4) if and only if (Configuration 4).} First we observe that $\pi_{ij}\neq \pi_{lm}$ for all $1\le i \neq j \le 3$, $1\le l \neq m \le 3$, $ (i,j) \neq (l,m)$. This is proved repeating verbatim the analogous proof at the beginning of Subcase \ref{coincid:pla}, i.e. can be excluded by considering the blow-up of $V_\E$ at a point $x \in \partial E_1\cap \partial E_2\cap\partial E_3$. More in general, none of the couples of planes $\pi_{ij}$, $\pi_{lm}$, $1\le i \neq j \le 3$, $1\le l \neq m \le 3$, $ (i,j) \neq (l,m)$ are parallel. Indeed suppose for instance that $\pi_{12} \neq \pi_{13}$ but $\pi_{12}$ is parallel to $\pi_{13}$, then it is immediate to see that $E_2$ cannot interact with $E_3$. The proof is now divided in subcases. 

\subsubsection{ (Subcase 1): $\pi_{12} \cap \pi_{13} \cap \pi_{23}=\emptyset$.}\label{subcase1:end} The three planes are parallel to a common line $L$. Hence they intersect pairwise in such a way to divide $\R^3$ in seven cylindrical regions $\Sigma_1,\Sigma_2,\Sigma_3,\Sigma_4,\Sigma_5,\Sigma_6,\Sigma_7$, where $\Sigma_7$ is a triangular cylinder. It is easy to check via simple combinatorics that this subcase does not allow $E_i$ to interact with $E_j$, for all $1\le i \neq j\le 3$, except for the case where, up to relabeling
\[
E_1\subset S_1:= \overline{\Sigma_1\cup\Sigma_2}, \quad E_2\subset S_2 := \overline{\Sigma_3\cup\Sigma_4}, \quad E_3\subset S_3 :=\overline{\Sigma_5\cup\Sigma_6}.
\]
In particular, $\partial E_1\cap\partial E_2 \cap \partial E_3 = \emptyset$. As in Case 2 of Theorem \ref{doublethm}, it follows that, for all $i=1,2,3$, $\overline{E_i \cap \intt S_i}$ is a smooth constant mean curvature surface with boundary meeting the boundary of the wedge $S_i$ at a constant angle of 120 degrees. Moreover, $\partial E_i$ does not intersect the edge of the wedge $S_i$, as otherwise by convexity it would contain a disk. Finally, by Proposition \ref{RT}, $\overline{E_i \cap \intt S_i}$ is of ring-type in the sense of Definition \ref{def:DTRT}. We can thus apply \cite[Theorem 1]{MCC} to find that the angle $\alpha_i$ of the wedge must satisfy
$
\alpha_i < \pi/3.
$
On the other hand, $\beta_i := \pi- \alpha_i$ are the angles of the section  orthogonal to $L$ of $\Sigma_7$, which is triangular. Thus,
\[
3\pi - \sum_{i=1}^3 \alpha_i= \sum_{i=1}^3 \beta_i = \pi,
\quad 
\mbox{ hence }
\quad
\sum_{i=1}^3\alpha_i = 2\pi,
\]
which is in contradiction with $\alpha_i < \pi/3$ for all $i =1,2,3$. Thus, Subcase 1  cannot happen.

\subsubsection{ (Subcase 2): $\pi_{12} \cap \pi_{13} \cap \pi_{23}=\{p\}$.} The three planes divide $\R^3$ in eight (unbounded) tetrahedral regions. In order to geometrically visualize the eight regions, assume that the three planes are the coordinates planes and that the eight regions $\Sigma_1,\Sigma_2,\Sigma_3,\Sigma_4,\Sigma_5,\Sigma_6,\Sigma_7,\Sigma_8$ are the eight octants, containing respectively the points $(1,1,1),(1,-1,1),(-1,-1,1),(-1,1,1),(1,1,-1),(1,-1,-1),(-1,-1,-1),(-1,1,-1)$. The general case is entirely analogous. Up to indices permutation and reflections, simple combinatorics shows that $E_1\subset \overline{\Sigma_1\cup\Sigma_5}$, $E_2\subset \overline{\Sigma_7\cup\Sigma_8}$, $E_3\subset \overline{\Sigma_2\cup\Sigma_3}$ and 
\begin{equation}\label{asymm}
\partial E_1 \cap \partial E_2 \subset \partial \Sigma_5 \cap \partial \Sigma_8, \qquad \mbox{while} \qquad \partial E_1 \cap \partial E_3 \subset \partial \Sigma_1 \cap \partial \Sigma_2.
\end{equation} 
As usual, we deduce as in \eqref{claim1} of Section \ref{main_double} that, if we let $\sigma_2$ and $\sigma_3$ be the two (closed) half-planes bounding the wedge $\overline{\Sigma_1\cup\Sigma_5}$, then $\overline E_1 \cap (\sigma_2 \cup \sigma_3) = \partial E_1 \cap (\sigma_2 \cup \sigma_3) = (\partial E_1\cap \partial E_2) \cup  (\partial E_1\cap \partial E_3)$.
Without loss of generality, we have
\[
\overline E_1 \cap \sigma_2 = \partial E_1 \cap \partial E_2 \quad \text{ and } \quad \overline E_1 \cap \sigma_3 = \partial E_1 \cap \partial E_3.
\] 
Furthermore, $S:=\overline{\partial E_1\cap \intt(\overline{\Sigma_1\cup\Sigma_5})}$ meets at a constant angle of 120 degrees all points of $\partial E_1 \cap \partial E_2 \setminus (\partial E_1\cap\partial E_2 \cap \partial E_3)$ and $\partial E_1 \cap \partial E_3 \setminus (\partial E_1\cap\partial E_2 \cap \partial E_3)$. Corollary \ref{corstat} shows that $S$ is a constant mean curvature surface, and hence the usual regularity analysis of \eqref{claim3} of Section \ref{main_double} applies to show that the surface is smooth up to the boundary, except for the points in $\partial E_1\cap\partial E_2 \cap \partial E_3$. We shall now prove that this is not possible. To this aim, we consider the following two options. The first option is
\begin{equation}\label{case1semifinal}
\partial E_1\cap \partial E_2 \cap \partial E_3 = \emptyset.
\end{equation}
In case \eqref{case1semifinal} does not hold, observe that 
$
\partial \Sigma_1 \cap \partial \Sigma_2 \cap \partial \Sigma_5 \cap \partial \Sigma_8 = \{p\}
$
and hence by \eqref{asymm} we conclude that the second option is:
\begin{equation}\label{case2semifinal}
\partial E_1\cap \partial E_2 \cap \partial E_3 = \{p\}.
\end{equation}
In case \eqref{case1semifinal} holds, we have that $\partial E_1$ does not touch the edge of the wedge $\overline{\Sigma_1\cup\Sigma_5}$, while if \eqref{case2semifinal} holds, $p$ is the only point where $\partial E_1$ touches the edge of the wedge $\overline{\Sigma_1\cup\Sigma_5}$. Indeed, if any of these assertions were false, $\partial E_1 \setminus (\partial E_2 \cup \partial E_3)$ would contain a disk, which leads to a contradiction as in the proof of \eqref{claim1} of Section \ref{main_double}. If \eqref{case1semifinal} holds, $\overline{\partial E_1 \setminus (\partial E_2 \cup \partial E_3)}$ is a smooth ring-type surface by Proposition \ref{RT}. Thus, it follows by \cite[Theorem 1]{PRT} that $S$ is a piece of a sphere. Let $R$ be one of the two rotations of $\R^3$ that brings $\sigma_3$ into $\sigma_2$. As $E_1$ is a ball intersected with the wedge $\overline{\Sigma_1\cup\Sigma_5}$, then either
\[
\partial E_1 \cap \partial E_2 =\overline E_1 \cap \sigma_2 \subseteq R(\overline E_1 \cap \sigma_3) = R(\partial E_1 \cap \partial E_3),
\]
or\[
\partial E_1 \cap \partial E_2 =\overline E_1 \cap \sigma_2 \supseteq R(\overline E_1 \cap \sigma_3) = R(\partial E_1 \cap \partial E_3).
\]
However, any of the two previous alternatives is impossible by \eqref{asymm}. This excludes \eqref{case1semifinal}. If \eqref{case2semifinal} holds, then we see that \eqref{asymm} is in contradiction with Proposition \ref{pointsym}. This shows that Subcase 2 does not occur.

\subsubsection{ (Subcase 3): $\pi_{12} \cap \pi_{13} \cap \pi_{23}=L$.} The three planes divide $\R^3$ in six cylindrical regions. Again via simple combinatorics it is easy to check that the only configuration that allows $E_i$ to interact with $E_j$, for all $1\le i \neq j\le 3$, is that each of the $E_i$ occupies the union $\Sigma_i$ of two consecutive cylindrical regions, with $\cup_{l,m} (\partial E_i\cap \pi_{lm})$ contained (up to an $\Haus^2$-measure zero set) in three half planes meeting at $L$. We start by the usual observation that $
\partial E_i \cap \pi_{ij}  = \partial E_j \cap \pi_{ij}$ for every $1\leq i<j\leq 3$, see \eqref{claim1} of Section \ref{main_double}. This yields:
\begin{equation}\label{intersecLa}
\partial E_i \cap L  = \partial E_1 \cap L, \quad \forall i =2,3.
\end{equation}
First, we wish to show that the case
\begin{equation}\label{case331a}
 \partial E_1 \cap L =\partial E_2 \cap L = \partial E_3 \cap L = \emptyset
\end{equation}
cannot hold. Indeed, by Corollary \ref{corstat}, for all $i =1,2,3$, $\partial E_i \cap \Sigma_i$ is a surface with constant mean curvature that meets $\pi_{12}$ and $\pi_{13}$ with constant angle, as can be seen by the usual blow-up analysis, see \eqref{claim2}-\eqref{claim3} of Theorem \ref{doublethm}. Moreover, by Proposition \ref{RT}, $\overline{\partial E_i\cap \Sigma_i}$ is of ring-type according to Definition \ref{def:DTRT}. Thus we are in the assumptions of \cite[Theorem 1]{MCC}, which tells us that the opening angle of the wedge $\Sigma_i$ must be
$
\alpha_i \le \pi/3$, for every $i =1,2,3.
$
Since $\sum_i\alpha_i = 2\pi$, then \eqref{case331a} cannot hold. The same argument, using Theorem \ref{pointrigid} in place of \cite[Theorem 1]{MCC}, excludes the case
\begin{equation}\label{intersecLabis}
 \partial E_1 \cap L =\partial E_2 \cap L = \partial E_3 \cap L = \{p\}.
\end{equation}
Thus, to conclude the analysis of (Subcase 3), we shall consider the case in which \eqref{case331a}-\eqref{intersecLabis} do not hold, i.e.:
\begin{equation}\label{case3322a}
\mbox{there exists $x_1\neq x_2$ such that }\partial E_1 \cap L = \partial E_2\cap L = \partial E_3\cap L = [x_1,x_2].
\end{equation}
Arguing similarly to \eqref{claim2}-\eqref{claim3} of Theorem \ref{doublethm}, we have that for all $i=1,2,3$:
\begin{itemize}
\item $\overline {\partial E_i \cap \Sigma_i}$ intersects $\partial \Sigma_i \setminus L$ with constant angle of $120$ degrees;
\item If $\partial E_i\cap \partial \Sigma_i \subset \pi_{jk}\cup \pi_{\ell m}$, then $\partial_{\pi_{jk}}(\partial E_i \cap \pi_{jk}) \setminus L$ and $\partial_{\pi_{\ell m}}(\partial E_i \cap \pi_{\ell m}) \setminus L$ are smooth open curves;
\item $\overline{\partial E_i\cap \Sigma_i}$ is a smooth surface with smooth boundary except possibly for $\{x_1,x_2\}$.
\end{itemize}
One last information that can be deduced as in \eqref{claim2}, i.e. by taking blow-ups of $V_\E$ at $x \in (x_1,x_2)$, is that
\begin{equation}\label{pij120}
\text{the opening angle of } \partial \Sigma_i \text{ is 120 degrees.}
\end{equation}
Our aim is to show that, for all $i =1,2,3$,
\begin{equation}\label{Eisball}
E_i \text{ is a ball intersected with }\Sigma_i.
\end{equation}
Suppose for a moment that \eqref{Eisball} holds. Then, all these balls must have the same radius by Corollary \ref{corstat} and must meet the boundary of the wedges in angles of 120 degrees. Combining this with \eqref{pij120}, we infer that $E_1,E_2$ and $E_3$ form a standard triple bubble as in Definition \ref{STTB} and we thus conclude that we are in (Configuration 4). This would conclude the proof of the classification. We are thus just left to prove \eqref{Eisball}.
To this aim, we can either adapt the same proof of Theorem \ref{pointrigid}, or alternatively we can apply \cite[Theorem 1]{FMC}. We decided to follow the second method. In the following, we will assume without loss of generality $i = 1$. We first state the required assumptions which, in \cite{FMC}, are collected in Hypothesis B, and subsequently we show why they are satisfied in our setting.
\\
\\
Let $S \subset \R^3$ be a surface with boundary lying in a collection of (finitely many) planes $\pi_j \subset \R^3$, such that the intersection angle with $\pi_j$ is a constant $\gamma_j$, and such that $S$ has constant mean curvature away from the intersections. Suppose also that the planes bound an open connected region $\mathcal{I}$ containing $S$. A disk-type surface $S$ is said to satisfy Hypothesis B if the following conditions hold\footnote{The fact that $S$ is of disk-type and satisfies the Angle Condition is not required in \cite[Hypothesis B]{FMC}, but it is part of the assumptions of \cite[Theorem 2]{FMC}, so we write it here for the sake of exposition.}.
\begin{description}
\item[(Topological Condition)] Let $D := \overline{B_1(0)}$ and let $v_i$, $i \in \{1,\dots, V\}$, be a finite collection of points in $\partial D$ clockwise-ordered. There is a local homeomorphism $\Phi$ of $D_v := D\setminus \{v_1,\dots, v_V\}$ onto $S$, i.e. for all $a \in D_v$, there is some neighborhood $B(a)$ of $a$ in $D_v$ such that $\Phi$ restricted to $B(a)$ is a homeomorphism;
\item[(Smoothness Condition)] Let $A_i$ be an (open) arc on $\partial D$ connecting $v_i$ and $v_{i + 1}$. Let also $\Phi: D\to S$ be the map of the first point. We require that $\Phi$ can be made smooth everywhere, except, possibly, at $v_i$. This means that for all point $a \in \intt D$, there is a neighborhood $U = U(a)$ and a homemorphism $\psi: B_1(0) \to U(a)$ such that $\Phi\circ\psi: B_1(0) \to S$ is a smooth embedded (open) surface, and that for all points $a \in \partial D\setminus \{v_1,\dots, v_V\}$, there is a neighborhood $U = U(a) \subset D$ and a homeomorphism $\psi: B_1^+ \to U(a)$ such that $\Phi\circ \psi: B_1^+ \to S$ is a smooth embedded surface with boundary. Here, $B_1^+ := \{(x_1,x_2) \in B_1(0): x_2 \ge 0\}$.
\item[(Vertex Condition)] For every $v_i$, there is a couple of intersecting planes $\pi_j$ and $\pi_k$, a neighborhood $U$ of $v_i$ in $D$ and a neighborhood $N$ of $O = \pi_j\cap \pi_k \cap \pi$ in $\pi \cap \overline {\mathcal{I}}$, such that $\Phi(B\setminus\{v_i\})$ is the graph of a function $u_i$ over $N \setminus O$. Here, $\pi$ is a plane orthogonal to both $\pi_j$ and $\pi_k$.
\item[(Separation Condition)] Define a vertex $\mathcal{V}_i$ to be the triple $\{(\pi_j,\pi_k),N,u_i\}$. If $u_i$ extends continuously up to $\pi_j\cap \pi_k$ with continuous first derivatives, then this determines a vertex point, again denoted by $\mathcal{V}_i$. We required that, for all $i$, $u_i$ determines a vertex point $\mathcal{V}_i$, $\Phi$ extends continuously to the boundary and that $\Phi^{-1}(\mathcal{V}_i) = v_i$.
\item[(Angle Condition)] For all couple of planes $\pi_j$ and $\pi_k$, the contact angles $\gamma_j$ and $\gamma_k$ with respectively $\pi_j$ and $\pi_k$ must lie in the interior of the rectangle of \cite[Fig. 3]{FMC}-\cite[Fig. 5]{CFR}.
\end{description}

In our case, we have only two planes, $\pi_1 := \pi_{12}$, $\pi_2 := \pi_{13}$, and the constant angles $\gamma_1,\gamma_2$ are equal to 120 degrees. Moreover, $\mathcal{I} = \Sigma_{1}$ and we will show that the vertex points are precisely $\mathcal{V}_i = x_i$. It is convenient to set $S := \overline{\partial E_1 \cap \Sigma_1}$. Recall that, by \eqref{pij120}, $\overline \Sigma_1$ is a wedge of opening angle of 120 degrees. Thus, one can check immediately that the \textbf{(Angle Condition)} is satisfied. Next, to build the map $\Phi$, we do the following. Consider a ball $B_R(c)$, where $c = \frac{x_1 + x_2}{2}$ and $R$ is so large that $E_1 \subset B_{\frac{R}{2}}(c)$. Suppose without loss of generality that $c = 0$. Consider the set 
\[
P =\partial B_R(c) \cap \overline{\Sigma_{1}}.
\]
The map $F(x) := R\frac{x}{|x|}$, defined for $x \neq 0$, is a global homeomorphism between $S$ and $P$. It is also immediate to find a homeomorphism $g$ between $D$ and $P$. Thus, there exists a homeomorphism $\Phi: D \to S$. We let $v_i := \Phi^{-1}(x_i)$, and in fact one can build $g$ to be a diffeomorphism between $D \setminus \{v_1,v_2\}$ and $P \setminus \{\Phi(x_1),\Phi(x_2)\}$. This is showing that $S$ is a disk-type surface and that the \textbf{(Topological Condition)} holds. To show the \textbf{(Smoothness Condition)}, we begin by considering a point $a \in \intt D$. We can simply use Proposition \ref{graph} to write $S$ as the graph of a Lipschitz (convex) function $\phi$ near $\Phi(a)$, say in a neighborhood $N(\Phi(a))$. As $S$ is of constant mean curvature, standard (interior) regularity theory shows that $\phi$ is smooth. Up to translating and rotating $S$, we can suppose that $S$ is parametrized by $\Gamma_\phi: (x_1,x_2)\in B_r(0) \mapsto ((x_1,x_2),\phi(x_1,x_2))\in N(\Phi(a))$ near $\Phi(a)$. Then, the map $\psi := \Phi^{-1}\circ \Gamma_\phi : B_r(0) \to \intt D$ provides the map required by the \textbf{(Smoothness Condition)}. Similarly, near points $a \in \partial D\setminus \{v_1,v_2\}$, we can again write $S$ as the graph over $B_1^+(0)$ of a Lipschitz function $f$ which satisfies an elliptic PDE since $S$ is of constant mean curvature, and fulfills some appropriate boundary condition due to the fact that $S$ intersects $\partial \Sigma_1$ with constant angle. As in Subsection \ref{subsubreg}, we find that $f$ is smooth. We can define $\psi$ similarly as above. This shows that the \textbf{(Smoothness Condition)} is satisfied at every point of $S \setminus \{x_1,x_2\}$. To show the \textbf{(Vertex Condition)}, i.e. to write, for a small $\delta > 0$, $S \cap B_\delta(x_i)$ as the graph of a convex (or concave) function on a plane orthogonal to $\pi_{1j}$ for $j =2,3$ at $x_i$  for $i =1,2$, we can reason similarly to Lemma \ref{grafo}. The functions built with that method are continuous. To see that they are $C^{1}$ up to $x_i$, one may employ \cite[Theorem 1]{SIMREG}. Thus, $x_i$ are vertices as in the \textbf{(Separation Condition)}. Moreover, by construction $\Phi$ is continuous up to $x_i$ and $v_i := \Phi^{-1}(x_i)$ for all $i =1,2$. This finishes the proof that our surface $S$ fulfills the assumptions of \cite[Theorem 1]{FMC}, and is therefore a piece of sphere. The analysis of the critical points for the triple bubble problem is thus complete.

\appendix

\section{First variation of convex $k$-bubbles}\label{app:stat}
Here we wish to give a sketch of the proof of Proposition \ref{equiva}, that we recall below. As said above, the proof closely follows \cite[Appendix B-C]{MNE}.
\begin{proof}[Proof of Proposition \ref{equiva}]
We introduce the intermediate condition:
\begin{equation}\label{ug1}
[\delta V_\E](g)=0 \quad \mbox{for every $g \in C^\infty_c(\R^n,\R^n)$ satisfying
$\int_{\partial^*E_i}(n_{E_i},g)d\mathcal{H}^{n-1} = 0$,}
\end{equation}
where we used Lemma \ref{lele} to infer $\theta=1$. Then, the following hold:
\[
\eqref{doubprob} \Leftrightarrow \eqref{ug1} \Rightarrow V_\E \text{ is stationary for the $k$-bubble problem}.
\]
Indeed, $\eqref{doubprob} \Rightarrow \eqref{ug1}$ is immediate. The converse is a linear algebra computation, which can be proved as in \cite[Corollary 6.9]{DRKS}. To prove that $\eqref{ug1} \Rightarrow V_\E \text{ is stationary for the $k$-bubble problem}$, it is sufficient to notice that, if $\Phi_t$ is the flow associated to $g$, the condition $|\Phi_t(E_i)| = |E_i|$
implies, taking the derivative at $t = 0$,
\[
0 = \int_{E_i}\dv(g)dx = \int_{\partial E_i^*}(g,n_{E_i})d\mathcal{H}^{n-1}.
\]
Thus, the proof is finished if we show that
\begin{equation}\label{interme}
V_\E \text{ is stationary for the $k$-bubble problem} \Rightarrow \eqref{ug1}.
\end{equation}
From now on, assume $V_\E$ is stationary for the $k$-bubble problem and $g \in C^\infty_c(\R^n,\R^n)$ is such that 
$$\int_{\partial^* E_i}(g,n_{E_i})d\mathcal{H}^{n-1} = 0.$$
 We define
\[
E_{k + 1} := \left(\bigcup_{i = 1}^k\overline{E_i}\right)^c.
\]
The key point, as in \cite[Lemma C.3]{MNE}, is to find $k$ vector fields $Y_1,\dots, Y_k \in C^\infty_c(\R^n,\R^n)$ with pairwise disjoint supports such that the vectors
\[
w_\ell := \left(\int_{\partial^* E_1}(Y_\ell,n_{E_1})d\mathcal{H}^{n-1},\dots, \int_{\partial^* E_k}(Y_\ell,n_{E_k})d\mathcal{H}^{n-1}, \int_{\partial^* E_{k+ 1}}(Y_\ell,n_{E_{k+1}})d\mathcal{H}^{n-1}\right) \in \R^{k + 1}, \mbox{ $\ell=1,\dots,k$}
\]
span $S := \{x \in \R^{k + 1}: \sum_{i}x_i = 0\}$. Once this is shown, the proof can be concluded as in \cite[Lemma C.3]{MNE}, except that stationarity is used instead of minimality. We now move to the existence of the vector fields $\{Y_s\}_{1\le s \le k}$.
\\
\\
In order to find vector fields $\{Y_\ell\}_{1\le \ell \le k}$ with properties as above, we wish to use the same proof of \cite[Lemma C.2]{MNE}. This can be done once we prove that for every $1\le i < j \le k+1$, if $\mathcal{H}^{n-1}(\partial E_i \cap \partial E_j) > 0$, then there exists a point $x_{ij} \in \partial E_i \cap \partial E_j$ and $\delta > 0$ with the following properties:
\begin{enumerate}
\item $B_{\delta}(x_{ij}) \cap \overline{E_\ell} = \emptyset,\quad \forall \ell \neq i,j$;\label{app:primo}
\item $V_\E \llcorner B_{\delta}(x_{ij})$ is a smooth $(n - 1)$-dimensional submanifold.\label{app:terzo}
\end{enumerate}
These properties are shown in \cite{MNE} using the minimality of the $k$-cluster. Here we will show it by exploiting convexity. Moreover, denoting with $\{e_i\}_{i=1}^{k+1}$ the canonical base of $\R^{k+1}$, we also need to show the following result, compare \cite[Lemma C.1]{MNE}: 
\begin{equation}\label{app:graph}
\mbox{the set $O := \{e_i - e_j: \mathcal{H}^{n - 1}(\partial E_i \cap \partial E_j) > 0\}$ spans $S$.}
\end{equation}
We start by showing the first statement. If $j < k + 1$, then the assertion is simple: in that case, $\partial E_i \cap \partial E_j$ is contained in a plane $\pi$, and it suffices to take $x_{ij} \in \intt_{\pi}(\partial E_i \cap \partial E_j)$. Properties \eqref{app:primo}-\eqref{app:terzo} follow rather easily if $\delta$ is chosen small enough. Assume now $j = k + 1$. Take now $x_{i(k+1)} \in \partial E_i \cap \partial E_{k + 1} \setminus \cup_{\ell \neq i, k +1 }(\partial E_\ell)$, which is possible by \eqref{intbound}. Since the sets $E_j$ are disjoint for all $j =1,\dots, k +1$, it follows that there exists a small $\delta >0$ such that $\overline{B_\delta(x_{i(k+1)})} \cap \overline E_\ell = \emptyset$ for all $\ell \neq i,k + 1$. This shows \eqref{app:primo}. To show \eqref{app:terzo}, we first notice that, up to decreasing $\delta > 0$, by Proposition \ref{graph}, we can suppose that $V_\E\llcorner B_\delta(x_{i(k+1)})$ is given by the graph of a Lipschitz function $u$. Then, it suffices to notice that by restricting the stationarity condition \eqref{doubprob1} to vector fields $g$ supported in $B_\delta(x_{i(k + 1)})$, one has that $V_\E\llcorner  B_\delta(x_{i(k+1)})$ has constant mean curvature. To see this, one can adapt the proof for minimizers of \cite[Section 17.5]{MAG} to the case of stationary points. In particular, $u$ is analytic and \eqref{app:terzo} is shown. We now move to showing \eqref{app:graph}. Notice that if for all $1 \le i \le k$, $\mathcal{H}^{n-1}(\partial E_i \cap \partial E_{k+1}) > 0$, \eqref{app:graph} holds. More generally, if, given any $i_0 \in \{1,\dots, k\}$, 
\begin{equation}\label{path}
\mbox{there exist $i_1,\dots, i_\ell$ such that $i_\ell = k + 1$ and $\mathcal{H}^{n-1}(\partial E_{i_{j}}\cap \partial E_{i_{j + 1}}) > 0 \,\, \forall j \in \{0,\dots, \ell - 1\},$}
\end{equation}
then \eqref{app:graph} holds. To see that \eqref{path} holds, take any point $x_0 \in E_{i_0}$ and a small ball $B_\eps(x_0) \subset E_{i_0}$. Consider the union $P$ of the  projections onto $\partial B_\eps(x_0)$ of the sets
\[
\mbox{$\partial E_a\cap \partial E_b$ with $\mathcal{H}^{n - 1}(\partial E_a\cap \partial E_b)=0$, for all $1\leq a\neq b\leq k+1$. }
\]
Since the projection is Lipschitz and all of these sets have $\mathcal{H}^{n-1}$ measure zero, it follows that $P$ has zero $\mathcal{H}^{n - 1}$ measure $0$. Since $\mathcal{H}^{n - 1}(\partial B_\eps(x_0)) > 0$, we can then find $w \in \partial B_\eps(x_0) \setminus P$. Let $r(t)$ the parametrization of the line $r(t) = x_0 + tw$. By construction, $r(0) \in E_i$ and for every $t$ large enough, $r(t) \in E_{k + 1}$ since every other $E_j$ is bounded and $(B_R(0))^c \subset E_{k + 1}$ for some large $R > 0$. Moreover there exists $t_1>0$ such that $r(t_1) \in \partial E_{i_0} \cap \partial E_b$ for some $b\neq i_0$. Then $\mathcal{H}^{n-1}(\partial E_{i_0} \cap \partial E_b) > 0$ by definition of $P$, and we set $i_1:=b$. Proceeding iteratively along $r(t)$, this construction provides us with the required path \eqref{path} and concludes the proof.

\end{proof}

\bibliographystyle{plain}
\bibliography{ARpaperDouble}

\begin{bibdiv}
\begin{biblist}

\bib{ADN}{article}{
      author={Agmon, S.},
      author={Douglis, A.},
      author={Nirenberg, L.},
       title={Estimates near the boundary for solutions of elliptic partial
  differential equations satisfying general boundary conditions. i},
        date={1959-11},
     journal={Communications on Pure and Applied Mathematics},
      volume={12},
      number={4},
       pages={623\ndash 727},
}

\bib{MR0102114}{article}{
      author={Alexsandrov, A.~D.},
       title={Uniqueness theorem for surfaces in the large},
        date={1956},
     journal={Vestnik Leningrad Univ.},
      volume={11},
       pages={5\ndash 17},
}

\bib{ABR}{book}{
      author={Axler, Sheldon},
      author={Bourdon, Paul},
      author={Ramey, Wade},
       title={{Harmonic Function Theory}},
   publisher={Springer New York},
        date={2001},
}

\bib{TRIF1}{article}{
      author={Choe, Jaigyoung},
       title={Sufficient conditions for constant mean curvature surfaces to be
  round},
        date={2002-05},
     journal={Mathematische Annalen},
      volume={323},
      number={1},
       pages={143\ndash 156},
}

\bib{CLM}{article}{
      author={Cicalese, Marco},
      author={Leonardi, Gian~Paolo},
      author={Maggi, Francesco},
       title={Sharp stability inequalities for planar double bubbles},
        date={2017},
     journal={Interfaces and Free Boundaries},
      volume={19},
      number={3},
       pages={305\ndash 350},
}

\bib{CONFIN}{article}{
      author={Concus, Paul},
      author={Finn, Robert},
       title={On the behavior of a capillary surface in a wedge},
        date={1969-06},
     journal={Proceedings of the National Academy of Sciences},
      volume={63},
      number={2},
       pages={292\ndash 299},
}

\bib{CFR}{article}{
      author={Concus, Paul},
      author={Finn, Robert~S.},
       title={Capillary wedges revisited},
        date={1996},
     journal={Siam Journal on Mathematical Analysis},
      volume={27},
       pages={56\ndash 69},
}

\bib{DLDPKT}{article}{
      author={De~Lellis, Camillo},
      author={De~Philippis, Guido},
      author={Kirchheim, Bernd},
      author={Tione, Riccardo},
       title={Geometric measure theory and differential inclusions},
        date={2021-12},
     journal={Annales de la Facult{\'{e}} des sciences de Toulouse :
  Math{\'{e}}matiques},
      volume={30},
      number={4},
       pages={899\ndash 960},
}

\bib{DDG}{article}{
      author={De~Philippis, Guido},
      author={De~Rosa, Antonio},
      author={Ghiraldin, Francesco},
       title={{Rectifiability of Varifolds with Locally Bounded First Variation
  with Respect to Anisotropic Surface Energies}},
        date={2017-08},
        ISSN={0010-3640},
     journal={Communications on Pure and Applied Mathematics},
      volume={71},
      number={6},
       pages={1123\ndash 1148},
}

\bib{DG}{article}{
      author={De~Rosa, Antonio},
      author={Gioffr{\`e}, Stefano},
       title={Quantitative stability for anisotropic nearly umbilical
  hypersurfaces},
        date={2019},
     journal={The Journal of Geometric Analysis},
      volume={29},
      number={3},
       pages={2318\ndash 2346},
}

\bib{DG2}{article}{
      author={De~Rosa, Antonio},
      author={Gioffr{\`e}, Stefano},
       title={Absence of bubbling phenomena for non-convex anisotropic nearly
  umbilical and quasi-einstein hypersurfaces},
        date={2021},
     journal={Journal f{\"u}r die reine und angewandte Mathematik (Crelles
  Journal)},
      volume={2021},
      number={780},
       pages={1\ndash 40},
}

\bib{DRKS}{article}{
      author={De~Rosa, Antonio},
      author={Kolasi{\'{n}}ski, S{\l}awomir},
      author={Santilli, Mario},
       title={{Uniqueness of Critical Points of the Anisotropic Isoperimetric
  Problem for Finite Perimeter Sets}},
        date={2020-08},
     journal={Archive for Rational Mechanics and Analysis},
      volume={238},
      number={3},
       pages={1157\ndash 1198},
}

\bib{AR}{article}{
      author={De~Rosa, Antonio},
      author={Tione, Riccardo},
       title={Regularity for graphs with bounded anisotropic mean curvature},
        date={2022},
     journal={Inventiones mathematicae},
      volume={230},
      number={2},
       pages={463\ndash 507},
}

\bib{De}{article}{
      author={Delauney, C.},
       title={Sur la surface de r{\'e}volution dont la courbure moyenne est
  constant},
        date={1841},
     journal={J. Math. Pures et Appl.},
       pages={309\ndash 320},
}

\bib{DMMN}{article}{
      author={Delgadino, Matias~G},
      author={Maggi, Francesco},
      author={Mihaila, Cornelia},
      author={Neumayer, Robin},
       title={Bubbling with l2-almost constant mean curvature and an
  alexandrov-type theorem for crystals},
        date={2018},
     journal={Archive for Rational Mechanics and Analysis},
      volume={230},
      number={3},
       pages={1131\ndash 1177},
}

\bib{DELWE}{misc}{
      author={Delgadino, Matias~G.},
      author={Weser, Daniel},
       title={{A Heintze--Karcher inequality with free boundaries and
  applications to capillarity theory}},
   publisher={arXiv},
        date={2022},
         url={https://arxiv.org/abs/2210.16376},
}

\bib{DEM}{article}{
      author={Delgadino, Matias~Gonzalo},
      author={Maggi, Francesco},
       title={Alexandrov's theorem revisited},
        date={2019-02},
     journal={Analysis {\&} {PDE}},
      volume={12},
      number={6},
       pages={1613\ndash 1642},
}

\bib{DLSW}{article}{
      author={Dorff, Rebecca},
      author={Lawlor, Gary},
      author={Sampson, Donald},
      author={Wilson, Brandon},
       title={Proof of the planar double bubble conjecture using
  metacalibration methods},
        date={2010-03},
     journal={Involve, a Journal of Mathematics},
      volume={2},
      number={5},
       pages={611\ndash 628},
}

\bib{EVAPDE}{book}{
      author={Evans, Lawrence~C.},
       title={{Partial Differential Equations}},
   publisher={American Mathematical Soc.},
        date={1998},
      volume={19},
}

\bib{EVG}{book}{
      author={Evans, Lawrence~C.},
      author={Gariepy, Ronald~F.},
       title={Measure theory and fine properties of functions},
   publisher={Chapman , Hall/CRC},
        date={2015},
}

\bib{FINB}{book}{
      author={Finn, Robert},
       title={{Equilibrium Capillary Surfaces}},
   publisher={Springer New York},
        date={1986},
}

\bib{FMC}{article}{
      author={Finn, Robert},
      author={Mccuan, John},
       title={Vertex theorems for capillary drops on support planes},
        date={2000},
        ISSN={0025-584X},
     journal={Math. Nachr},
      volume={209},
       pages={115\ndash 135},
}

\bib{FF1}{article}{
      author={Foisy, Joel},
       title={{Soap bubble clusters in $\R^2$ and $\R^3$}},
        date={1991},
     journal={Undergraduate thesis at Williams College},
}

\bib{FF2}{article}{
      author={Foisy, Joel},
      author={Alfaro, Manuel},
      author={Brock, Jeffrey},
      author={Hodges, Nickelous},
      author={Zimba, Jason},
       title={{The standard double soap bubble in $\R^2$ uniquely minimizes
  perimeter}},
        date={199305},
     journal={Pacific Journal of Mathematics},
      volume={159},
}

\bib{GM}{misc}{
      author={Giaquinta, Mariano},
      author={Martinazzi, Luca},
       title={An introduction to the regularity theory for elliptic systems,
  harmonic maps and minimal graphs},
        date={2012},
}

\bib{GT}{book}{
      author={Gilbarg, David},
      author={Trudinger, Neil~S.},
       title={{E}lliptic {P}artial {D}ifferential {E}quations of {S}econd
  {O}rder},
   publisher={Springer},
        date={1977},
}

\bib{GRIS}{book}{
      author={Grisvard, Pierre},
       title={{Elliptic Problems in Nonsmooth Domains}},
   publisher={Society for Industrial and Applied Mathematics},
        date={2011},
}

\bib{HHHS}{article}{
      author={Hass, Joel},
      author={Hutchings, Michael},
      author={Schlafly, Roger},
       title={The double bubble conjecture.},
    language={eng},
        date={1995},
     journal={Electronic Research Announcements of the American Mathematical
  Society [electronic only]},
      volume={1},
      number={3},
       pages={98\ndash 102},
         url={http://eudml.org/doc/227249},
}

\bib{HLRS}{article}{
      author={Heilmann, Cory},
      author={Lai, Yuan},
      author={Reichardt, Ben},
      author={Spielman, Anita},
       title={Component bounds for area minimizing double bubbles},
        date={1999},
     journal={NSF SMALL undergraduate research Geometry Group report, Williams
  College},
}

\bib{HRT}{article}{
      author={Hirsch, Jonas},
      author={Tione, Riccardo},
       title={On the constancy theorem for anisotropic energies through
  differential inclusions},
        date={2021-04},
     journal={Calculus of Variations and Partial Differential Equations},
      volume={60},
      number={3},
}

\bib{HH}{article}{
      author={Hutchings, Michael},
       title={{The structure of area-minimizing double bubbles}},
        date={1997-06},
     journal={The Journal of Geometric Analysis},
      volume={7},
      number={2},
       pages={285\ndash 304},
}

\bib{HMRR1}{article}{
      author={Hutchings, Michael},
      author={Morgan, Frank},
      author={Ritor{\'{e}}, Manuel},
      author={Ros, Antonio},
       title={Proof of the double bubble conjecture},
        date={2000-07},
     journal={Electronic Research Announcements of the American Mathematical
  Society},
      volume={6},
      number={6},
       pages={45\ndash 49},
}

\bib{HMRR}{article}{
      author={Hutchings, Michael},
      author={Morgan, Frank},
      author={Ritor{\'e}, Manuel},
      author={Ros, Antonio},
       title={Proof of the double bubble conjecture},
        date={2002},
     journal={Annals of Mathematics},
       pages={459\ndash 489},
}

\bib{JWXZ}{misc}{
      author={Jia, Xiaohan},
      author={Wang, Guofang},
      author={Xia, Chao},
      author={Zhang, Xuwen},
       title={{Heintze-Karcher inequality and capillary hypersurfaces in a
  wedge}},
   publisher={arXiv},
        date={2022},
         url={https://arxiv.org/abs/2209.13839},
}

\bib{JOST}{book}{
      author={Jost, J{\"u}rgen},
       title={{Two-Dimensional Geometric Variational Problems}},
   publisher={Wiley},
        date={1991},
}

\bib{KNS}{article}{
      author={Kinderlehrer, D.},
      author={Nirenberg, L.},
      author={Spruck, J.},
       title={{Regularity in elliptic free boundary problems I}},
        date={1978-12},
     journal={Journal d{\textquotesingle}Analyse Math{\'{e}}matique},
      volume={34},
      number={1},
       pages={86\ndash 119},
}

\bib{KRA}{book}{
      author={Krantz, Steven~G.},
       title={{Geometric Function Theory}},
   publisher={Birkh{$\"a$}user Boston},
        date={2007},
}

\bib{L}{article}{
      author={Lawlor, Gary~R.},
       title={{Double Bubbles for Immiscible Fluids in $\mathbb{R}^n$}},
        date={2014},
        ISSN={1559-002X},
     journal={Journal of Geometric Analysis},
      volume={24},
      number={1},
       pages={190\ndash 204},
         url={https://doi.org/10.1007/s12220-012-9333-1},
}

\bib{LOP}{book}{
      author={L{\'{o}}pez, Rafael},
       title={{Constant Mean Curvature Surfaces with Boundary}},
   publisher={Springer Berlin Heidelberg},
        date={2013},
}

\bib{RLC}{article}{
      author={L{\'{o}}pez, Rafael},
       title={Capillary surfaces with free boundary in a wedge},
        date={2014-09},
     journal={Advances in Mathematics},
      volume={262},
       pages={476\ndash 483},
}

\bib{MAG}{book}{
      author={Maggi, Francesco},
       title={{Sets of Finite Perimeter and Geometric Variational Problems}},
   publisher={Cambridge University Press},
        date={2012},
}

\bib{MCC}{article}{
      author={McCuan, John},
       title={Symmetry via spherical reflection and spanning drops in a wedge},
        date={1997-10},
     journal={Pacific Journal of Mathematics},
      volume={180},
      number={2},
       pages={291\ndash 323},
}

\bib{MMRM}{article}{
      author={Merkle, Monica Moulin~Ribeiro},
       title={Symmetries in some extremal problems between two parallel
  hyperplanes},
        date={2016-01-12},
      eprint={1601.02959},
}

\bib{MIE2}{article}{
      author={Miersemann, Eric},
       title={{Asymptotic Expansion at a Corner for the Capillarity Problem}},
        date={1988},
     journal={Pacific Journal of Math.},
      volume={134},
      number={2},
}

\bib{MIE1}{article}{
      author={Miersemann, Erich},
       title={{On Capillary Free Surfaces without Gravity}},
        date={1985},
     journal={Zeitschrift f\"ur Analysis und ihre Anwendungen},
      volume={4},
      number={5},
       pages={429\ndash 436},
}

\bib{MNS}{misc}{
      author={Milman, Emanuel},
      author={Neeman, Joe},
       title={The structure of isoperimetric bubbles on $\mathbb{R}^n$ and
  $\mathbb{S}^n$},
   publisher={arXiv},
        date={2022},
         url={https://arxiv.org/abs/2205.09102},
}

\bib{MNE}{article}{
      author={Milman, Emanuel},
      author={Neeman, Joe},
       title={{The Gaussian Double-Bubble and Multi-Bubble Conjectures}},
        date={2022-01},
     journal={Annals of Mathematics},
      volume={195},
      number={1},
}

\bib{MM1}{article}{
      author={Morgan, Frank},
       title={The double bubble conjecture},
        date={1995},
     journal={FOCUS},
}

\bib{MM2}{book}{
      author={Morgan, Frank},
       title={{Geometric Measure Theory - A Beginner's Guide}},
   publisher={Elsevier},
        date={2016},
}

\bib{MW}{article}{
      author={Morgan, Frank},
      author={Wichiramala, Wacharin},
       title={{The standard double bubble is the unique stable double bubble in
  $\mathbb{R}^2$}},
        date={2002-04},
     journal={Proceedings of the American Mathematical Society},
      volume={130},
      number={9},
       pages={2745\ndash 2751},
}

\bib{PT}{article}{
      author={Pacella, Filomena},
      author={Tralli, Giulio},
       title={Overdetermined problems and constant mean curvature surfaces in
  cones},
        date={2019-12},
     journal={Revista Matem{\'{a}}tica Iberoamericana},
      volume={36},
      number={3},
       pages={841\ndash 867},
}

\bib{PPAO2}{article}{
      author={Paolini, E},
      author={Tortorelli, VM},
       title={The quadruple planar bubble enclosing equal areas is symmetric},
        date={2020},
     journal={Calculus of Variations and Partial Differential Equations},
      volume={59},
      number={1},
       pages={1\ndash 9},
}

\bib{PPAO1}{article}{
      author={Paolini, Emanuele},
      author={Tamagnini, Andrea},
       title={Minimal clusters of four planar regions with the same area},
        date={2018},
     journal={ESAIM: Control, Optimisation and Calculus of Variations},
      volume={24},
      number={3},
       pages={1303\ndash 1331},
}

\bib{PRT}{article}{
      author={Park, Sung-ho},
       title={Every ring type spanner in a wedge is spherical},
        date={2005-05},
     journal={Mathematische Annalen},
      volume={332},
      number={3},
       pages={475\ndash 482},
}

\bib{PP}{book}{
      author={Plateau, Joseph},
       title={{Statique exp{\'e}rimentale et th{\'e}orique des liquides soumis
  aux seules forces mol{\'e}culaires}},
   publisher={Gauthier Villars},
        date={1873},
      number={v. 2},
         url={https://books.google.com.kw/books?id=fihQxdAUefMC},
}

\bib{Re}{article}{
      author={Reichardt, Ben~W.},
       title={{Proof of the Double Bubble Conjecture in $\mathbb{R}^n$}},
        date={2008},
     journal={Journal of Geometric Analysis},
      volume={18},
      number={1},
       pages={172\ndash 191},
}

\bib{SAN1}{misc}{
      author={Santilli, Mario},
       title={Uniqueness of singular convex hypersurfaces with lower bounded
  k-th mean curvature},
        date={2020},
}

\bib{SIMREG}{article}{
      author={Simon, Leon},
       title={{Regularity of capillary surfaces over domains with corners}},
        date={1980},
     journal={Pacific Journal of Mathematics},
      volume={88},
      number={2},
       pages={363\ndash 377},
         url={https://doi.org/},
}

\bib{SIM}{book}{
      author={Simon, Leon},
       title={Lectures on {G}eometric {M}easure {T}heory},
   publisher={Australian National University},
        date={2008},
}

\bib{MSp}{book}{
      author={Spivak, Michael},
       title={{A Comprehensive Introduction to Differential Geometry}},
   publisher={Publish or Perish, Inc.},
        date={1999},
      volume={3},
}

\bib{SM}{article}{
      author={Sullivan, John~M.},
      author={Morgan, Frank},
       title={{Open Problems in Soap Bubble Geometry}},
        date={1996-12},
     journal={International Journal of Mathematics},
      volume={07},
      number={06},
       pages={833\ndash 842},
}

\bib{TAM}{article}{
      author={Tam, Luen-Fai},
       title={{Regularity of capillary surfaces over domains with corners:
  borderline case.}},
        date={1986},
     journal={Pacific Journal of Mathematics},
      volume={124},
      number={2},
       pages={469 \ndash  482},
         url={https://doi.org/},
}

\bib{WWX1}{misc}{
      author={Wang, Guofang},
      author={Weng, Liangjun},
      author={Xia, Chao},
       title={{Alexandrov-Fenchel inequalities for convex hypersurfaces in the
  half-space with capillary boundary}},
   publisher={arXiv},
        date={2022},
         url={https://arxiv.org/abs/2206.04639},
}

\bib{WHI}{article}{
      author={Whitney, Hassler},
       title={Analytic extensions of differentiable functions defined in closed
  sets},
        date={1934},
     journal={Transactions of the American Mathematical Society},
      volume={36},
      number={1},
       pages={63\ndash 89},
}

\bib{WW}{thesis}{
      author={Wichiramala, Wacharin},
       title={{The Planar Triple Bubble Problem}},
        type={Ph.D. Thesis},
        date={2002},
}

\bib{WICK}{article}{
      author={Wickramasekera, Neshan},
       title={A sharp strong maximum principle and a sharp unique continuation
  theorem for singular minimal hypersurfaces},
        date={2013-12},
     journal={Calculus of Variations and Partial Differential Equations},
      volume={51},
      number={3-4},
       pages={799\ndash 812},
}

\end{biblist}
\end{bibdiv}

\end{document}